\documentclass[a4paper,11pt]{article}

\usepackage[normalem]{ulem}

\usepackage{BeaCed_article_1}

\providecommand{\keywords}[1]{\textbf{\textit{Keywords---}} #1}

\title{Isoradial immersions}
\author{%
  C\'edric Boutillier%
  \thanks{%
    {\small Sorbonne Université, CNRS,
      Laboratoire de Probabilités Statistique et Modélisation, LPSM, UMR 8001,
      F-75005 Paris, France; Institut Universitaire de France.
      \texttt{cedric.boutillier@sorbonne-universite.fr}
}},
  David Cimasoni%
  \thanks{%
    {\small Université de Genève, Section de Mathématiques, 1211 Genève 4, Suisse.
     \texttt{david.cimasoni@unige.ch}
    }}
\thanks{%
{\small corresponding author}
},
  B\'eatrice de Tili\`ere%
  \thanks{{\small %
PSL University-Dauphine, CNRS, UMR 7534, CEREMADE, 75016 Paris, France; Institut Universitaire de France.}
{\small\texttt{detiliere@ceremade.dauphine.fr}}
}
}

\date{}

\newcommand{\new}{\textcolor{black}}

\begin{document}
\maketitle

\begin{abstract}
Isoradial embeddings of planar graphs play a crucial role in the study of several models of
statistical mechanics, such as the Ising and dimer models. Kenyon and Schlenker~\cite{KeSchlenk} give a combinatorial characterization of planar graphs admitting an isoradial embedding, and describe the space of such embeddings.
In this paper we prove two results of the same type for generalizations of
isoradial embeddings: \emph{isoradial immersions} and \emph{minimal immersions}.
We show that a planar graph admits a flat isoradial immersion if and only if its
train-tracks do not form closed loops, and that a bipartite graph has a minimal
immersion if and only if it is minimal. In both cases we describe the space of
such immersions. The techniques used are different in both settings, and distinct from those
of~\cite{KeSchlenk}.
We also give an application of our results to the
dimer model defined on bipartite graphs admitting minimal immersions.
\end{abstract}


\section{Introduction}
\label{sec:intro}

An \emph{isoradial embedding} of a planar graph~$\Gs$ is an embedding of~$\Gs$ into the plane so that each face is inscribed in a circle of radius~$1$, and so that the circumcenter of any given face is in the interior of that face. Joining each vertex of~$\Gs$ with the circumcenters of the two adjacent faces, we see that such an isoradial embedding is equivalent to a so-called \emph{rhombic embedding} of the associated \emph{quad-graph}~$\GR$, \emph{i.e.}, an embedding where each face of~$\GR$ is given by a rhombus of side-length~$1$. This is illustrated in Figure~\ref{fig:iso_primal}.

\begin{figure}[ht]
  \centering
  \includegraphics[width=8cm,angle=90]{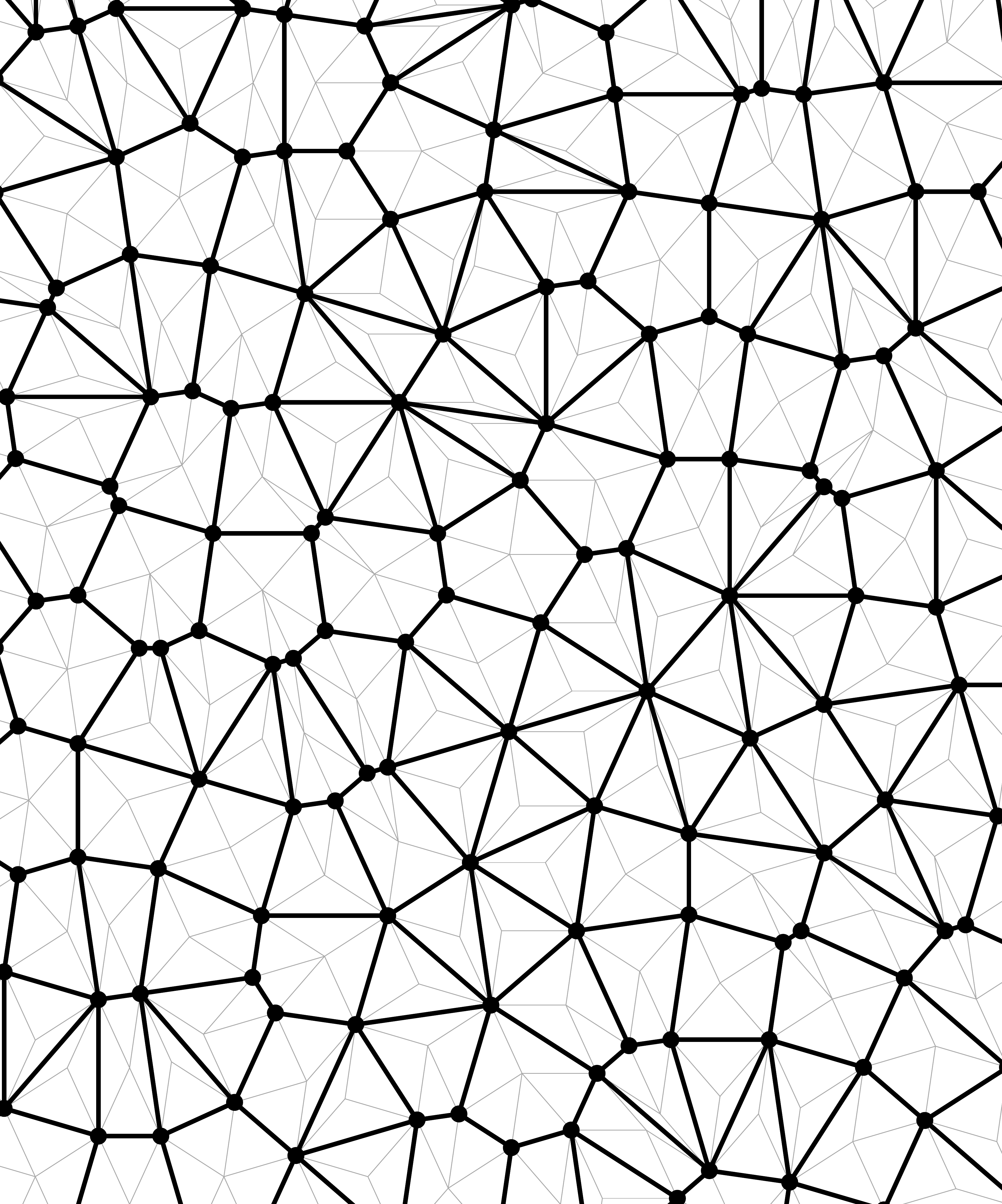}
  \caption{A piece of an isoradial embedding of an infinite planar graph $\Gs$
  (in black) and the underlying  quad-graph $\GR$ (in gray).}
  \label{fig:iso_primal}
\end{figure}

Rhombic embeddings were introduced by Duffin~\cite{Duffin}, as a natural class of graphs for which the Cauchy--Riemann equations admit a nice discretization. These ideas were rediscovered by Mercat~\cite{Mercat}, who built an extensive theory of discrete holomorphicity in one complex variable. Kenyon~\cite{Kenyon:crit} introduced the concept of isoradial embedding as defined above, and proved
it to be very useful in the study of the dimer model
and the Green function~\cite{Kenyon:crit}.
Several other models of statistical mechanics, such as the Ising model, are also shown to exhibit remarkable properties when studied on isoradial graphs with well-chosen weights, see e.g.~\cite{BdT-ZIsing1,BdT-ZIsing2,CimIsing}.
One of the deep underlying reasons lies in the star-triangle transformation and the corresponding Yang--Baxter equations, whose solutions admit an explicit parameterization when the graph of the model is isoradially embedded, see~\cite{Baxter:exactly} and references therein.
This culminates in the work of Chelkak--Smirnov who developed the theory of discrete complex analysis on isoradial graphs to a high point of sophistication~\cite{Chelkak-Smirnov1}, and used it to prove conformal invariance and universality of the Ising model~\cite{Chelkak-Smirnov2} precisely on this class of graphs. The study of statistical mechanics on isoradial graphs is still a very active area of research, as
shown by the recent results obtained on bond percolation~\cite{Grimmett-Mano,Grimmett}, dimer heights~\cite{deTil:iso,Li}, the random cluster model~\cite{Be-Du-Sm,Du-Li-Mano}, random rooted spanning forests~\cite{BdT-ZLaplace}, and the Ising model~\cite{BdT-ZIsing3}. We refer to the survey~\cite{BdTsurvey} for more details on the history of isoradial graphs and their role in statistical physics.

Coming back to the strictly combinatorial study of isoradial graphs, a natural question immediately arises: when does a fixed planar graph~$\Gs$ admit an isoradial embedding, and if so, what is the space of such embeddings of~$\Gs$? Let us briefly explain the answer to this question obtained by Kenyon and Schlenker in~\cite{KeSchlenk}. Given an arbitrary planar, embedded graph~$\Gs$, replace each edge by two (non-oriented) segments crossing each other at this edge, as described in Figure~\ref{fig:quad}. Gluing these segments together, we obtain a set of planar curves called \emph{train-tracks}, see Section~\ref{sub:gen_def} 
for a more formal definition. The main result of~\cite{KeSchlenk} is that a graph~$\Gs$ admits an isoradial embedding if and only if its train-tracks form neither closed loops nor self-intersections, and two train-tracks intersect at most once. To describe the space of isoradial embeddings of such a graph~$\Gs$, let us denote by~$Z_\Gs$ the space of maps~$\alpha$ associating to each oriented train-track~$\tr$ of~$\Gs$ a direction~$\alpha(\tr)\in\RR/2\pi\ZZ$, so that the same train-track with the opposite orientation~$\tl$ receives the opposite direction~$\alpha(\tl)=\alpha(\tr)+\pi$. An isoradial embedding of~$\Gs$ defines an element of~$Z_\Gs$ as follows: assign to an oriented train-track the direction of the parallel sides of the rhombi crossing this train-track, say from left to right. Note that any two oriented train-tracks~$\tr_1,\tr_2$ intersecting (once) define a rhombus, so the corresponding two directions must satisfy some condition for this rhombus to be embedded in the plane with the right orientation: with the orientations of Figure~\ref{fig:quad} and the conventions above, we need to have $\alpha(\tr_2)-\alpha(\tr_1)\in(0,\pi)$.
These conditions define a subspace of~$Z_\Gs$, which is shown in~\cite{KeSchlenk} to be equal to the space of isoradial embeddings of~$\Gs$.

The aim of the present paper is to prove two results of the same flavor
for the largest possible families of planar graphs.
We now 
state these two results, referring to Sections~\ref{sec:flat} and~\ref{sec:minimal} for complete and precise statements.

As mentioned above,
in the definition of an isoradial embedding, each rhombus has a \emph{rhombus angle}~$\theta=\alpha(\tr_2)-\alpha(\tr_1)$ in~$(0,\pi)$.
What if we relax this condition and 
allow
rhombus angles in~$(0,2\pi)$, thus yielding
folded rhombi? 
\new{One of the goals of this paper is to prove that, in the bipartite case, this
situation can be precisely understood. 
In the bipartite case indeed, it is possible to
orient the train-tracks in a consistent way, \emph{i.e.}, so that train-tracks turn
counterclockwise around
vertices of a given class of the bipartition and clockwise around vertices of the other class;
such a consistent orientation is uniquely defined from the bipartite structure, and canonical up to complete inversion. Then allowing the rhombus angles to be in $(0,2\pi)$ and asking that the total angle at vertices and faces is equal to $2\pi$ yields the notion of \emph{minimal immersion}, see Sections~\ref{sub:def} and~\ref{sub:min} for a formal definition.}
The main
result of Section~\ref{sec:minimal} can be loosely stated as follows, see
Theorem~\ref{thm:min} in Section~\ref{sub:min} for the precise statement.

\begin{thm}
\label{thm:1}
\new{Consider a bipartite, planar graph $\Gs$ with the consistent orientation of the train-tracks. Then, the graph $\Gs$}
admits a minimal immersion if and only if its train-tracks do not self-intersect and two  oriented train-tracks never intersect twice in the same direction. 
In any case,
the space of minimal immersions of~$\Gs$ is given by an explicit subspace~$Y_\Gs$ of~$Z_\Gs$ defined by one condition for each vertex and for each face of~$\Gs$.
\end{thm}

To be slightly more precise, if~$\Gs$ satisfies the two conditions of Theorem~\ref{thm:1}, then the space~$Y_\Gs$ is shown to contain (and in the~$\ZZ^2$-periodic case, to coincide with) the non-empty subspace~$X_\Gs$ of angle maps which are monotone with respect to a natural cyclic order on the set of oriented train-tracks of~$\Gs$, see Section~\ref{sub:min}.

Graphs satisfying these two conditions are referred to as~\emph{minimal graphs}. They were introduced by Thurston~\cite{Thurston} (the original arXiv paper appeared in 2004), play an important role in the string theory community~\cite{Gulotta,Akira_Ueda} where they are referred to as \emph{consistent dimer models}, and in the seminal paper of Goncharov--Kenyon \cite{GK}. In the specific case of $\ZZ^2$-periodic graphs, some of the concepts we use already appear in the literature: local and global ordering around vertices \cite[Lemma 4.2]{Akira_Ueda}, minimal immersions (in an understated way) in the proof of~\cite[Lemma~3.9]{GK}, the space $X_\Gs$~\cite[Theorem~3.3]{GK}.

As explained above, the study of isoradial embeddings is motivated by statistical physics, and the present work is no exception.
Indeed, in Theorem~\ref{thm:phase},
we use the result above to show that there exists~$\alpha\in Z_\Gs$ such that
some naturally occurring associated edge-weights~\cite{Kenyon:crit,Fock}
satisfy the so-called \emph{Kasteleyn condition} if and only if the graph~$\Gs$ is minimal,
see Section~\ref{sub:condition}.
This allows us to harness Kasteleyn theory~\cite{Kast61,Kuperberg} in our \new{papers~\cite{nous,BCdT:genusg}},
which deal with the dimer model on minimal graphs with the weights of~\cite{Fock}. 
\new{As a consequence, minimal graphs form the largest class of graphs where dimer models with such weights can be solved
in a satisfactory way. Otherwise said, going beyond minimal immersions would give models with negative weights whose probabilistic interpretation is problematic. Moreover, the space~$Y_\Gs$ of minimal immersions of a given minimal graph~$\Gs$ coincides
with the space of dimer models on~$\Gs$ with fixed ``abstract spectral data''
(\emph{i.e.} an~$M$-curve together with a divisor on it, see~\cite{BCdT:genusg}).
Also, the minimal immersion of~$\Gs$ corresponding to an element~$\alpha\in Y_\Gs$ gives a natural embedding of~$\Gs^*$,
as it can be shown to coincide in the rational limit with the~$t$-embedding of~$\Gs^*$ described in~\cite{KLRR}, see~\cite[Section~8.1]{nous}.
Finally, it is likely that a significant amount of the discrete complex analysis of~\cite{Chelkak-Smirnov1} extends from isoradial
to minimal graphs in a rather straightforward way, but we did not explore this question. Summarizing, it seems that minimal 
graphs rather than isoradial 
ones give the most natural framework to study such aspects of statistical mechanics models.}

Although \new{in the bipartite case,} minimal graphs form a larger family than isoradial ones, they are still quite specific amongst planar graphs. With the goal of ``immersing'' with rhombi the most general possible planar graphs, we \new{remove the assumption of being bipartite (as for isoradial embeddings) and we furthermore} relax the conditions on the rhombus angles:
in Section~\ref{sec:flat}, 
we allow for the rhombus angle~$\theta$ to be any lift in~$\RR$ of the angle~$\alpha(\tr_2)-\alpha(\tr_1)\in\RR/2\pi\ZZ$, and define the dual rhombus angle to be~$\theta^*=\pi-\theta$. 
\new{The resulting theory is trivial without additional condition, so we impose the following geometrically natural one:}
theses angles add up to~$2\pi$ around each vertex and face of~$\Gs$ when pasting these ``generalized rhombi'' together using the combinatorial information of~$\Gs$. This leads to the notion of \emph{flat isoradial immersion} of the planar graph~$\Gs$,
\new{whose formal definition can be found in Section~\ref{sub:def}.
In order to keep track of the different notions introduced, let us emphasize that a minimal immersion is a flat isoradial immersion of a bipartite graph with all rhombus angles restricted to be in~$(0,2\pi)$.}
There is a natural equivalence relation on \new{flat isoradial} immersions of a fixed graph, see Definition~\ref{def:equiv} in Section~\ref{sub:statement_main}. We now state the main theorem of Section~\ref{sec:flat}.

\begin{thm}
\label{thm:2}
A planar graph~$\Gs$ admits a flat isoradial immersion if and only if its train-tracks do not form closed loops. When this is the case,
the space of equivalence classes of flat isoradial immersions of~$\Gs$ is given by the space~$Z_\Gs$.
\end{thm}

One of the main tools in the proof of this theorem is a combinatorial result, Lemma~\ref{lem:corner} of
Section~\ref{sub:cycles} which, although slightly technical, might be of independent interest.
Proving Theorem~\ref{thm:2} amounts to solving a linear system defined on an auxiliary graph. In the specific $\ZZ^2$-periodic case, it turns out that the same linear system appears in a completely different context,
see the physics papers~\cite[Equations (2.2) and (2.3)]{Franco_2006} and~\cite[Equations (3.1) and (3.2)]{Gulotta} for $R$-charges in some supersymmetric conformal field theory.
\new{Therefore, even though the possible role of flat isoradial immersions in statistical mechanics and discrete complex analysis
remains elusive, their intriguing appearance in superconformal field theory does suggest interesting further research.}
  
We conclude this introduction by pointing out that, even though Theorem~\ref{thm:1}, Theorem~\ref{thm:2} and the main
result of~\cite{KeSchlenk} are of the same nature, they belong to different categories: isoradial embeddings
are purely geometric objects while flat isoradial immersions are almost entirely combinatorial, and minimal
immersions lie in between. As a consequence, the sets of tools used in the proofs of these three statements are
almost entirely disjoint. For example, the proof of the ``only if'' direction is immediate
in~\cite[Theorem~3.1]{KeSchlenk}, but requires a discrete Gauss--Bonnet formula in Theorem~\ref{thm:1}, and the
combinatorial study of train-track deformations in Theorem~\ref{thm:2}.
We refer the reader to Figure~\ref{fig:summary} for a visualization of the classes of graphs
and of the geometric realizations involved in these three results.

This paper is organised as follows. 
\begin{enumerate}
\item[$\bullet$] Section~\ref{sec:comb} deals with the purely combinatorial aspects of our work: the fundamental 
notions are defined in Section~\ref{sub:gen_def}, while Section~\ref{sub:cycles} contains a crucial result on 
train-tracks of an arbitrary planar graph. Section~\ref{sub:ttmin} deals with 
oriented train-tracks of
bipartite graphs, defining a natural cyclic order on them, and studying this order in the case of minimal graphs. 
\item[$\bullet$] In Section~\ref{sec:flat}, we first introduce the general notion of (flat) isoradial immersion (Section~\ref{sub:def}), then state the precise version of Theorem~\ref{thm:2} in Section~\ref{sub:statement_main}, and give the proof in Section~\ref{sub:proof_statement_main}. The last Section~\ref{sub:rem} gathers additional remarks on the definition of isoradial immersions, and on the proof of Theorem~\ref{thm:2}. In particular, we provide a geometric interpretation of isoradial immersions using rhombi that are iteratively folded along their primal or dual edge; we also prove a 
discrete Gauss--Bonnet formula.
\item[$\bullet$] 
Section~\ref{sec:minimal} deals with Theorem~\ref{thm:1} above. In Section~\ref{sub:min}, we define minimal immersions of bipartite graphs, give a geometric characterization of these immersions, define the spaces~$X_\Gs$ and~$Y_\Gs$ and state the main result. All proofs are contained in Section~\ref{sub:proofmin}, while Section~\ref{sub:condition} deals with the aforementioned implications on the dimer model.
\end{enumerate}

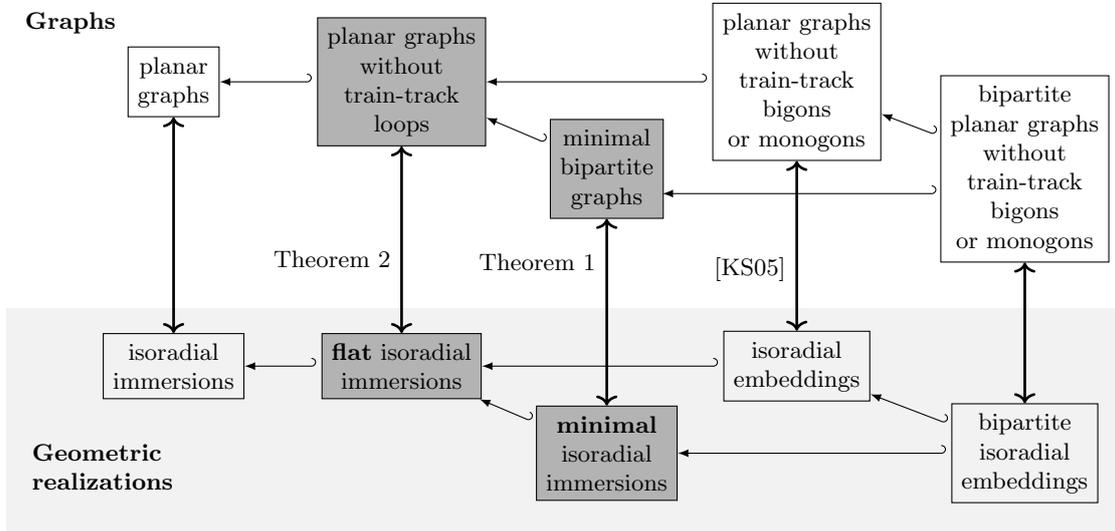
\begin{figure}[ht]
  \centering
  \begin{tikzpicture}[%
    on grid,
    node distance=30mm,
    every text node part/.style={%
    font=\footnotesize,
    align=center,
  },
    before/.style={%
      rectangle,
      draw=black,
    },
    now/.style={%
      rectangle,
      draw=black,
      fill=black!30
    }
    ]
  \node (planar) [before]  {planar \\graphs};
  \node (noloop) [now, right=of planar] {planar graphs \\ without  \\ train-track \\
    loops};
  \node (minimal) [now, right=of noloop,xshift=-3mm,yshift=-3em] {minimal\\bipartite\\graphs};
  \node (minimal2) [below=0.5em of minimal.east, xshift=-0.4em] {};
  \node (KS) [before, right=of minimal,yshift=3em,xshift=-5mm] {planar graphs \\ without \\
      train-track \\ bigons\\
    or monogons} ;
  \node (bipKS) [before,right=of KS,yshift=-3em] {bipartite \\planar graphs \\ without \\
      train-track\\  bigons\\
    or monogons} ;
  \node (bipKS2) [below=0.5em of bipKS.west,xshift=0.4em] {};
  \path [left hook-latex, shorten <=2pt]  (noloop) edge (planar);
  \path [left hook-latex, shorten <=2pt] (minimal)edge (noloop);
  \path [left hook-latex, shorten <=2pt](KS) edge (noloop);
  \path [left hook-latex, shorten <=2pt](bipKS2) edge (minimal2);
  \path [left hook-latex, shorten <=2pt](bipKS)  edge (KS);

  \fill  (-2.2,-3) [fill=black!5] rectangle (12.5,-6);

  \begin{scope}[yshift=-10cm]
  \node (imm) [before, below=of planar, yshift = -2em] {isoradial\\immersions};
  \node (flat) [now, below=of noloop , yshift = -2em] {\textbf{flat} isoradial\\immersions};
  \node (immini) [now, below=of minimal, yshift=-2em] {\textbf{minimal} \\ isoradial\\immersions};
  \node (KSemb) [before, below=of KS, yshift=-2em]  {isoradial
    \\embeddings};
  \node (bipKSemb) [before, below=of bipKS, yshift=-2em] {bipartite\\isoradial\\embeddings};
\end{scope}

  \path [left hook-latex, shorten <=2pt](flat)     edge (imm);
  \path [left hook-latex, shorten <=2pt](immini)   edge (flat);
  \path [left hook-latex, shorten <=2pt](KSemb)    edge (flat);
  \path [left hook-latex, shorten <=2pt](bipKSemb) edge (KSemb);
  \path [left hook-latex, shorten <=2pt](bipKSemb) edge (immini);

  \path (planar) edge[<->, line width=1pt] (imm);
  \path (noloop) edge[<->, line width=1pt] node (a) [below left] {Theorem~\ref{thm:2}} (flat);
  \path (minimal) edge[<->, line width=1pt] node [below left, yshift=2.4em] {Theorem~\ref{thm:1}} (immini);
  \path (KS) edge[<->, line width=1pt] node [below left] {\cite{KeSchlenk}} (KSemb);
  \path (bipKS) edge[<->, line width=1pt] (bipKSemb);
  \node [above=5mm of planar.west]  (b) {\hspace{-1.5cm}\textbf{Graphs}};
  \node [below=9mm of imm.west] (c) {\hspace{-0.15cm}\textbf{Geometric}\\ \textbf{realizations}};
  \end{tikzpicture}
  \caption{Visualization of Theorems~\ref{thm:1} and~\ref{thm:2} (in shaded
    boxes) and the main result of~\cite{KeSchlenk}.
    Inside each horizontal half, the arrows connecting the various
classes of planar graphs (resp.\ geometric realizations) indicate inclusions. The
vertical double arrows indicate an equivalence between a
graph belonging to a particular class and it admitting a particular geometric
realization.
The first vertical arrow stems from the definition of isoradial immersions
  (see Section~\ref{sub:def}),
whereas the last one is a specialization of the results
of~\cite{KeSchlenk} to the bipartite case.}
\label{fig:summary}
\end{figure}

\subsection*{Acknowledgments}

This project was started when the second-named author was visiting the first and third-named authors at the LPSM, Sorbonne Universit\'e, whose hospitality is thankfully acknowledged. 
\new{The authors would also like to thank the anonymous referees for their constructive remarks.}
The first- and third-named authors are partially supported by the \emph{DIMERS} project~ANR-18-CE40-0033 funded by the French National Research Agency.
The second-named author is partially supported by the Swiss National Science Foundation.

\section{Planar graphs and train-tracks}
\label{sec:comb}

The aim of this section is to introduce the basic combinatorial notions that are used throughout this article. We also prove four combinatorial lemmas that play a crucial role in the proof of the main results of Sections~\ref{sec:flat} and~\ref{sec:minimal}.

\subsection{General definitions}
\label{sub:gen_def}

Consider a locally finite graph~$\Gs=(\Vs,\Es)$ embedded in the plane so that
its bounded faces are topological discs, and denote by~$\Gs^*=~(\Vs^*,\Es^*)$ the dual embedded graph.
To avoid immediately diving into technical details,
we first assume that all faces of~$\Gs$ are bounded; we then turn to the
general case, where unbounded faces require extra care.

Given a graph~$\Gs$ as above, the associated \emph{quad-graph}~$\GR$
has vertex set~$\Vs\sqcup\Vs^*$ and edge set defined as follows: a primal
vertex~$v\in\Vs$ and a dual vertex~$f\in\Vs^*$ are joined by an edge each time
the vertex~$v$ lies on the boundary of the face corresponding to the dual
vertex~$f$. Note that a vertex can appear twice on the boundary of a face, which
gives rise to two edges between the corresponding vertices of~$\GR$.
Since the faces of~$\Gs$ are topological discs, the quad-graph~$\GR$ embeds in the plane with faces
consisting of (possibly degenerate) quadrilaterals. For example, a degree~$1$ vertex
or a loop in~$\Gs$ gives rise to a quadrilateral with two adjacent sides identified,
see Figure~\ref{fig:g_quad-graph_train_tracks} (left).

Following~\cite{KeSchlenk}, see also~\cite{Kenyon:crit}, we define a
\emph{train-track} as a maximal path of adjacent quadrilaterals of~$\GR$ which
does not turn: when it enters a quadrilateral, it exits through the opposite
edge.
Train-tracks are also known as rapidity lines in the field of integrable systems, see for example~\cite{Baxter:8V}.
By a slight abuse of terminology, we actually think of a train-track
as the corresponding path in the dual graph~$(\GR)^*$ crossing opposite sides of
the quadrilaterals. (These notions still make sense when a quadrilateral is
degenerate.)
Note that the graphs~$\Gs$ and~$\Gs^*$ have the same set of train-tracks, which we denote by~$\T$.
Note also that the local finiteness of~$\Gs$ implies that~$\T$ is at most denumerable.
We define the \emph{graph of train-tracks} of~$\Gs$, denoted by~$\Gs^\T$, as the
planar embedded graph~$(\GR)^*$.

By construction, its set of vertices~$\Vs^\T$ corresponds to the edges of~$\Gs$ (hence the
forthcoming notation~$e\in\Vs^\T$), while its faces correspond to
vertices and faces of~$\Gs$. By definition, the graph~$\Gs$ and its dual~$\Gs^*$ define the same graph of train-tracks~$\Gs^\T=(\Gs^*)^\T$, which is~$4$-regular: a typical edge~$v_1v_2\in\Es$ and corresponding train-tracks~$t_1,t_2\in\T$ are illustrated in Figure~\ref{fig:quad}.

We now turn to the general case, where~$\Gs$ can admit one or several
unbounded faces. In such a case, first construct the associated graph~$(\GR)^*$
as above, considering each unbounded face as a vertex of the dual graph~$\Gs^*$.
Then, define~$\Gs^\T$ as the subgraph of~$(\GR)^*$ obtained by removing all
edges included in an unbounded face of~$\Gs$; finally, define~$\T$ as the set of
maximal paths in~$\Gs^\T$ crossing opposite sides of the quadrilaterals. (Note
that the embedding of~$(\GR)^*$ into the plane might depend on the embedding
of~$\Gs^*$, but the planar embedded graph~$\Gs^\T\subset\RR^2$ is fully
determined by~$\Gs\subset\RR^2$.) By construction, its set of vertices~$\Vs^\T$
corresponds to the
edges of $\Gs$,
while its set of faces~$\Fs^\T$ corresponds to \emph{inner vertices}, not
adjacent to an unbounded face, and (bounded) \emph{inner faces} of~$\Gs$.
Note that in general, the graph~$\Gs^\T$ is no longer~$4$-regular: indeed, unbounded faces give rise to vertices of degree~$3$ and~$2$, where one train-track, or both, stop in the middle of the corresponding quadrilateral.
Note finally that in the presence of unbounded faces, the graph~$\Gs$ and its dual~$\Gs^*$ no longer define the same graph of train-tracks. However, this notion remains (locally) self-dual away from unbounded faces.
We refer to Figure~\ref{fig:g_quad-graph_train_tracks} for an illustration of a graph of train-tracks with all the possible pathologies.

\begin{figure}[ht]
  \centering
  \def\svgwidth{12cm}
\begingroup%
  \makeatletter%
  \providecommand\color[2][]{%
    \errmessage{(Inkscape) Color is used for the text in Inkscape, but the package 'color.sty' is not loaded}%
    \renewcommand\color[2][]{}%
  }%
  \providecommand\transparent[1]{%
    \errmessage{(Inkscape) Transparency is used (non-zero) for the text in Inkscape, but the package 'transparent.sty' is not loaded}%
    \renewcommand\transparent[1]{}%
  }%
  \providecommand\rotatebox[2]{#2}%
  \newcommand*\fsize{\dimexpr\f@size pt\relax}%
  \newcommand*\lineheight[1]{\fontsize{\fsize}{#1\fsize}\selectfont}%
  \ifx\svgwidth\undefined%
    \setlength{\unitlength}{1227.001143bp}%
    \ifx\svgscale\undefined%
      \relax%
    \else%
      \setlength{\unitlength}{\unitlength * \real{\svgscale}}%
    \fi%
  \else%
    \setlength{\unitlength}{\svgwidth}%
  \fi%
  \global\let\svgwidth\undefined%
  \global\let\svgscale\undefined%
  \makeatother%
  \begin{picture}(1,0.44095426)%
    \lineheight{1}%
    \setlength\tabcolsep{0pt}%
    \put(0,0){\includegraphics[width=\unitlength,page=1]{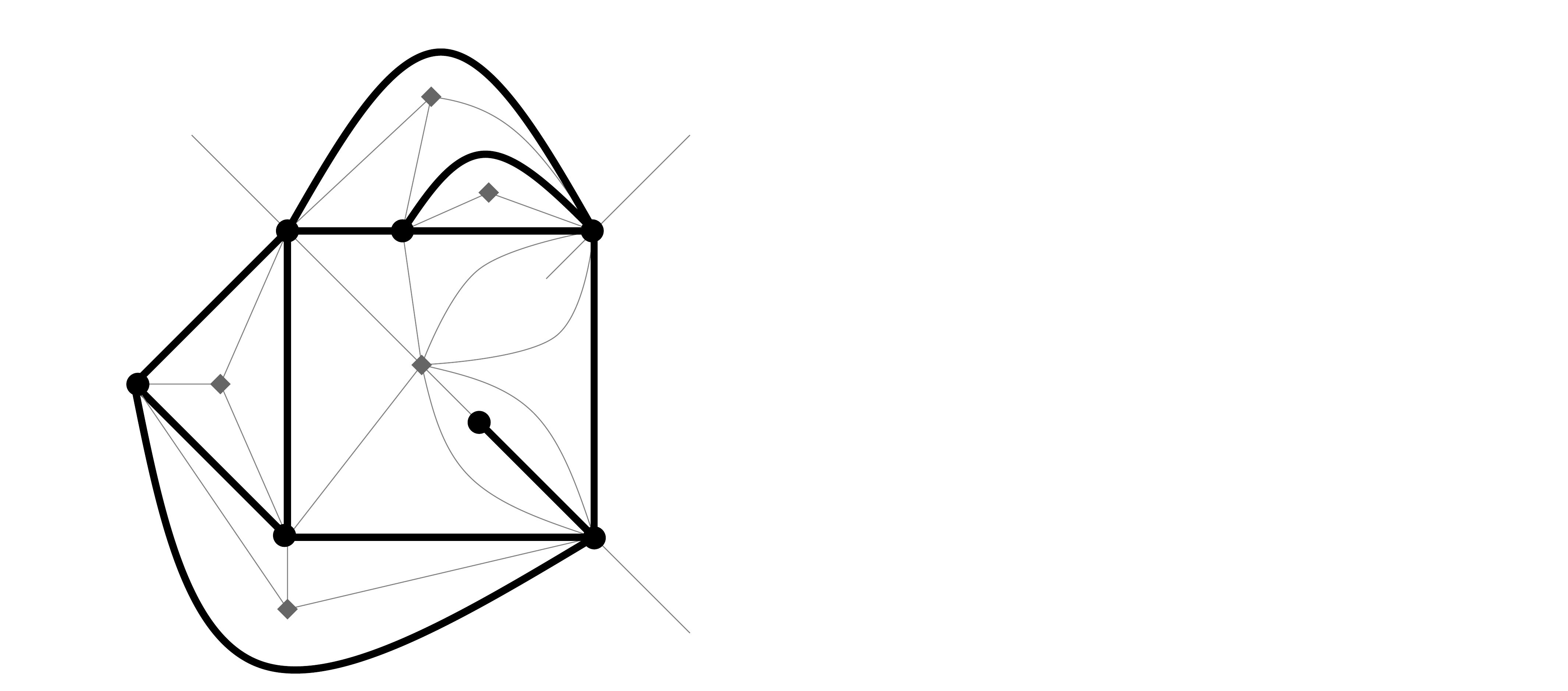}}%
    \put(0.4167757,0.11497404){\color[rgb]{0,0,0}\makebox(0,0)[lt]{\lineheight{1.25}\smash{\begin{tabular}[t]{l}$d$\end{tabular}}}}%
    \put(0,0){\includegraphics[width=\unitlength,page=2]{g_quadgraph_train_tracks.pdf}}%
    \put(0.03780296,0.07829926){\color[rgb]{0,0,0}\makebox(0,0)[lt]{\lineheight{1.25}\smash{\begin{tabular}[t]{l}$b$\end{tabular}}}}%
    \put(0.02853279,0.30446043){\color[rgb]{0,0,0}\makebox(0,0)[lt]{\lineheight{1.25}\smash{\begin{tabular}[t]{l}$a$\end{tabular}}}}%
    \put(0.43002733,0.26167319){\color[rgb]{0,0,0}\makebox(0,0)[lt]{\lineheight{1.25}\smash{\begin{tabular}[t]{l}$c$\end{tabular}}}}%
    \put(0,0){\includegraphics[width=\unitlength,page=3]{g_quadgraph_train_tracks.pdf}}%
  \end{picture}%
\endgroup%

  \caption{Left: a planar graph $\Gs$ (black). The dual $\Gs^*$ has
    vertices corresponding to bounded faces, marked with grey diamonds, and an
    additional vertex for the unbounded face of $\Gs$, which is here at infinity.
    The quad-graph $\GR$ is represented with grey edges. Those with a dashed end are
    connected to the dual vertex at infinity.
    Train-tracks of $\Gs$, paths of adjacent quadrilaterals, are materialized by
    colored lines ($a$: blue, $b$: green,
$c$: orange, $d$: cyan). Right: the corresponding graph of train-tracks $\Gs^\T$.}
\label{fig:g_quad-graph_train_tracks}
\end{figure}

We now come to orientation of train-tracks. Each train-track~$t\in\T$ can be oriented in two directions: we let~$\tl,\tr$ denote its oriented versions, and $\Torient$ be the set of such oriented train-tracks.
\new{Given an inner oriented edge $(v_1,v_2)$ of $\Gs$ and an oriented train-track~$\tr$ crossing it,
we say that these orientations  are \emph{coherent} if~$(v_1,v_2)$ and~$\tr$ define the positive (counterclockwise) orientation of the plane, see Figure~\ref{fig:quad}.
We morever fix the subscripts $1,2$ so that $\tr_1,\tr_2$ also defines the positive orientation of the plane.}

\begin{figure}[ht]
\centering
\def\svgwidth{4cm}
\begingroup%
  \makeatletter%
  \providecommand\color[2][]{%
    \errmessage{(Inkscape) Color is used for the text in Inkscape, but the package 'color.sty' is not loaded}%
    \renewcommand\color[2][]{}%
  }%
  \providecommand\transparent[1]{%
    \errmessage{(Inkscape) Transparency is used (non-zero) for the text in Inkscape, but the package 'transparent.sty' is not loaded}%
    \renewcommand\transparent[1]{}%
  }%
  \providecommand\rotatebox[2]{#2}%
  \newcommand*\fsize{\dimexpr\f@size pt\relax}%
  \newcommand*\lineheight[1]{\fontsize{\fsize}{#1\fsize}\selectfont}%
  \ifx\svgwidth\undefined%
    \setlength{\unitlength}{103.77946005bp}%
    \ifx\svgscale\undefined%
      \relax%
    \else%
      \setlength{\unitlength}{\unitlength * \real{\svgscale}}%
    \fi%
  \else%
    \setlength{\unitlength}{\svgwidth}%
  \fi%
  \global\let\svgwidth\undefined%
  \global\let\svgscale\undefined%
  \makeatother%
  \begin{picture}(1,0.78046644)%
    \lineheight{1}%
    \setlength\tabcolsep{0pt}%
    \put(0,0){\includegraphics[width=\unitlength,page=1]{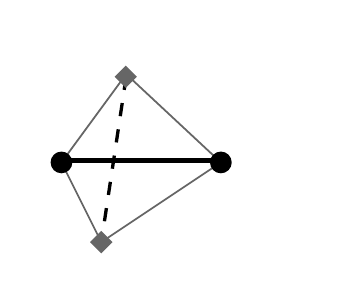}}%
    \put(0.03756421,0.31864368){\color[rgb]{0,0,0}\makebox(0,0)[lt]{\lineheight{1.25}\smash{\begin{tabular}[t]{l}$v_1$\end{tabular}}}}%
    \put(0.67434347,0.31877835){\color[rgb]{0,0,0}\makebox(0,0)[lt]{\lineheight{1.25}\smash{\begin{tabular}[t]{l}$v_2$\end{tabular}}}}%
    \put(0,0){\includegraphics[width=\unitlength,page=2]{fig_quadri.pdf}}%
    \put(0.37644358,0.58861182){\color[rgb]{0,0,0}\makebox(0,0)[lt]{\lineheight{1.25}\smash{\begin{tabular}[t]{l}$f_2$\end{tabular}}}}%
    \put(0,0){\includegraphics[width=\unitlength,page=3]{fig_quadri.pdf}}%
    \put(0.29593672,0.01545441){\color[rgb]{0,0,0}\makebox(0,0)[lt]{\lineheight{1.25}\smash{\begin{tabular}[t]{l}$f_1$\end{tabular}}}}%
    \put(0.79433151,0.67222096){\color[rgb]{0,0,0}\makebox(0,0)[lt]{\lineheight{1.25}\smash{\begin{tabular}[t]{l}$\tr_1$\end{tabular}}}}%
    \put(-0.00492099,0.72358474){\color[rgb]{0,0,0}\makebox(0,0)[lt]{\lineheight{1.25}\smash{\begin{tabular}[t]{l}$\tr_2$\end{tabular}}}}%
  \end{picture}%
\endgroup%

\caption{An edge \new{$v_1v_2$} of~$\Gs$ and the two corresponding train-tracks \new{oriented so that $(v_1,v_2)$ and $\tr_1$, resp. $\tr_2$, are coherent, and labeled so that $\tr_1,\tr_2$ defines a positive orientation of the plane.}}
\label{fig:quad}
\end{figure}

\new{In the case of bipartite graphs, there exists a natural orientation of train-tracks, defined as follows.}
Recall that a graph~$\Gs$ is \emph{bipartite} if its vertex set can be partitioned into two disjoint sets~$\Vs=\Bs\sqcup\Ws$ of \emph{black} and \emph{white} vertices, such that each edge of~$\Gs$ joins a black vertex $b$ to a white one $w$. Then, a  \new{\emph{consistent}} orientation of the train-tracks is obtained by
requiring train-tracks to turn counterclockwise around white vertices and clockwise around black ones.
\new{By construction, this is the unique orientation which is coherent at every edge with the
orientation of the edge from the white to the black vertex.}

Without any hypothesis on~$\Gs$, the associated train-tracks can display all
sorts of behavior: for example, a train-track might form a closed loop, it might
intersect itself, and two train-tracks might intersect more than once. As mentioned in the introduction, the main
result of~\cite{KeSchlenk} is that a graph~$\Gs$ admits an isoradial embedding if and only if its train-tracks never behave in this
way. The main theorem of Section~\ref{sec:flat} should be seen as a result of
the same type: a graph~$\Gs$ admits a more general geometric realization, called
\emph{isoradial immersion}, if and only if its train-tracks do not form closed
loops. Finally, the main theorem of Section~\ref{sec:minimal} is of the same
flavor: a bipartite graph~$\Gs$ \new{with consistently oriented train-tracks} admits a so-called \emph{minimal immersion} if
and only if its \new{oriented} train-tracks do not form closed loops, self-intersections, or
parallel bigons, \new{where a \emph{parallel bigon} is a pair of oriented train-tracks that intersect twice at two points $x$ and $y$ and both are oriented from $x$ to $y$~\cite{Thurston}. The three forbidden configurations are} illustrated in Figure~\ref{fig:forbidden_tt}.

\begin{figure}[ht]
  \centering
  \def\svgwidth{8cm}
\begingroup%
  \makeatletter%
  \providecommand\color[2][]{%
    \errmessage{(Inkscape) Color is used for the text in Inkscape, but the package 'color.sty' is not loaded}%
    \renewcommand\color[2][]{}%
  }%
  \providecommand\transparent[1]{%
    \errmessage{(Inkscape) Transparency is used (non-zero) for the text in Inkscape, but the package 'transparent.sty' is not loaded}%
    \renewcommand\transparent[1]{}%
  }%
  \providecommand\rotatebox[2]{#2}%
  \newcommand*\fsize{\dimexpr\f@size pt\relax}%
  \newcommand*\lineheight[1]{\fontsize{\fsize}{#1\fsize}\selectfont}%
  \ifx\svgwidth\undefined%
    \setlength{\unitlength}{602.75340958bp}%
    \ifx\svgscale\undefined%
      \relax%
    \else%
      \setlength{\unitlength}{\unitlength * \real{\svgscale}}%
    \fi%
  \else%
    \setlength{\unitlength}{\svgwidth}%
  \fi%
  \global\let\svgwidth\undefined%
  \global\let\svgscale\undefined%
  \makeatother%
  \begin{picture}(1,0.40188218)%
    \lineheight{1}%
    \setlength\tabcolsep{0pt}%
    \put(0,0){\includegraphics[width=\unitlength,page=1]{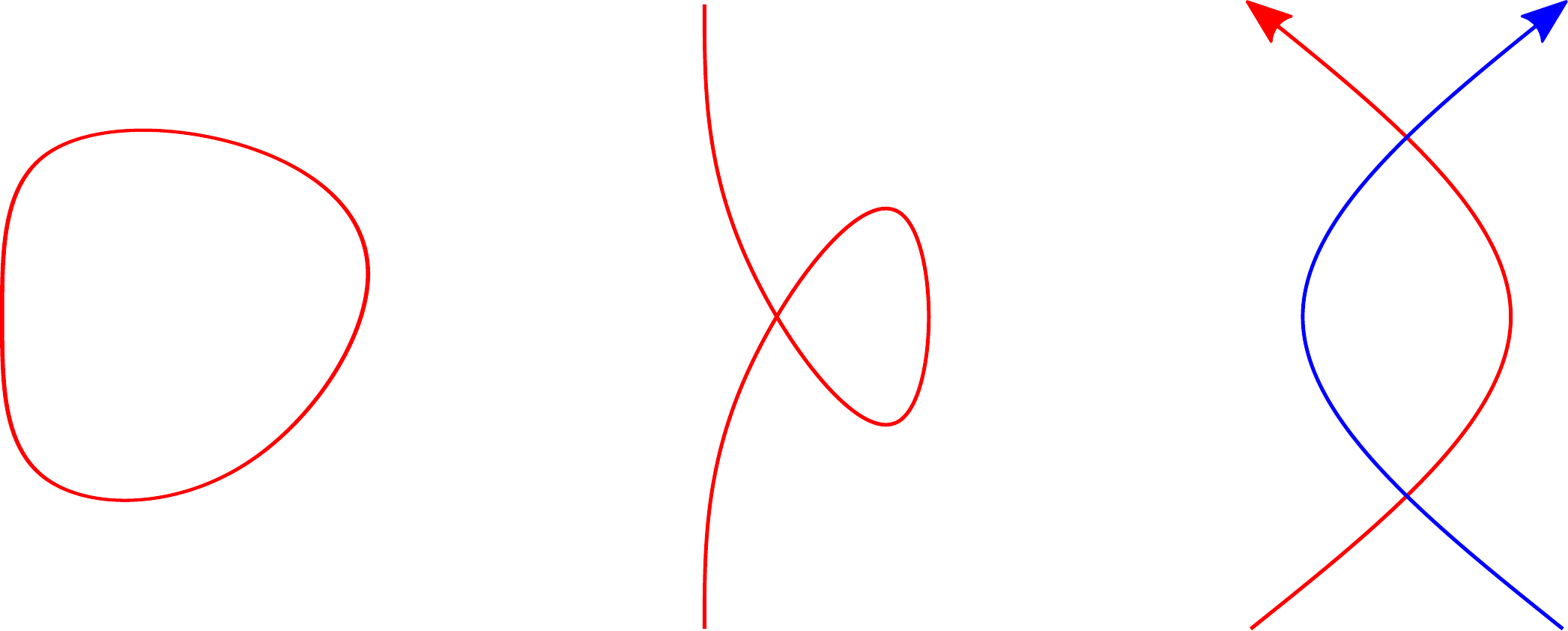}}%
    \put(0.88230427,0.03889341){\color[rgb]{0,0,0}\makebox(0,0)[lt]{\lineheight{1.25}\smash{\begin{tabular}[t]{l}$x$\end{tabular}}}}%
    \put(0.88728172,0.34001157){\color[rgb]{0,0,0}\makebox(0,0)[lt]{\lineheight{1.25}\smash{\begin{tabular}[t]{l}$y$\end{tabular}}}}%
  \end{picture}%
\endgroup%

  \caption{\new{Train-tracks form} a closed loop (left), a self-intersection
  (middle), or a parallel bigon (right),
\new{if~$\Gs^\T$ contains the corresponding graph as a minor.
}}
  \label{fig:forbidden_tt}
\end{figure}

\subsection{Cycles in the graph of train-tracks}
\label{sub:cycles}

The basic definitions being given, we now turn to the first result of this article, Lemma~\ref{lem:corner}, a fairly technical combinatorial statement. This lemma is used in Section~\ref{sec:flat} to describe the space of
solutions of a linear system indexed by vertices of $\GR$. But we believe that
under its general form here, it is interesting in itself and may have applications
to other contexts.

Following the standard terminology, we say that a finite subgraph of an arbitrary
abstract graph~$\mathsf{\Gamma}$ is a \emph{cycle}, resp. a \emph{simple cycle}, if all of
its vertices are of even degree, resp. of degree~$2$. We denote by
$\mathcal{E}(\mathsf{\Gamma})$ the set of cycles of $\mathsf{\Gamma}$ ($\mathcal{E}$ stands for
\emph{even}).
Recall that the symmetric
difference endows $\mathcal{E}(\mathsf{\Gamma})$ with the structure of
an~$\mathbb{F}_2$-vector space.

An intermediate step for Lemma~\ref{lem:corner} consists in exhibiting
a bijection
between the cycles of $\Gs^\T$ and those of a natural auxiliary graph. Before
establishing this bijection in Lemma~\ref{lem:phiiso}, we discuss some features of 
cycles of $\Gs^\T$. 

Vertices of a cycle~$c\in\mathcal{E}(\Gs^\T)$
fall into three categories:
degree~$4$ vertices
where by definition,~$c$ locally follows two (non-necessarily distinct)
intersecting train-tracks; degree~$2$ vertices where~$c$ stays on the same
train-track, \emph{i.e.}, where it crosses opposite sides of the corresponding
quadrilateral; and degree~$2$ vertices where~$c$ switches train-tracks,\emph{
i.e.}, where it ``turns'' inside the quadrilateral. We call these latter
vertices \emph{corners} of~$c$. 

From the above observations we define a procedure for drawing $c$ on the graph, giving a 
unique decomposition of $c$ into \emph{closed curves}:
start from an edge of $c$,
follow the cycle in one of the two possible directions by going straight at
each intersection.
When coming back to the starting point, we have covered a cycle component of $c$ in our
decomposition, referred to as a \emph{closed curve}.
If the whole of $c$ has been covered,
stop; else remove this component from $c$, and iterate the procedure from an
uncovered edge of $c$. Since the
cycle is finite, the algorithm of course
terminates.
Note also that specifying
the type of behavior at the degree 4 vertices yields
uniqueness of the decomposition of $c$ into closed curves,
and that the behaviour chosen minimizes the number of corners:
the total number of corners in all closed curves appearing in this decomposition of
$c$ is the same as the number of corners of $c$ itself,
see Figure~\ref{fig:cycles_decomp} for an example.

\begin{figure}[ht]
  \centering
  \def\svgwidth{10cm}
\begingroup%
  \makeatletter%
  \providecommand\color[2][]{%
    \errmessage{(Inkscape) Color is used for the text in Inkscape, but the package 'color.sty' is not loaded}%
    \renewcommand\color[2][]{}%
  }%
  \providecommand\transparent[1]{%
    \errmessage{(Inkscape) Transparency is used (non-zero) for the text in Inkscape, but the package 'transparent.sty' is not loaded}%
    \renewcommand\transparent[1]{}%
  }%
  \providecommand\rotatebox[2]{#2}%
  \newcommand*\fsize{\dimexpr\f@size pt\relax}%
  \newcommand*\lineheight[1]{\fontsize{\fsize}{#1\fsize}\selectfont}%
  \ifx\svgwidth\undefined%
    \setlength{\unitlength}{317.62767504bp}%
    \ifx\svgscale\undefined%
      \relax%
    \else%
      \setlength{\unitlength}{\unitlength * \real{\svgscale}}%
    \fi%
  \else%
    \setlength{\unitlength}{\svgwidth}%
  \fi%
  \global\let\svgwidth\undefined%
  \global\let\svgscale\undefined%
  \makeatother%
  \begin{picture}(1,0.28335063)%
    \lineheight{1}%
    \setlength\tabcolsep{0pt}%
    \put(0,0){\includegraphics[width=\unitlength,page=1]{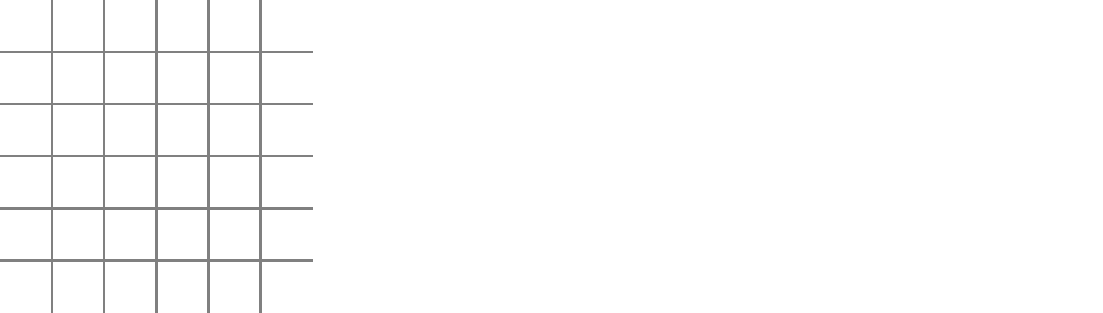}}%
    \put(0.36363332,0.10962972){\color[rgb]{0,0,0}\makebox(0,0)[lt]{\lineheight{1.25}\smash{\begin{tabular}[t]{l}$=$\end{tabular}}}}%
    \put(0.73671156,0.11806276){\color[rgb]{0,0,0}\makebox(0,0)[lt]{\lineheight{1.25}\smash{\begin{tabular}[t]{l}$+$\end{tabular}}}}%
    \put(0,0){\includegraphics[width=\unitlength,page=2]{cycle_decomposition.pdf}}%
  \end{picture}%
\endgroup%

  \caption{Example of a cycle (left) with corners partitioned into two parts (with
    different shades of purple), corresponding to the corners of the closed
  curves (right) entering into its canonical decomposition into closed curves.}
  \label{fig:cycles_decomp}
\end{figure}

When none of the train-tracks of $\Gs$ form a closed loop,
any closed curve has at least one corner.
In this case, we relate cycles on $\Gs^\T$ with cycles on another, auxiliary abstract
graph $\Gs'=(\Vs',\Es')$, constructed from $\Gs^\T$ by setting $\Vs'=\T$ and
$\Es'=\Vs^\T$. 
More precisely, vertices of~$\Gs'$ are given by the train-tracks, and each
intersection in~$\Gs^\T$ between two train-tracks (resp. self-intersection of
one train-track) defines an edge between the two corresponding vertices
of~$\Gs'$ (resp. a loop at the corresponding vertex of~$\Gs'$),
see Figure~\ref{fig:auxiliary_graph}.

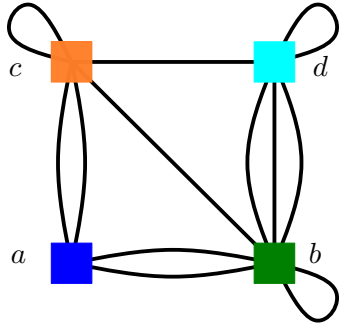
\begin{figure}[ht]
\centering
\def\svgwidth{5cm}
\begingroup%
  \makeatletter%
  \providecommand\color[2][]{%
    \errmessage{(Inkscape) Color is used for the text in Inkscape, but the package 'color.sty' is not loaded}%
    \renewcommand\color[2][]{}%
  }%
  \providecommand\transparent[1]{%
    \errmessage{(Inkscape) Transparency is used (non-zero) for the text in Inkscape, but the package 'transparent.sty' is not loaded}%
    \renewcommand\transparent[1]{}%
  }%
  \providecommand\rotatebox[2]{#2}%
  \newcommand*\fsize{\dimexpr\f@size pt\relax}%
  \newcommand*\lineheight[1]{\fontsize{\fsize}{#1\fsize}\selectfont}%
  \ifx\svgwidth\undefined%
    \setlength{\unitlength}{280.99889592bp}%
    \ifx\svgscale\undefined%
      \relax%
    \else%
      \setlength{\unitlength}{\unitlength * \real{\svgscale}}%
    \fi%
  \else%
    \setlength{\unitlength}{\svgwidth}%
  \fi%
  \global\let\svgwidth\undefined%
  \global\let\svgscale\undefined%
  \makeatother%
  \begin{picture}(1,0.86872273)%
    \lineheight{1}%
    \setlength\tabcolsep{0pt}%
    \put(0,0){\includegraphics[width=\unitlength,page=1]{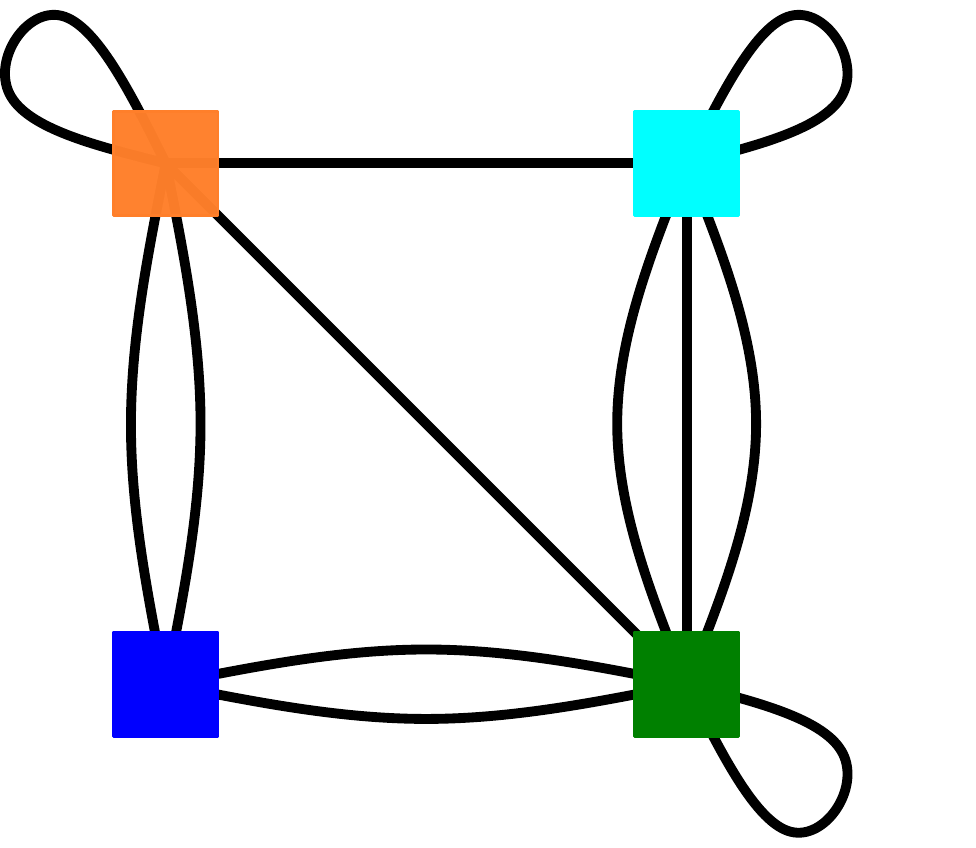}}%
    \put(0.80519066,0.66390001){\color[rgb]{0,0,0}\makebox(0,0)[lt]{\lineheight{1.25}\smash{\begin{tabular}[t]{l}$d$\end{tabular}}}}%
    \put(0.79451414,0.16745645){\color[rgb]{0,0,0}\makebox(0,0)[lt]{\lineheight{1.25}\smash{\begin{tabular}[t]{l}$b$\end{tabular}}}}%
    \put(0.00981381,0.16745645){\color[rgb]{0,0,0}\makebox(0,0)[lt]{\lineheight{1.25}\smash{\begin{tabular}[t]{l}$a$\end{tabular}}}}%
    \put(0.00447569,0.66389978){\color[rgb]{0,0,0}\makebox(0,0)[lt]{\lineheight{1.25}\smash{\begin{tabular}[t]{l}$c$\end{tabular}}}}%
  \end{picture}%
\endgroup%

\caption{The auxiliary graph $\Gs'$ appearing in the proof of
  Lemma~\ref{lem:corner}, for the example of graph $\Gs$ given in
Figure~\ref{fig:g_quad-graph_train_tracks}.}
\label{fig:auxiliary_graph}
\end{figure}

Observe that any closed curve~$c\subset\Gs^\T$ defines a
cycle~$\varphi(c)\subset\Gs'$ with the corners of~$c$ corresponding to edges
of~$\varphi(c)$. We use the unique decomposition of cycles into closed curves described above to extend $\varphi$ on $\mathcal{E}(\Gs^\T)$ by declaring that $\varphi$
sends a cycle to the sum of the images of its components in that particular
decomposition.

\begin{lem}
  \label{lem:phiiso}
  Let $\Gs^\T$ be a connected graph of train-tracks, none of which forms a closed loop.
  Then, the function $\varphi\colon\mathcal{E}(\Gs^\T)\rightarrow\mathcal{E}(\Gs')$
  is an isomorphism of $\mathbb{F}_2$-vector spaces.
\end{lem}

\begin{proof}
  First, $\varphi$ is indeed a linear map. This can be checked by noticing that when
  decomposing a cycle $c$ as a sum of smaller cycles $c=\sum_i c_i$, the
  cardinality of the set of indices $i$ such that $e\in\Vs^\T$ is a corner of $c_i$
  is odd if $e$ is a corner of $c$, and even if not. Because we look
  at linear combinations with coefficients mod 2, this is enough to conclude
  that $\varphi(c)=\sum_i\varphi(c_i)$.

  The application $\varphi$ is clearly onto:
  all simple cycles of~$\Gs'$ can be realized, and they generate the full space
  of cycles~$\mathcal{E}(\Gs')$. We now wish to show that~$\varphi$ is injective.

Let us first assume that~$\Gs^\T$ is finite. In such a case, the auxiliary graph~$\Gs'$ is finite as well. Since we assume~$\Gs^\T$ to be connected, so is~$\Gs'$, and we can determine the dimension of~$\mathcal{E}(\Gs')$ using the standard computation
\[
1-\dim(\mathcal{E}(\Gs'))=\chi(\Gs')=|\Vs'|-|\Es'|=|\T|-|\Vs^\T|\,,
\]
where~$\chi$ stands for the Euler characteristic. 

To evaluate the dimension of~$\mathcal{E}(\Gs^\T)$,
let us denote
by~$\Vs^\T=\Vs_2\sqcup\Vs_3\sqcup\Vs_4$ the partition of the vertices of~$\Gs^\T$ into vertices of degree~$2$,~$3$ and~$4$. We have the obvious equality~$2|\Vs_2|+3|\Vs_3|+4|\Vs_4|=2|\Es^\T|$, and the slightly less obvious one~$2|\Vs_2|+|\Vs_3|=2|\T|$ which uses the fact that~$\Gs^\T$ has no closed loops. Together, they lead to the equation~$2|\Vs^\T|=|\Es^\T|+|\T|$, giving
\begin{equation}\label{equ:euler}
1-\dim(\mathcal{E}(\Gs^\T))=\chi(\Gs^\T)=|\Vs^\T|-|\Es^\T|=|\T|-|\Vs^\T|\,.
\end{equation}

By the two equations displayed above, we see that~$\varphi\colon\mathcal{E}(\Gs^\T)\to\mathcal{E}(\Gs')$ is a surjective linear map between finite dimensional vector spaces of the same dimension, and therefore an isomorphism.

The general (possibly infinite) case can be seen as a consequence of the finite case, as follows. Since~$\varphi=\varphi_\Gs\colon\mathcal{E}(\Gs^\T)\to\mathcal{E}(\Gs')$ is already known to be a surjective linear map, we are left with proving that the only~$c\in\mathcal{E}(\Gs^\T)$ such that~$\varphi(c)$ vanishes is~$c=0$, the empty cycle. Such a~$c$ is contained in a finite connected subgraph~$\Hs^\T$ of~$\Gs^\T$; let~$\Hs'$ be the corresponding auxiliary graph, and~$\varphi_\Hs\colon\mathcal{E}(\Hs^\T)\to\mathcal{E}(\Hs')$ be the corresponding linear map, which we know is an isomorphism since~$\Hs^\T$ is finite, connected, and contains no closed loop. By naturality of~$\varphi$, the following diagram commutes
\[
\begin{tikzcd}
\mathcal{E}(\Gs^\T)\arrow[r, "\varphi_\Gs"]& \mathcal{E}(\Gs')\\
\mathcal{E}(\Hs^\T)\arrow[u,"j"]\arrow[r, "\varphi_\Hs"]& \mathcal{E}(\Hs')\arrow[u,"j'"]\,,
\end{tikzcd}
\]
where the vertical arrows are the inclusions of spaces of cycles induced by the inclusion of subgraphs. As mentioned, there exists~$c_\Hs\in\mathcal{E}(\Hs^\T)$ such that~$c=j(c_\Hs)$, leading to
\[
0=\varphi_\Gs(c)=\varphi_\Gs(j(c_\Hs))=j'(\varphi_\Hs(c_\Hs))\,.
\]
Since~$j'\circ\varphi_\Hs$ is injective, it follows that~$c_\Hs$ is equal to~$0$, and so is~$c=j(c_\Hs)$.
\end{proof}

Using the isomorphism $\varphi$ to transport a natural basis of
$\mathcal{E}(\Gs')$ to $\mathcal{E}(\Gs^\T)$ yields the following combinatorial
statement about $\Gs^\T$.

\begin{lem}\label{lem:corner}
Let~$\Gs^\T$ be a connected graph of train-tracks, none of which forms a closed loop. Then, there exists a set~$M\subset\Vs^\T$, an injective map~$\tau\colon M\to\T$, and a basis~\new{$(c_e)_{e\in\Vs^\T\setminus M}$} of~$\mathcal{E}(\Gs^\T)$ consisting of closed curves indexed by the vertices of~$\Vs^\T\setminus M$, such that:
\begin{enumerate}
\item{each vertex~$e\in M$ belongs to the associated train-track~$\tau(e)\in\T$, and the image of~$\tau$ consists of all train-tracks but one;}
\item{\new{for every $e\in\Vs^\T\setminus M$, the vertex $e$ is a corner of~$c_e$, and all other corners of~$c_e$} belong to~$M$.
}
\end{enumerate}
\end{lem}

Recall that we use the notation $e$ for vertices of $\Gs^\T$ because they correspond to edges of $\Gs$. This statement is quite technical, but its translation in plain English less so, at least in the finite case: there exists a set\new{~$M$} of~$|\T|-1$ \new{``marked vertices''
($M$ stands for ``marked'')} such that each vertex $e$ in this subset has a distinct train-track $\tau(e)$ associated by the map $\tau$; and a basis~\new{$(c_e)_{e\in\Vs^\T\setminus M}$} of cycles of~$\Gs^\T$ consisting of closed curves, such that one can mark all vertices of~$\Gs^\T$ starting from~\new{$M$} and then marking the unique missing corner of each closed curve~$c_e$.
Note that this statement is easily seen not to hold when some train-track makes a loop.

\begin{proof}[Proof of Lemma~\ref{lem:corner}]
Since~$\Gs^\T$ is connected, so is~$\Gs'$ which therefore contains a spanning tree~$\Ts'$, rooted at an arbitrary vertex~$t_0\in\Vs'=\T$.
Let~$M\subset \Vs^\T=\Es'$ denote the set of vertices of~$\Gs^\T$ corresponding to the edges of~$\Ts'$, and~$\tau\colon M\to\T$ be the map defined by associating to each edge of~$\Ts'$ its adjacent vertex in the direction opposite to the root~$t_0$. Since~$\Ts'$ has no cycle, the map~$\tau$ is injective, and since~$\Ts'$ spans all vertices, the image of~$\tau$ is~$\T\setminus\{t_0\}$. This shows the first point.

To conclude the proof, \new{write~$\varepsilon_e$ for the element of~$\Es'\setminus\Ts'$  corresponding to~$e\in\Vs^\T\setminus M$, let~$\varepsilon'_e\in\mathcal{E}(\Gs')$ be the unique (simple) cycle in~$\Ts\cup\{\varepsilon_e\}$, and set~$c_e=\varphi^{-1}(\varepsilon'_e)\in\mathcal{E}(\Gs^\T)$}. The crucial fact is that $c_e$ is a closed curve. Indeed, $\varepsilon'_e$ is a simple cycle.
Every pair of edges of $\varepsilon'_e$ connected to the same vertex $t$ correspond
to subsequent corners of $c_e$ connected by the train-track $t$. Therefore, when performing the decomposition procedure into closed curves on the cycle $c_e$, we obtain a unique component, implying that $c_e$ is indeed a closed curve. 

Since~\new{$(\varepsilon'_e)_{e\in\Vs^\T\setminus M}$} is a basis of~$\mathcal{E}(\Gs')$
and~$\varphi$ an isomorphism,~\new{$(c_e)_{e\in\Vs^\T\setminus M}$} is a basis of~$\mathcal{E}(\Gs^\T)$. Via the isomorphism~$\varphi$, the statement of the second point translates into the following tautology: for all~\new{$e\in\Vs^\T\setminus M$}, the edge~$\varepsilon_e$ belongs to~$\varepsilon'_e$ and all other edges of~$\varepsilon'_e$ belong to~$\Ts'$.
\end{proof}

\subsection{Train-tracks in minimal bipartite graphs}
\label{sub:ttmin}

In the whole of this section, we suppose that the planar embedded graph~$\Gs$ is
bipartite, and has no unbounded faces. Recall from Section~\ref{sub:gen_def} that the 
train-tracks of~$\Gs$ can be consistently oriented so that black vertices are on the right, and white vertices on the left of the directed paths. Recall also that the above orientation is 
\new{coherent at every edge} with the orientation of train-tracks arising from oriented edges, directed from white vertices to black ones. 
We denote by~$\Tbip$ the corresponding set of oriented train-tracks,
which is nothing but a copy of~$\T$, 
with all train-tracks consistently oriented.

In most of the present section and Section~\ref{sec:minimal}, we focus on a special class of planar bipartite graphs known as \emph{minimal graphs}, introduced by Thurston~\cite{Thurston}. These graphs are central in the seminal paper
~\cite{GK} by Goncharov--Kenyon;
they also occur in string theory
where they are referred to as \emph{consistent dimer models}, see~\cite{Gulotta,Akira_Ueda}.

\begin{defi}
\label{defi:min}
A planar, embedded, bipartite graph~$\Gs$ is said to be \emph{minimal} if~$\Gs^\T$
contains neither self-intersecting train-tracks, nor parallel bigons, see
Figure~\ref{fig:forbidden_tt}.
\end{defi}

Throughout the rest of this section, we assume that the train-tracks of~$\Gs$ do not form closed loops,
a condition which always holds for minimal graphs. Indeed, with our conventions,
no connected component of~$\Gs^\T$ consists of a simple closed curve; therefore, a closed loop will
either self-intersect, or meet another train-track and thus form a parallel bigon.
Our goal is to prove two combinatorial results on the order of train-tracks of minimal bipartite graphs.

When the graph $\Gs$ is $\mathbb{Z}^2$-periodic,
a partial cyclic order on train-tracks is defined as follows:
the projection of train-tracks
on the quotient graph $\Gs/\mathbb{Z}^2$ drawn on a torus are non-trivial
oriented loops whose homology class in $H_1(\TT;\ZZ)\simeq \ZZ^2$ is given by
a pair of coprimes integers. The total cyclic order on such pairs thus induces
a natural (partial) cyclic order on train-tracks, which has
been exploited in the papers~\cite{Gulotta,Akira_Ueda,GK}.

We now extend this partial cyclic order to train-tracks of arbitrary
bipartite planar graphs with neither train-track loops nor unbounded faces.
Note that these conditions are equivalent to requiring that all train-tracks
are bi-infinite (oriented) curves.

First, we say that two oriented train-tracks (or more generally, two non-closed oriented planar curves) are \emph{parallel} if they satisfy one of the following conditions:
\begin{itemize}
\item{they intersect infinitely many times in the same direction, see
  Figure~\ref{fig:parallel_tt}, left,}
\item{they are disjoint, and there exists a topological disc~$D\subset\RR^2$
  that they cross in the same direction, see Figure~\ref{fig:parallel_tt}, right.}
\end{itemize}
Two oriented train-tracks~$\tr_1,\tr_2$ are called \emph{anti-parallel} if~$\tr_1$ and~$\tl_2$ are parallel.

\begin{figure}[ht]
  \centering
  \def\svgwidth{8cm}
\begingroup%
  \makeatletter%
  \providecommand\color[2][]{%
    \errmessage{(Inkscape) Color is used for the text in Inkscape, but the package 'color.sty' is not loaded}%
    \renewcommand\color[2][]{}%
  }%
  \providecommand\transparent[1]{%
    \errmessage{(Inkscape) Transparency is used (non-zero) for the text in Inkscape, but the package 'transparent.sty' is not loaded}%
    \renewcommand\transparent[1]{}%
  }%
  \providecommand\rotatebox[2]{#2}%
  \newcommand*\fsize{\dimexpr\f@size pt\relax}%
  \newcommand*\lineheight[1]{\fontsize{\fsize}{#1\fsize}\selectfont}%
  \ifx\svgwidth\undefined%
    \setlength{\unitlength}{556.44907659bp}%
    \ifx\svgscale\undefined%
      \relax%
    \else%
      \setlength{\unitlength}{\unitlength * \real{\svgscale}}%
    \fi%
  \else%
    \setlength{\unitlength}{\svgwidth}%
  \fi%
  \global\let\svgwidth\undefined%
  \global\let\svgscale\undefined%
  \makeatother%
  \begin{picture}(1,0.54634502)%
    \lineheight{1}%
    \setlength\tabcolsep{0pt}%
    \put(0,0){\includegraphics[width=\unitlength,page=1]{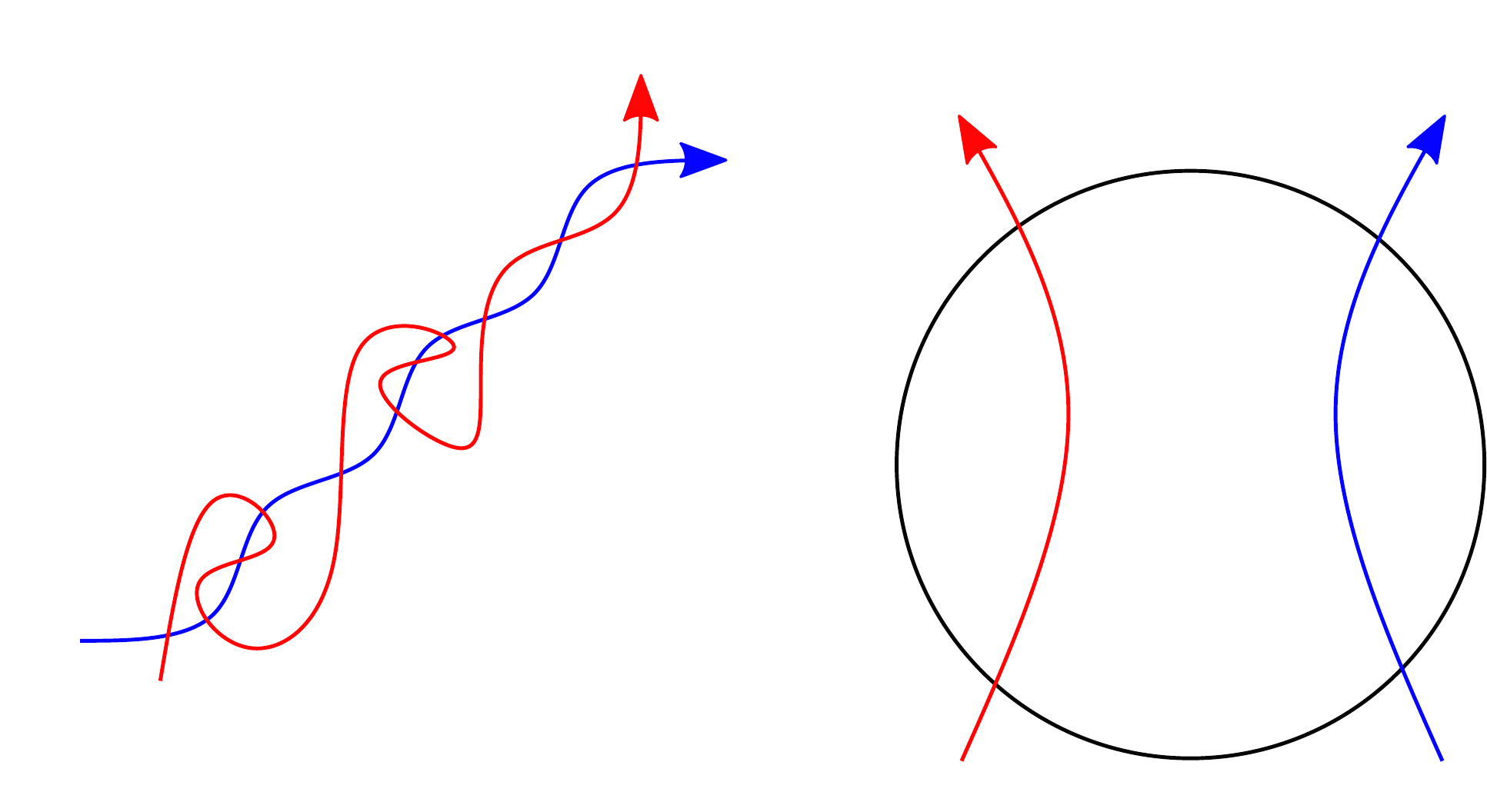}}%
    \put(0.79252515,0.34956145){\color[rgb]{0,0,0}\makebox(0,0)[lt]{\lineheight{1.25}\smash{\begin{tabular}[t]{l}$D$\end{tabular}}}}%
    \put(0,0){\includegraphics[width=\unitlength,page=2]{parallel_train_tracks.pdf}}%
  \end{picture}%
\endgroup%

  \caption{Parallel train-tracks.}
  \label{fig:parallel_tt}
\end{figure}

Let us now consider a triple of oriented train-tracks~$(\tr_1,\tr_2,\tr_3)$ of~$\Gs$, pairwise non-parallel.
\new{Consider a compact disk~$B$ outside of which these train-tracks do not meet, apart from possible anti-parallel ones,
and order~$(\tr_1,\tr_2,\tr_3)$ cyclically according to the outgoing points of the corresponding oriented curves in the circle~$\partial B$. Note that since the three oriented train-tracks are pairwise non-parallel, this cyclic order does not depend on the choice of the disk~$B$.}

\begin{defi}
\label{def:order}
We call this partial cyclic order on~$\Tbip$ the \emph{global cyclic order}.
\end{defi}

\begin{rem}\label{rem:order}$\,$
\begin{itemize}
 \item[$\bullet$]
   When $\Gs$ is $\ZZ^2$-periodic, it is easy to see that this notion of partial order coincides with the one described above. In this context, being parallel (resp. anti-parallel) is equivalent to having the same homology class (resp. opposite homology classes).
 \item In~\cite{KeSchlenk}, Kenyon and Schlenker consider a (non-necessarily bipartite) planar graph~$\Gs$ whose train-tracks do not form closed loops, do not self-intersect, and such that any two train-tracks intersect at most once. The authors construct a topological embedding of~$\Gs^\T$ into the unit disc such that the image of each train-track is a smooth path connecting distinct boundary points of the disc, see~\cite[Lemma~3.2]{KeSchlenk}. This construction can be adapted to our setting, yielding a continuous map
 from $\Gs^\T$ to the unit disk, where parallel train-tracks connect the same boundary points of the disc. Furthermore, the (partial) cyclic order on~$\Tbip$ given by the outgoing points of oriented train-tracks in the unit circle coincides with the global cyclic order defined above.
\end{itemize}
\end{rem}

The first result is the elementary, yet fundamental observation that for minimal graphs, this global cyclic order on~$\Tbip$ induces the standard cyclic order around any vertex. More formally, consider a bipartite planar graph~$\Gs$, and a vertex~$v\in\Vs$ of degree~$n\ge 1$.
The~$n$ incident edges are crossed by a set $\Tbip(v)$ of~$n$ oriented train-tracks strands,
each one joining two consecutive edges, see Figure~\ref{fig:t_bip_v} (left). In general, it can happen that two of these strands belong to the same train-track, e.g. in the case of multiple edges. Furthermore, there is no reason in general for the global cyclic order on~$\Tbip$ to restrict to the \emph{local cyclic order} on~$\Tbip(v)$, \emph{i.e.}, the total cyclic order given by outgoing points of these strands. This actually holds if~$\Gs$ is minimal as stated by the following lemma. Note that in the $\ZZ^2$-periodic case, this is essentially the content of Lemma 4.2. of~\cite{Akira_Ueda}. 

\begin{figure}[htb]
  \centering
  \def\svgwidth{10cm}
\begingroup%
  \makeatletter%
  \providecommand\color[2][]{%
    \errmessage{(Inkscape) Color is used for the text in Inkscape, but the package 'color.sty' is not loaded}%
    \renewcommand\color[2][]{}%
  }%
  \providecommand\transparent[1]{%
    \errmessage{(Inkscape) Transparency is used (non-zero) for the text in Inkscape, but the package 'transparent.sty' is not loaded}%
    \renewcommand\transparent[1]{}%
  }%
  \providecommand\rotatebox[2]{#2}%
  \newcommand*\fsize{\dimexpr\f@size pt\relax}%
  \newcommand*\lineheight[1]{\fontsize{\fsize}{#1\fsize}\selectfont}%
  \ifx\svgwidth\undefined%
    \setlength{\unitlength}{363.09954815bp}%
    \ifx\svgscale\undefined%
      \relax%
    \else%
      \setlength{\unitlength}{\unitlength * \real{\svgscale}}%
    \fi%
  \else%
    \setlength{\unitlength}{\svgwidth}%
  \fi%
  \global\let\svgwidth\undefined%
  \global\let\svgscale\undefined%
  \makeatother%
  \begin{picture}(1,0.40552087)%
    \lineheight{1}%
    \setlength\tabcolsep{0pt}%
    \put(0,0){\includegraphics[width=\unitlength,page=1]{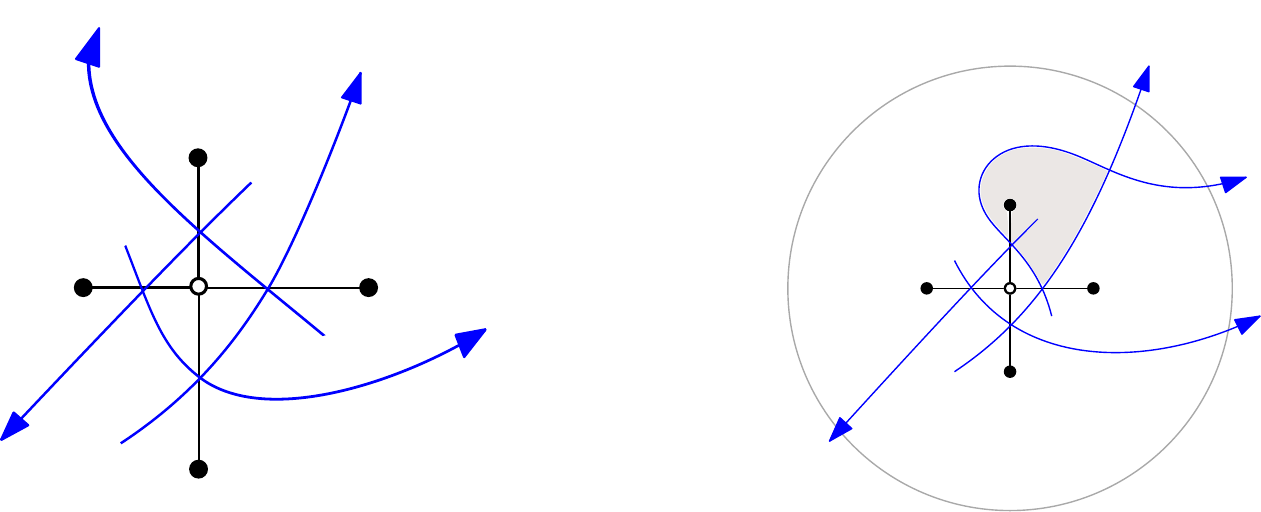}}%
    \put(0.25839025,0.18747075){\color[rgb]{0,0,0}\makebox(0,0)[lt]{\lineheight{1.25}\smash{\begin{tabular}[t]{l}$e_1$\end{tabular}}}}%
    \put(0.12068526,0.25167053){\color[rgb]{0,0,0}\makebox(0,0)[lt]{\lineheight{1.25}\smash{\begin{tabular}[t]{l}$e_2$\end{tabular}}}}%
    \put(0.29237606,0.3588384){\color[rgb]{0,0,0}\makebox(0,0)[lt]{\lineheight{1.25}\smash{\begin{tabular}[t]{l}$\vec{t}_1$\end{tabular}}}}%
    \put(0.0043158,0.38096185){\color[rgb]{0,0,0}\makebox(0,0)[lt]{\lineheight{1.25}\smash{\begin{tabular}[t]{l}$\vec{t}_2$\end{tabular}}}}%
    \put(0.36372134,0.10004705){\color[rgb]{0,0,0}\makebox(0,0)[lt]{\lineheight{1.25}\smash{\begin{tabular}[t]{l}$\vec{t}_n$\end{tabular}}}}%
    \put(0.16126753,0.04991777){\color[rgb]{0,0,0}\makebox(0,0)[lt]{\lineheight{1.25}\smash{\begin{tabular}[t]{l}$e_n$\end{tabular}}}}%
    \put(0.16212861,0.18357971){\color[rgb]{0,0,0}\makebox(0,0)[lt]{\lineheight{1.25}\smash{\begin{tabular}[t]{l}$v$\end{tabular}}}}%
  \end{picture}%
\endgroup%

  \caption{Left: the set $\Tbip(v)$ of oriented train-tracks strands around  the (white) vertex~$v$. Right: Proof of Lemma~\ref{lem:vertex_1}.}
  \label{fig:t_bip_v}
\end{figure}
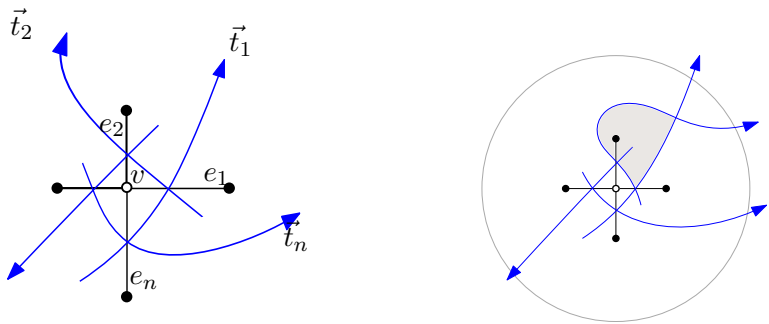

\begin{lem}
\label{lem:vertex_1}
Let~$\Gs$ be a planar, bipartite, minimal graph, and let~$v$ be an arbitrary vertex of~$\Gs$.
Then, the elements of~$\Tbip(v)$ belong to distinct
train-tracks, and the global cyclic order on~$\Tbip$ restricts to the local cyclic order on~$\Tbip(v)$.
\end{lem}
\begin{proof}
Fix a vertex~$v$ of a minimal bipartite graph~$\Gs$,
and consider the set~$\Tbip(v)$ of adjacent train-tracks strands
as illustrated in Figure~\ref{fig:t_bip_v} (left).
First, observe that in order for two of these strands to belong to the same train-track, we need to connect the outgoing point of one of these strands with the ingoing point of another one. As one easily checks, this either introduces a self-intersection or a parallel bigon, thus contradicting minimality of~$\Gs$.
Finally, assume by means of contradiction that the global cyclic order on~$\Tbip$ does not restrict to the local cyclic order on~$\Tbip(v)$. Then, this inevitably creates a parallel bigon, see e.g. Figure~\ref{fig:t_bip_v} (right), contradicting the minimality of~$\Gs$, and concluding the proof.
\end{proof}

Our second result deals with the local order of train-tracks around faces.
Since our graph~$\Gs$ is bipartite, any face~$f$ of~$\Gs$ is of even degree, say~$2m$; in the degenerate case of a degree~$1$ vertex inside the face, the adjacent edge is counted twice.

The $2m$ edges bounding $f$
are crossed by a set~$\Tbip(f)$ of~$2m$ oriented
train-tracks strands,
each one joining two consecutive boundary edges,
see Figure~\ref{fig:tt_around_face}. There is a natural partition~$\Tbip(f)=\Tbip^\bullet(f)\sqcup\Tbip^\circ(f)$, where~$\Tbip^\bullet(f)$ is the set of strands turning around a black vertex of~$\partial f$, \emph{i.e.}, turning counterclockwise around~$f$, and~$\Tbip^\circ(f)$ is the set of strands turning around a white vertex of~$\partial f$, \emph{i.e.}, clockwise around~$f$. In general, all the train-track strands in~$\Tbip^\bullet(f)$ can belong to the same train-track, and similarly for~$\Tbip^\circ(f)$:
a face of degree~$2$ is the easiest example. Furthermore, there is no reason in general for the global cyclic order on~$\Tbip$ to restrict to the \emph{local cyclic orders} on~$\Tbip^\bullet(f)$ and on~$\Tbip^\circ(f)$, \emph{i.e.}, the total cyclic orders given by outgoing points of these strands. Once again, this turns out to hold if~$\Gs$ is minimal.

\begin{figure}[ht]
  \centering
  \def\svgwidth{6cm}
  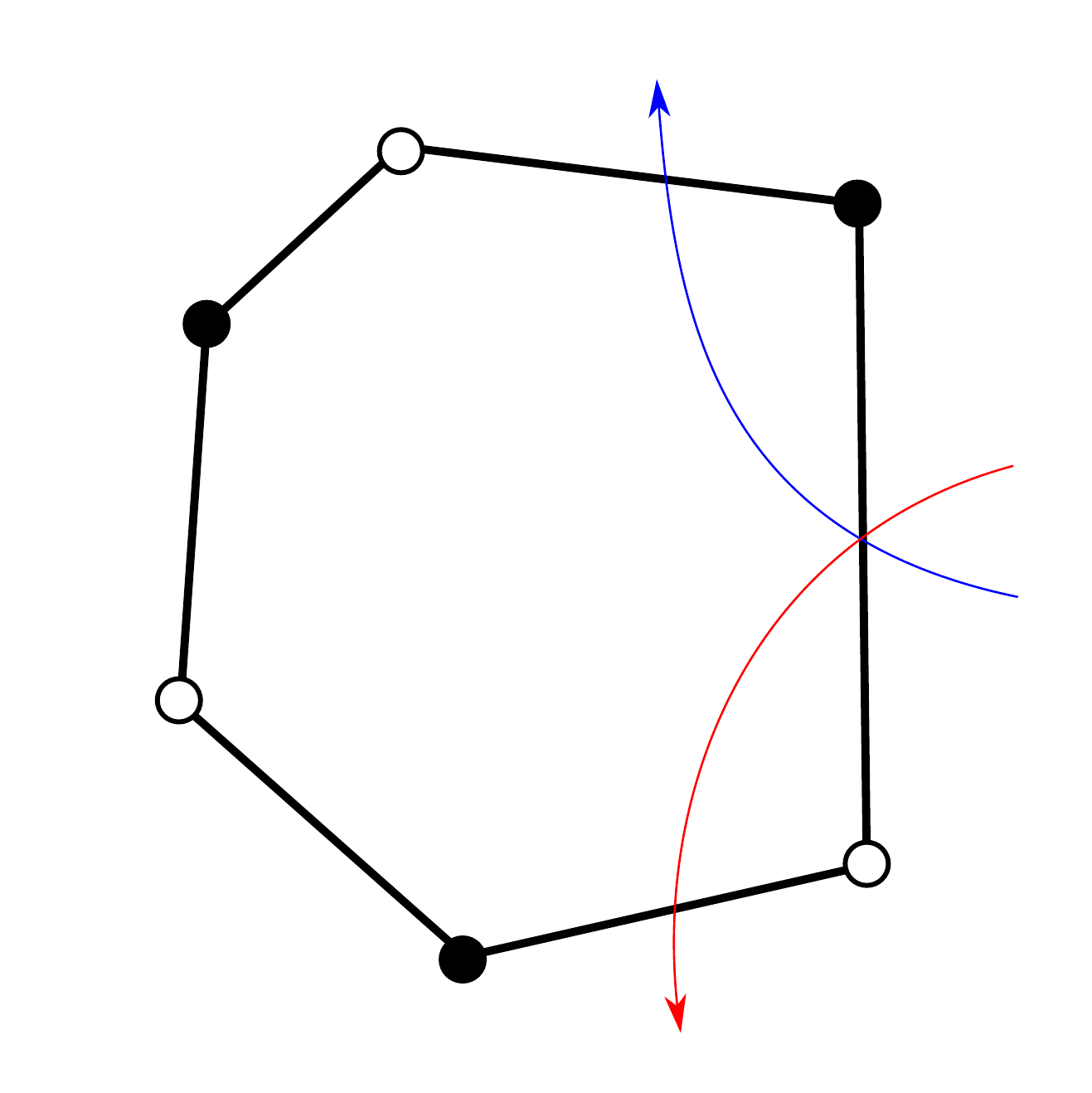
  \caption{The set $\Tbip(f)=
    \Tbip^{\circ}(f)\sqcup   \Tbip^{\bullet}(f)$.
  }
  \label{fig:tt_around_face}
\end{figure}

\begin{lem}
\label{lem:face_1}
Let~$\Gs$ be a planar, bipartite, minimal graph, and let~$f$ be an arbitrary face of~$\Gs$. Then, there is a pair of train-track strands in~$\Tbip^\bullet(f)$ which belong to distinct non-parallel train-tracks, and similarly for~$\Tbip^\circ(f)$. Furthermore, the global cyclic order on~$\Tbip$ restricts to the local cyclic orders on~$\Tbip^\bullet(f)$ and on~$\Tbip^\circ(f)$.
\end{lem}
\begin{proof}
  Let us fix a face~$f$ of a bipartite minimal graph~$\Gs$, and assume by means of contradiction that all train-track strands in~$\Tbip^\bullet(f)$ belong to the same train-track. Using the fact that train-tracks cannot self-intersect, we obtain that these strands are connected cyclically as illustrated in Figure~\ref{fig:bigons_around_face} (left). This creates a parallel bigon, thus contradicting the minimality of~$\Gs$. Furthermore, two elements of~$\Tbip^\bullet(f)$ cross the face~$f$ in anti-parallel fashion, so they cannot belong to parallel train-tracks 
  (which have to be disjoint since the graph is minimal). The same argument holds for~$\Tbip^\circ(f)$. Let us finally assume that the global cyclic order on~$\Tbip$ does not restrict to the local cyclic orders on~$\Tbip^\bullet(f)$ or on~$\Tbip^\circ(f)$. In such a case, we inevitably have a self-intersection or a parallel bigon, see e.g. Figure~\ref{fig:bigons_around_face} (right). This contradicts the minimality of~$\Gs$ and concludes the proof.
\end{proof}

  \begin{figure}[ht]
  \centering
  \def\svgwidth{10cm}
  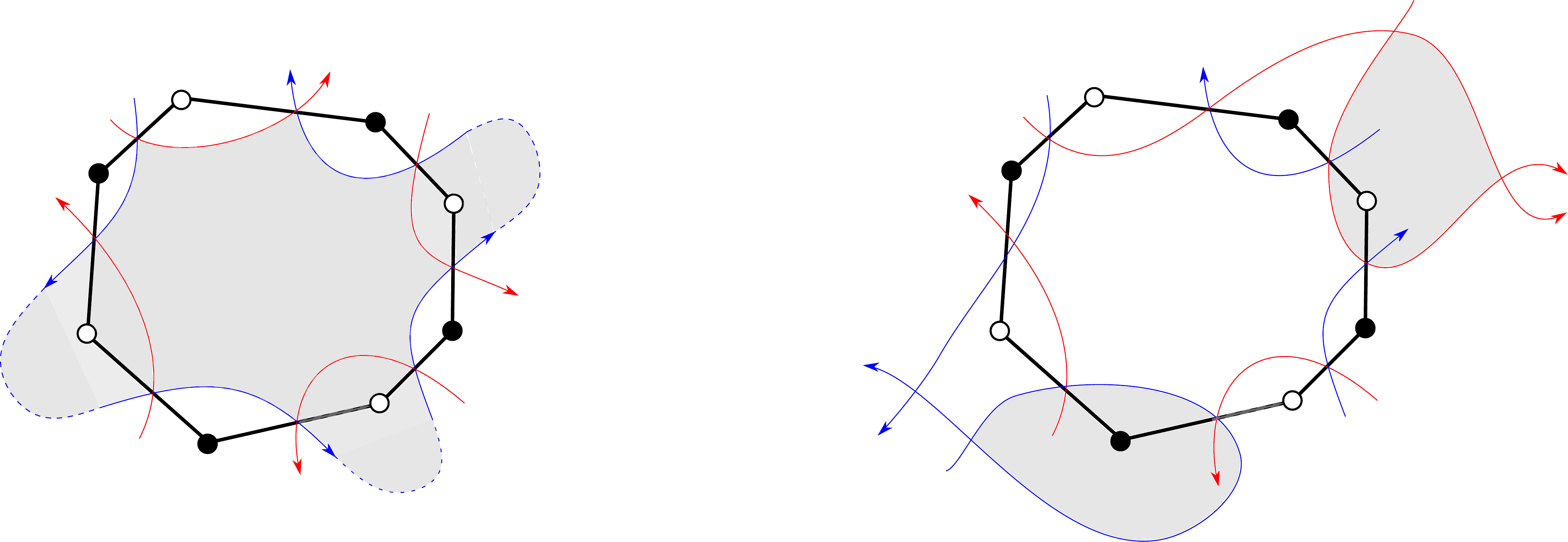
  \caption{Proof of Lemma~\ref{lem:face_1}.
    Left: strands of $\Tbip^\bullet(f)$ coming from the same train-track. Right:
    global cyclic order not restricting to local cyclic orders of $\Tbip^\circ(f)$
    and $\Tbip^\bullet(f)$. In both cases, shaded areas are bounded by parallel
    bigons or self-intersections.
  }
  \label{fig:bigons_around_face}
\end{figure}  

A natural question is whether some converse to these lemmas holds.

\begin{rem}
In~\cite{Akira_Ueda}, the authors prove the following statement: if a $\ZZ^2$-periodic bipartite planar graph~$\Gs$ is such that its train-tracks do not form closed loops, do not self-intersect, two parallel train-tracks are always disjoint, and for any vertex~$v$, the global cyclic order on~$\Tbip$ restricts to the local cyclic order on~$\Tbip(v)$, then~$\Gs$ is minimal. Note however that the biperiodicity is used in a crucial way, more precisely the fact that the train-tracks define a finite number of asymptotic directions.

As a consequence of Theorem~\ref{thm:min}, we actually obtain a converse to these two lemmas: if a bipartite planar graph~$\Gs$ has no train-track loops, no vertex of degree~$1$,
and is such that the conclusions of Lemmas~\ref{lem:vertex_1} and~\ref{lem:face_1} hold, then~$\Gs$ is minimal, see Corollary~\ref{cor:min} of Section~\ref{sec:minimal}.
\end{rem}

We conclude this section with the following:

\begin{question}
Given a minimal bipartite graph~$\Gs$, do the cyclic orders on~$\Tbip(v)$ for all~$v\in\Vs$ and on~$\Tbip^\bullet(f)$ and~$\Tbip^\circ(f)$ for all~$f\in\Fs$ determine the global cyclic order on~$\Tbip$?
\end{question}

As a consequence of our results, we will answer this question positively in
the~$\ZZ^2$-periodic case, see Corollary~\ref{cor:X=Y} of Section~\ref{sub:proofmin}.

\section{Flat isoradial immersions of planar graphs}
\label{sec:flat}

This section deals with a geometric notion, called \emph{isoradial immersion} and defined in Section~\ref{sub:def}, which extends the classical notion of isoradial embedding.
Our main result is stated in Section~\ref{sub:statement_main}; it is the pendent in this more general context of Kenyon and Schlenker's theorem~\cite{KeSchlenk};
the proof is the subject of Section~\ref{sub:proof_statement_main}. One last Section~\ref{sub:rem} proves additional features of our theory.

\subsection{Definitions}
\label{sub:def}

\new{We start this section by defining the notion of \emph{rhombic immersion}.}
Consider a planar graph $\Gs$, the corresponding quad-graph $\GR$, and the graph of train-tracks $\Gs^\T$. 
Suppose that each pair of directed train-tracks~$\tr,\tl\in\Torient$ is assigned a pair of opposite directions~$e^{i\alpha(\tr)}$ and~$e^{i\alpha(\tl)}=-e^{i\alpha(\tr)}$; the angle~$\alpha(\tr),\alpha(\tl)\in\RR/2\pi\ZZ$ are referred to as the \emph{angles of the train-tracks} $\tr,\tl$, and we denote by $Z_\Gs$ the set of possible angle maps, i.e.
\[
Z_\Gs=\{\alpha:\Torient\rightarrow \RR/2\pi\ZZ \,|\,\forall\,\tr\in \Torient,\  \alpha(\tl)=\alpha(\tr)+\pi \}\,.
\]

Every element $\alpha\in Z_\Gs$ defines an immersion of $\GR$ in
$\mathbb{R}^2$, by realizing every directed edge of $\GR$ crossed by a
train-track $\tr$ from left to right as a translation of the unit vector
$e^{i\alpha(\tr)}$. Every inner (quadrilateral) face of $\GR$ 
is mapped to a rhombus of unit-edge length in the plane. We call such an immersion a \emph{rhombic immersion} of $\GR$.
Given the cyclic order around vertices in $\GR$, the data of the rhombic
immersion is equivalent to that of~$\alpha$.

Every face of $\GR$ corresponds to an edge $e$ of $\Gs$ by construction.
Using the notation of Section~\ref{sub:gen_def} and Figure~\ref{fig:immersed_quad},
we associate to the oriented edge~$e=(v_1,v_2)$,
a \emph{rhombus angle}~$\theta_e\in[0,2\pi)$, which is the unique lift in~$[0,2\pi)$ of~$[\alpha(\tr_2)-\alpha(\tr_1)]\in\RR/2\pi\ZZ$, where the square bracket denotes the equivalence class of a real number in~$\RR/2\pi\ZZ$.
\new{Recall from Section~\ref{sub:gen_def} that the orientation of~$\tr_1$, resp. $\tr_2$, is chosen to be coherent with the orientation of~$e$ and that the subscripts $1,2$ are chosen so that $\tr_1,\tr_2$ defines a positive orientation of the plane}. This rhombus angle is independent of the orientation of the edge $e$, \new{as reversing the orientation of~$e$ also reverses the orientation of both train-tracks while keeping the subscripts fixed}; 
for an internal edge $e$,
it is the geometric angle of the rhombus image of the quadrilateral corresponding to $e$, at the vertices (corresponding to the images by the immersion of) $v_1$ and $v_2$.
The \emph{dual rhombus angle}, measured at vertices $f_1$ or $f_2$ is given by
$\theta_{e^*}=\pi-\theta_e$.

\begin{figure}[ht]
  \centering
  \def\svgwidth{6cm}
  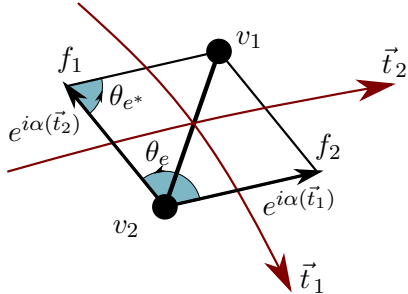
  \caption{The image by a rhombic immersion of an inner (quadrilateral) face of
    $\GR$, with vertices $v_1$ and $v_2$ corresponding to vertices of $\Gs$, and $f_1$ and $f_2$
    corresponding to faces of $\Gs$. Edges are represented by unit vectors associated to train-tracks.}
  \label{fig:immersed_quad}
\end{figure}

Note that in many papers, see e.g.~\cite{Kenyon:crit}, the notation~$\theta_e$ is used for the rhombus \emph{half-angle}, but in our setting
using $\theta_e$ for the rhombus \emph{angle} is more convenient. There are three types of rhombi: \emph{embedded} ones when~$\theta_e$ belongs to~$(0,\pi)$, \emph{folded} ones for~$\theta_e\in(\pi,2\pi)$, and \emph{degenerate} ones when~$\theta_e\in\{0,\pi\}$. 

To each edge $e$ of $\Gs$, we also assign an arbitrary lift~$\tilde{\theta}_e\in\RR$ of
$[\alpha(\vec{t_2})-\alpha(\vec{t_1})]\in\RR/2\pi\ZZ$.
Since~$\tilde{\theta}_e=\theta_e+2\pi k_e$ for a unique~$k_e\in\ZZ$, such a choice of lift amounts to the choice of an element~$k=(k_e)_{e\in\Es}\in\ZZ^\Es$. In a similar way as above, to each edge $e$ we associate a \emph{generalized rhombus} with angle $\tilde{\theta}_e$ at the adjacent vertices, and dual angle~$\tilde{\theta}_{e^*}=\pi-\tilde{\theta}_e$ at the adjacent dual vertices.
Geometrically, such a generalized rhombus should be thought of as being
iteratively folded along its diagonals, see Section~\ref{sub:rem} for more details.

\begin{defi}
  The \emph{isoradial immersion} of $\Gs$ defined by
  the train-track angles $\alpha\in Z_\Gs$
  and the integers~$k\in\ZZ^\Es$
  is the gluing of the generalized rhombi using the combinatorial information of
  $\Gs$.
\end{defi}

The surface obtained by gluing the (generalized) rhombi is naturally
equipped with a flat metric with conical singularities at inner vertices and faces.
On this surface, the graphs $\GR$ and $\Gs$ are naturally \emph{embedded}. When
projecting this surface onto the plane, the images of vertices and edges coincide
with the images of the rhombic immersion defined by $\alpha$.
We define the cone angle at a vertex (resp. face) singularity as the sum of the
lifts~$\tilde{\theta}_e$ (resp.~$\tilde{\theta}_{e^*}$) for all edges incident to this vertex
(resp. face), a total angle easily seen to be a multiple of~$2\pi$.

\begin{defi}
An isoradial immersion is said to be \emph{flat} if all the cone angles are equal to~$2\pi$, \emph{i.e.}, if
for any inner vertex~$v$ and any inner face~$f$ of~$\Gs$, we have
\begin{equation}\label{equ:flat}
\sum_{e\sim v} \tilde{\theta}_e=2\pi\quad \text{and}\quad \sum_{e^* \sim f} \tilde{\theta}_{e^*}=2\pi\,.
\end{equation}
\end{defi}

Here are some remarks on the above definitions.
\begin{rem}\label{rem:def_iso_imm}$\,$
\begin{enumerate}
  \item Given the embedding of $\GR$ in the plane, an isoradial immersion is
    equivalent to the data of an element $\alpha\in Z_\Gs$ (or equivalently of a
    rhombic immersion) and of an element $k\in\ZZ^\Es$.
 \item Let us recall, see for example~\cite{Kenyon:crit,KeSchlenk}, that the terminology \emph{isoradial} stems from the fact that in a rhombic immersion of $\GR$, the boundary vertices of each primal/dual face corresponding to a dual/primal vertex $f/v$
 are mapped to a unit circle whose center is the immersed vertex $f/v$.
 \item The notion of flat isoradial immersion can be formulated very naturally in terms of the graph of train-tracks~$\Gs^\T$: the angles~$\alpha$ are associated to the train-tracks, the integers~$k$ are associated to the vertices of~$\Gs^\T$, and they need to verify one condition for each inner vertex $v$ and inner face $f$ of~$\Gs$, \emph{i.e.}, for each bounded face of~$\Gs^\T$:
\begin{align}\label{equ:flat_2}
\sum_{e\sim v} k_e=\textstyle{1-\frac{1}{2\pi}s_v}\quad\text{and}\quad -\displaystyle{\sum_{e^* \sim f} k_e} =\textstyle{1-\frac{1}{2\pi}s_f},
\end{align}
where $s_v=\sum_{e\sim v}\theta_e$ and~$s_f=\sum_{e^*\sim f}(\pi-\theta_e)$.
 \item In the absence of unbounded faces, the notion of flat isoradial immersion
 is self-dual up to a sign change of $k$: by exchanging
   simultaneously the roles on the one hand of vertices and faces of $\Gs$, and
   on the other hand of primal and dual edges in~\eqref{equ:flat_2}, one sees that
   $\alpha$ and $k$ define an isoradial immersion of $\Gs$ if and only if $\alpha$ and $-k$
   define an isoradial immersion of the dual $\Gs^*$, see also Section~\ref{sub:rem}.
    In the general case, this
   notion remains locally self-dual away from the unbounded faces.
\end{enumerate}
\end{rem}

In the specific case where $k\equiv 0$, \emph{i.e.}, when lifted rhombus angles
$\tilde{\theta}_e\in[0,2\pi)$, and when the latter are moreover restricted to
being in $(0,\pi)$, a flat isoradial immersion is an \emph{isoradial embedding} as introduced in~\cite{Kenyon:crit,KeSchlenk}. 
As mentioned earlier, Kenyon and Schlenker~\cite{KeSchlenk} prove that a planar embedded graph~$\Gs$ admits an isoradial embedding if and only if the train-tracks make neither closed loops nor self-intersections and two train-tracks never intersect more than once. Furthermore, they identify the space of all isoradial embeddings of such a graph~$\Gs$ as a subspace of~$Z_\Gs$. The aim of this section and the next is to prove 
a theorem of the same flavor in the more general setting of isoradial immersions. 

\subsection{Statement of the main result}
\label{sub:statement_main}

The main question answered in this section is that of characterizing the planar embedded graphs admitting a flat isoradial immersion, and of understanding the space of such immersions. Stating this result is the subject of the present section; the proof is provided in Section~\ref{sub:proof_statement_main}.

We need one preliminary remark. Consider a planar embedded graph~$\Gs$, together with angles~$\alpha\in Z_\Gs$ and integers~$k\in\ZZ^\Es$ such that the corresponding isoradial immersion is flat. Moving along the entirety of a fixed train-track~$t\in\T$, we encounter vertices of~$\Gs^\T$, \emph{i.e.}, edges of~$\Gs$: alternatively add~$+1$ and~$-1$ to the corresponding integers~$k_e$. Note that vertices corresponding to self-intersections of~$t$ receive a contribution of~$+1-1=0$;
\new{this can be checked by recalling that train-tracks, also known as \emph{zig-zag paths}, alternately turn 
right and left when crossing edges}.
This defines a new set of integers~$k'=(k'_e)$, that are said to be obtained from~$k$ by a \emph{shift along~$t$}.
One easily checks that the isoradial immersion given by the same~$\alpha\in Z_\Gs$ and this new set of integers~$k'\in\ZZ^\Es$ is still flat: indeed, for any fixed face of~$\Gs^\T$, the vertices of~$t$ that belong to this face (if any) can be grouped in pairs of consecutive vertices of~$t$, whose added contributions vanish. (Vertices corresponding to self-intersections will appear twice in this count, with opposite signs).
\new{Note that these shift operations are commutative: shifting a set of integers along~$t_1$
and then along~$t_2$ gives the same result as shifting it first along~$t_2$ and then along~$t_1$.}

This motivates the following definition.

\begin{defi}
\label{def:equiv}
Given~$\alpha\in Z_\Gs$, two sets of integers~$k,k'\in\ZZ^\Es$ satisfying the flatness condition~\eqref{equ:flat_2} are said to be \emph{equivalent} if~$k'$ can be obtained from~$k$ via \new{a potentially infinite set of} shifts along train-tracks, \new{with each train-track supporting a finite number of shifts}.
\end{defi}

We now have all the ingredients to state our main theorem.

\begin{thm}\label{thm:main}
An arbitrary planar, embedded graph~$\Gs$ admits a flat isoradial immersion if and only if its train-tracks do not form closed loops. When this is the case, for every~$\alpha\in Z_\Gs$, there exists~$k\in\ZZ^\Es$, unique up to equivalence, such that~$\alpha$ and~$k$ define a flat isoradial immersion of~$\Gs$.
\end{thm}

We mention one immediate reformulation of this theorem. Let us say that two flat isoradial immersions of~$\Gs$ are \emph{equivalent} if given by the same~$\alpha$ and equivalent~$k$'s. Then, the space of equivalence classes of flat isoradial immersions of~$\Gs$ is given by~$Z_\Gs$ if its train-tracks do not form closed loops, and is empty otherwise.

Proving Theorem~\ref{thm:main} consists in solving the linear
system~\eqref{equ:flat}. As noted in the introduction, the same system appears
in a surprisingly different setting (equations for R-charges for supersymmetric
fields). In~\cite{Gulotta},
the author solves this system for a class of $\ZZ^2$-periodic bipartite graphs,
containing periodic minimal graphs.
Some elements in the proof like, equivalence classes of solutions,
evolution of the solution under elementary moves,
sketched there without detailed computations,
have the same flavor as the arguments we use to tackle
the problem in greater generality.

\subsection{Proof of Theorem~\ref{thm:main}}
\label{sub:proof_statement_main}

In order to prove Theorem~\ref{thm:main}, we need two preliminary results. 
The first, Lemma~\ref{lem:flat},
consists in translating the equations~\eqref{equ:flat} defining flatness, each
of which is an equation on a face cycle of $\Gs^\T$, into equations on oriented
closed curves of $\Gs^\T$, as defined in Section~\ref{sub:cycles}. The main
content is that, for each closed curve, the flatness equations
combine
into an equation involving corner vertices only. The second, 
Proposition~\ref{prop:Reidemeister}, computes the evolution of a solution $k$ to
the flatness equations~\eqref{equ:flat_2} under elementary local moves. Although of independent interest, it is used to prove that a graph $\Gs$ having a train-track that forms a closed loop does not have an isoradial immersion. 
The actual proof of Theorem~\ref{thm:main} comes after these two preliminary results.

\bigskip

\textbf{Flatness equations on closed curves.} Recall the definition of closed curves of the graph $\Gs^\T$ given at the beginning of Section~\ref{sub:cycles}, and that of corners.
Whereas a closed curve was interpreted there simply as a subgraph because of
its description as an $\mathbb{F}_2$-chain, we see it here
as an oriented curve, bounding a finite collection of oriented faces, possibly
with multiplicities, \emph{i.e.}, as a $\ZZ$-chain.
There are exactly two possible orientations for a closed curve. We choose one of
them by orienting arbitrarily one of the edges, and propagating this orientation
along the other edges by the drawing procedure presented in
Section~\ref{sub:cycles}.

Following the standard notation coming from algebraic topology, we write this as~$c=\partial(\sum_{\fs}n_\fs\fs)$ with~$n_\fs\in\ZZ$ vanishing for all but finitely many faces. An example is given in Figure~\ref{fig:oriented_closed_curve}.

\begin{figure}[ht]
  \centering
  \def\svgwidth{4cm}
  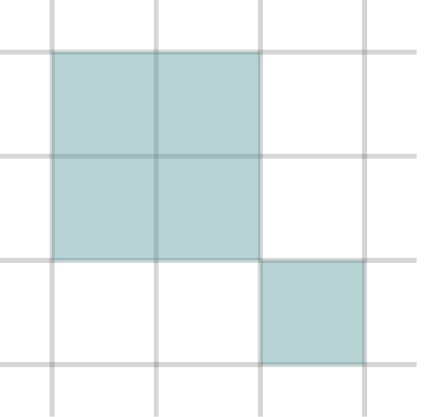
  \caption{An oriented closed curve $c$ and labelled faces such that
  $c=\partial(\fs_1+\fs_2+\fs_3+\fs_4-\fs_5)$. Vertices $e_1$, $e_2$, $e_3$, $e_4$ of the
four types are also indicated.}
  \label{fig:oriented_closed_curve}
\end{figure}

The following lemma shows that the flatness condition allows to uniquely determine
the value of $k$ at a corner of a closed curve as soon as we know its
value at all the other corners.

\begin{lem}
\label{lem:flat}
Fix angles~$\alpha\in Z_\Gs$.
Let~$c$ be any closed curve in~$\Gs^\T$, and $e_0$ a corner of $c$.
Then, the flatness condition inside
all faces appearing with non-zero multiplicity
in~$c=\partial(\sum_{\fs}n_\fs\fs)$ allows to express~$k_{e_0}$ as an affine
function of the values of $k$ at the other corners,
which does not depend on the integers associated to vertices that are not
corners of~$c$.

In particular, if the value of
$k$ at all corners except $e_0$ are fixed
arbitrarily, then the value at $e_0$ is determined uniquely.

\end{lem}
\begin{proof}
  Consider the flatness equations~(\ref{equ:flat}) for all faces, and add them up with their respective multiplicities~$n_\fs$. We now show that for all vertices of~$\Gs^\T$ that are not corners of~$c$, the corresponding variable~$\tilde{\theta}_e$ does not appear in the resulting equation. Indeed, there are four types of such vertices, labelled~$e_1,e_2,e_3,e_4$ in Figure~\ref{fig:oriented_closed_curve}: 
\begin{enumerate}
\item{vertices disjoint from the faces bounded by~$c$;}
\item{vertices of degree~$2$ of~$c$, where~$c$ stays on the same train-track;}
\item{vertices of degree~$4$ of~$c$, where~$c$ intersects itself transversally;}
\item{vertices not on~$c$ but inside a collection of faces bounded by~$c$.}
\end{enumerate}
In the first case, the variable~$\tilde{\theta}_e$ does not appear in the equation, so there is nothing to prove. In the second one, using the notation of Figure~\ref{fig:contrib_faces_around_vertex} (left), the four corresponding adjacent faces~$\fs_1,\fs_2,\fs_3,\fs_4$ satisfy~$n_{\fs_1}=n_{\fs_2}=n_{\fs_3}-1=n_{\fs_4}-1$ and contribute~$(n_{\fs_1}+n_{\fs_3})\tilde{\theta}_e+(n_{\fs_2}+n_{\fs_4})\tilde{\theta}_{e^*}=(n_{\fs_1}+n_{\fs_3})\pi$, so~$\tilde{\theta}_e$ indeed does not appear. In the third case, the four corresponding adjacent faces~$\fs_1,\fs_2,\fs_3,\fs_4$ satisfy~$n_{\fs_1}=n_{\fs_3}=n_{\fs_2}-1=n_{\fs_4}+1$ and contribute~$(n_{\fs_1}+n_{\fs_3})\tilde{\theta}_e+(n_{\fs_2}+n_{\fs_4})\tilde{\theta}_{e^*}=(n_{\fs_1}+n_{\fs_3})\pi$, as illustrated in Figure~\ref{fig:contrib_faces_around_vertex} (right). Finally, in the last case, the four corresponding adjacent faces all have the same multiplicity~$n_\fs$ and contribute~$n_\fs(2\tilde{\theta}_e+2\tilde{\theta}_{e^*})=2\pi n_\fs$.

\begin{figure}[ht]
  \centering
  \def\svgwidth{10cm}
  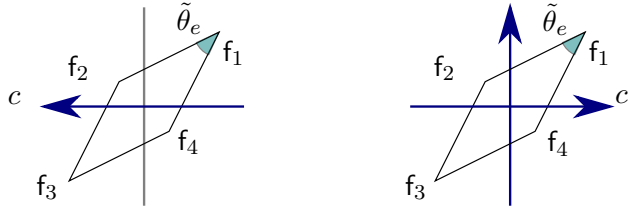
  \caption{Contributions of the faces of $\Gs^{\T}$ (\emph{i.e.}, rhombus vertices) around a vertex of a closed oriented curve $c$.}
  \label{fig:contrib_faces_around_vertex}
\end{figure}
As a consequence, the resulting equation is a linear equation in the integers~$k_e$ corresponding to the corners of~$c$. The lemma follows.
\end{proof}

\bigskip

\textbf{Elementary local moves}.
Used in the proof of Theorem~\ref{thm:main} and of independent interest, the
following proposition describes the evolution of the solution $k$ to flatness
under elementary local moves on the graph $\Gs^\T$, which are derived from
Reidemeister moves from knot theory, by forgetting information about crossings,
see Figure~\ref{fig:reidemeister}.

\begin{figure}[ht]
  \centering
  \def\svgwidth{12cm}
  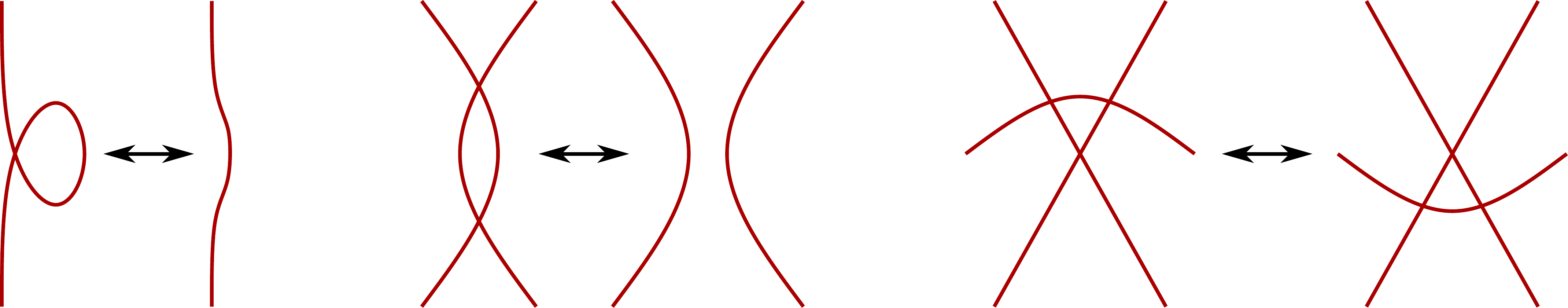
  \caption{The three elementary local moves on the graph $\Gs^\T$, analogous to
  the Reidemeister moves.}
  \label{fig:reidemeister}
\end{figure}

\begin{prop}
\label{prop:Reidemeister}
Let~$\Gs$ and~$\Gs'$ be two planar, embedded graphs whose train-tracks are related by a sequence of the three local transformations illustrated in Figure~\ref{fig:reidemeister}. Fix~$\alpha\in Z_\Gs$ and integers~$k\in\ZZ^\Es$ defining a flat isoradial immersion of~$\Gs$. Then, there exists~$k'\in\ZZ^{\Es'}$ such that~$\alpha\in Z_\Gs=Z_{\Gs'}$ and~$k'$ define a flat isoradial immersion of~$\Gs'$.  
\end{prop}

\begin{figure}[ht]
  \centering
  \def\svgwidth{\textwidth}
  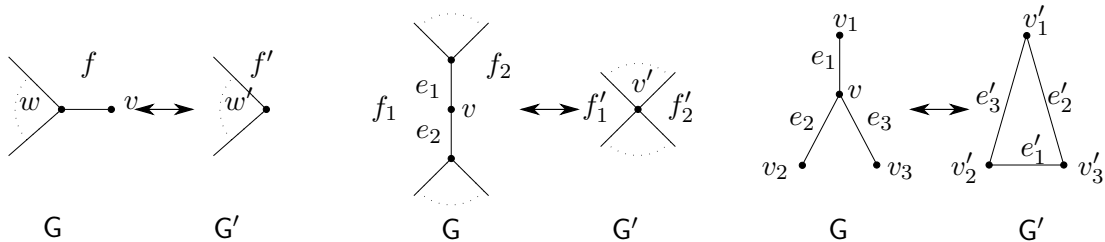
  \caption{Differing neighborhood in $\Gs$ and $\Gs'$ when passing from one to
    the other by the three local moves on train-tracks from
    Figure~\ref{fig:reidemeister}, illustrating the proof of
  Proposition~\ref{prop:Reidemeister}.}
  \label{fig:local_move_G}
\end{figure}

\begin{proof}
Let~$\Gs$ and~$\Gs'$ be two planar graphs with train-tracks related by the first
elementary move. The notion of flat isoradial immersion being self-dual, it can
be assumed that the self-intersection in~$\Gs^\T$ corresponds to an
edge~$e=wv\in\Es$, with~$v$ a degree~$1$ vertex inside the train-track loop (see
Figure~\ref{fig:local_move_G}, left). Let us denote by~$f\in\Fs$ the face
of~$\Gs$ adjacent to~$v$, by~$w'\in\Vs'$ the vertex of~$\Gs'$ corresponding
to~$w\in\Vs$, and by~$f'\in\Fs'$ the face of~$\Gs'$ corresponding to~$f$, as
illustrated in Figure~\ref{fig:local_move_G}. Let us first fix~$\alpha\in Z_\Gs$ and integers~$k\in\ZZ^\Es$ defining a flat isoradial immersion of~$\Gs$. Since the two train-tracks strands intersecting at~$e$ belong to the same train-track~$t$, we have~$\theta_e=0$ which by flatness at~$v$ implies the value~$k_e=1$. Undoing the train-track loop amounts to removing the corresponding edge~$e$, and rhombus, whose angles are~$\tilde{\theta}_e=2\pi$ and~$\tilde{\theta}_{e^*}=-\pi$. This creates an angle defect in~$\Gs'$: with the same angles~$\alpha\in Z_\Gs=Z_{\Gs'}$ and integers~$k$ restricted to~$\Es'=\Es\setminus\{e\}$, the angle at~$w'$ is equal to~$2\pi-2\pi=0$, and the angle inside the face~$f'$ is given by~$2\pi-((-\pi)+(-\pi))=4\pi$. To recover flatness, pick an arbitrary direction on~$t$, and starting at~$f'$, alternatively add~$+1$ and~$-1$ to the integers corresponding to the encountered vertices. One easily checks that the resulting~$k'\in\ZZ^{\Es'}$ defines a flat isoradial immersion of~$\Gs'$. (Note also that choosing the other direction of~$t$ gives equivalent integers.) Conversely, let us fix~$\alpha\in Z_{\Gs'}=Z_\Gs$ and integers~$k'\in\ZZ^{\Es'}$ defining a flat isoradial immersion of~$\Gs'$. Flatness at~$v$ requires the value~$k_e=1$ for the new rhombus of~$\Gs$, which creates an angle defect of~$\pm 2\pi$ in the adjacent vertex~$w$ and face~$f$ of~$\Gs$. However, flatness can be recovered by the same strategy as above.

Let us now consider two planar graphs~$\Gs$ and~$\Gs'$ related by the second
elementary move. The notion of flat isoradial immersion being self-dual, we can
assume that these graphs are locally given as described in the center of
Figure~\ref{fig:local_move_G}. We also assume the notation of this figure. Let us fix~$\alpha\in Z_\Gs$ and integers~$k\in\ZZ^\Es$ defining a flat isoradial immersion of~$\Gs$. We claim that the integers~$k'\in\ZZ^{\Es'}$ given by the restriction of~$k$ to~$\Es'=\Es\setminus\{e_1,e_2\}$ defines a flat isoradial immersion of~$\Gs'$. Indeed, flatness at~$v$ implies that the rhombus angles of~$e_1$ and~$e_2$ satisfy~$\tilde{\theta}_{e_1}+\tilde{\theta}_{e_2}=2\pi$; hence, the contribution of these two rhombi to the angle inside the face~$f_1$ is~$\tilde{\theta}_{e^*_1}+\tilde{\theta}_{e^*_2}=(\pi-\tilde{\theta}_{e_1})+(\pi-\tilde{\theta}_{e_2})=0$, and similarly for~$f_2$. Therefore, removing these two rhombi preserves flatness inside the adjacent faces. Furthermore, flatness at~$v_1$ and~$v_2$ amount to equalities~$2\pi=\omega_1+\tilde{\theta}_{e_1}$ and~$2\pi=\omega_2+\tilde{\theta}_{e_2}$ for some~$\omega_1,\omega_2\in\RR$ gathering the rest of the angle contributions. The equation~$\tilde{\theta}_{e_1}+\tilde{\theta}_{e_2}=2\pi$ implies~$\omega_1+\omega_2=2\pi$, which means flatness at the vertex~$v'$. Conversely, let us fix integers~$k'\in\ZZ^{\Es'}$ defining a flat isoradial immersion of~$\Gs'$. Flatness at~$v'$ means~$\omega_1+\omega_2=2\pi$ for some~$\omega_1$ (resp.~$\omega_2$) gathering angle contributions from the rhombi above (resp. below) the vertex~$v'$. There is a unique integer~$k_{e_1}$ (resp.~$k_{e_2}$) such that the corresponding rhombus angle satisfies~$\omega_1+\tilde{\theta}_{e_1}=2\pi$ (resp.~$\omega_2+\tilde{\theta}_{e_2}=2\pi$). The corresponding integers~$k\in\ZZ^\Es$ define an isoradial immersion of~$G$ which is flat at~$v_1$ and~$v_2$ by construction; it is also flat at~$v$ since the newly defined rhombus angles satisfy~$\tilde{\theta}_{e_1}+\tilde{\theta}_{e_2}=2\pi$. Finally, this equality implies that flatness is preserved at adjacent faces, as above.

\begin{figure}[ht]
  \centering
  \def\svgwidth{12cm}
  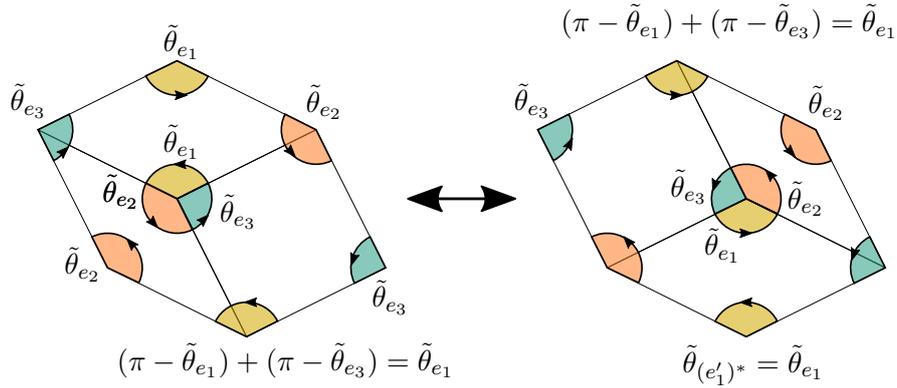
  \caption{Notation for the third elementary move of
  Proposition~\ref{prop:Reidemeister}.}
  \label{fig:angles_trois_losanges}
\end{figure}

Let us finally consider two planar graphs~$\Gs$ and~$\Gs'$ related by the third
elementary move. Without loss of generality, it can be assumed that these graphs
are locally given as illustrated in Figure~\ref{fig:local_move_G}, right.
Fix~$\alpha\in Z_\Gs=Z_{\Gs'}$ and integers~$k\in\ZZ^\Es$ defining a flat
isoradial immersion of~$\Gs$. Using the notation of Figure~\ref{fig:local_move_G}, let~$k'\in\ZZ^{\Es'}$ be given by~$k'_{e'_i}=-k_{e_i}$ for~$i=1,2,3$, and~$k'_e=k_e$ for the other edges, which are common to~$\Gs$ and~$\Gs'$. This choice is motivated by the fact that for any dual edges~$e$ and~$e^*$, the rhombus angles~$\tilde{\theta}_e=\theta_e+2k_e\pi$ and~$\tilde{\theta}_{e^*}=\theta_{e^*}+2k_{e*}\pi$ satisfy~$\tilde{\theta}_{e}+\tilde{\theta}_{e^*}=\theta_{e}+\theta_{e^*}=\pi$, which implies~$k_{e^*}=-k_e$. Therefore, we have the equality~$\tilde{\theta}_{(e_i')^*}=\tilde{\theta}_{e_i}$ for~$i=1,2,3$. Using this together with the equation~$\tilde{\theta}_{e_1}+\tilde{\theta}_{e_2}+\tilde{\theta}_{e_3}=2\pi$, one easily checks that~$k'$ defines a flat isoradial immersion of~$\Gs'$: the corresponding rhombus angle, expressed in terms of~$\tilde{\theta}_{e_1},\tilde{\theta}_{e_2},\tilde{\theta}_{e_3}$, are illustrated in Figure~\ref{fig:angles_trois_losanges}. The converse is proved analogously (or simply using the self-duality of flat isoradial immersions).
\end{proof}

\bigskip

\textbf{Proof of Theorem~\ref{thm:main}.}
Let us first consider a planar graph~$\Gs$ whose train-tracks make at least one closed loop, and assume by means of contradiction that~$\Gs$ admits a flat isoradial immersion. Using the three local moves of Figure~\ref{fig:reidemeister}, the corresponding graph of train-tracks can be transformed to contain a subgraph consisting of a simple closed loop crossed by one train-track, as illustrated in Figure~\ref{fig:reduction_loop} (left):
\new{fix one closed loop, remove superfluous intersections with this loop using the second and third local moves,
and remove self-intersections of the closed loop to make it simple using the first move.}
\new{Up to reflection along the vertical axis}, the corresponding graph~$\Gs'$ is locally given as in Figure~\ref{fig:reduction_loop} (right), whose notation we assume. 
Since~$\Gs$ admits a flat isoradial immersion, so does~$\Gs'$ by Proposition~\ref{prop:Reidemeister}. By flatness at the face~$f$, the rhombus angles satisfy~$(\pi-\tilde{\theta}_{e_1})+(\pi-\tilde{\theta}_{e_2})=2\pi$. But this implies the equality~$\tilde{\theta}_{e_1}+\tilde{\theta}_{e_2}=0$, which contradicts the flatness at~$v$.

\begin{figure}[ht]
    \centering
    \def\svgwidth{6cm}
\begingroup%
  \makeatletter%
  \providecommand\color[2][]{%
    \errmessage{(Inkscape) Color is used for the text in Inkscape, but the package 'color.sty' is not loaded}%
    \renewcommand\color[2][]{}%
  }%
  \providecommand\transparent[1]{%
    \errmessage{(Inkscape) Transparency is used (non-zero) for the text in Inkscape, but the package 'transparent.sty' is not loaded}%
    \renewcommand\transparent[1]{}%
  }%
  \providecommand\rotatebox[2]{#2}%
  \newcommand*\fsize{\dimexpr\f@size pt\relax}%
  \newcommand*\lineheight[1]{\fontsize{\fsize}{#1\fsize}\selectfont}%
  \ifx\svgwidth\undefined%
    \setlength{\unitlength}{669.00228878bp}%
    \ifx\svgscale\undefined%
      \relax%
    \else%
      \setlength{\unitlength}{\unitlength * \real{\svgscale}}%
    \fi%
  \else%
    \setlength{\unitlength}{\svgwidth}%
  \fi%
  \global\let\svgwidth\undefined%
  \global\let\svgscale\undefined%
  \makeatother%
  \begin{picture}(1,0.32041245)%
    \lineheight{1}%
    \setlength\tabcolsep{0pt}%
    \put(0,0){\includegraphics[width=\unitlength,page=1]{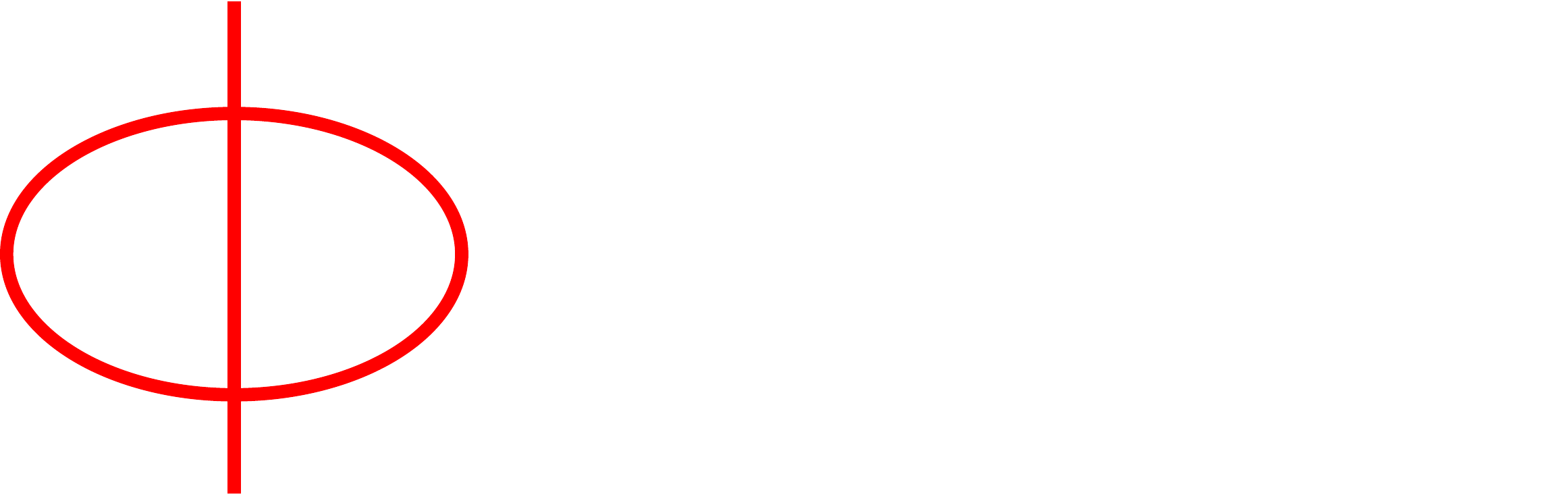}}%
    \put(0.92012385,0.14974488){\color[rgb]{0,0,0}\makebox(0,0)[lt]{\lineheight{1.25}\smash{\begin{tabular}[t]{l}$v$\end{tabular}}}}%
    \put(0,0){\includegraphics[width=\unitlength,page=2]{reduction_loop.pdf}}%
    \put(0.77258767,0.15355513){\color[rgb]{0,0,0}\makebox(0,0)[lt]{\lineheight{1.25}\smash{\begin{tabular}[t]{l}$f$\end{tabular}}}}%
    \put(0,0){\includegraphics[width=\unitlength,page=3]{reduction_loop.pdf}}%
    \put(0.74565915,0.28634246){\color[rgb]{0,0,0}\makebox(0,0)[lt]{\lineheight{1.25}\smash{\begin{tabular}[t]{l}$e_1$\end{tabular}}}}%
    \put(0.74792364,0.01057571){\color[rgb]{0,0,0}\makebox(0,0)[lt]{\lineheight{1.25}\smash{\begin{tabular}[t]{l}$e_2$\end{tabular}}}}%
  \end{picture}%
\endgroup%

    \caption{
    Proof of Theorem~\ref{thm:main}:  
    reduction in the case where a
    train-track of $\Gs$ makes a closed loop. Left: local train-track configuration. Right: corresponding neighborhood in the graph, \new{with~$v$ of degree~$2$ and the other vertex of degree at least~$2$}.}
    \label{fig:reduction_loop}
\end{figure}

Let us now assume that~$\Gs$ is a planar graph whose train-tracks do not form closed loops, and fix angles~$\alpha\in Z_\Gs$. We want to find~$k\in\ZZ^\Es$ defining a flat isoradial immersion of~$\Gs$. Note that using (the second move of) Proposition~\ref{prop:Reidemeister}, it can be assumed that~$\Gs^\T$ is connected.
Take a set of vertices~$M\subset\Vs^\T$ as in Lemma~\ref{lem:corner}.
\new{We now show that, for any choice of integers~$(k_e)_{e\in M}$, there exists a unique
completion~$k=(k_e)_{e\in \Vs^\T}$ defining a flat isoradial immersion of~$\Gs$.} 
By Lemma~\ref{lem:corner}, we have a basis~\new{$(c_e)_{e\in\Vs^\T\setminus M}$} of~$\mathcal{E}(\Gs^\T)$
made of closed curves 
such that~\new{each $e\in\Vs^\T\setminus M$ is a corner of~$c_e$} and all other corners \new{of~$c_e$} belong to~$M$.
By Lemma~\ref{lem:flat}, flatness inside the faces determines a unique value for~$k$ on the edges corresponding to these remaining vertices. To be more precise, the closed curve~\new{$c_e$} endowed with an arbitrary orientation can be written as~$c_e=\sum_\fs n_{e\fs}\partial\fs$ with~$n_{e\fs}\in\ZZ$, and the corresponding~$\ZZ$-linear combination of the flatness equations for the bounded faces~$\fs$ gives us this result. These oriented closed curves~$(c_e)$ form a basis of the (free abelian group of) cycles of~$\Gs^\T$. Since the boundary of any bounded face~$\fs$ is such a cycle, it can be expressed as~$\partial\fs=\sum_e m_{\fs e}c_e$ with~$m_{\fs e}\in\ZZ$. Therefore, the flatness equations for the~$c_e$'s, which we know are satisfied, imply the flatness equations for the faces, \emph{i.e.}, the flatness of the immersion.

Finally, consider a planar graph~$\Gs$ whose train-tracks do not form closed loops, and angles~$\alpha\in Z_\Gs$ together with~$k,k'\in\ZZ^\Es$ both defining flat isoradial immersions of~$\Gs$. Fix a set of vertices~$M\subset\Vs^\T$ as in Lemma~\ref{lem:corner}, and recall the equivalence relation of Definition~\ref{def:equiv}. Since we have an injective map~$\tau\colon M\to\T$ with~$e\in\tau(e)$ for all~$e\in M$, there exists~$k''\sim k$ such that~$k''$ and~$k'$ coincide on all edges corresponding to elements of~$M$. By Lemma~\ref{lem:corner}, any~$e\in\Vs^\T\setminus M$ is the corner of some closed curve~$c_{e}$ of~$\Gs^\T$ such that all corners of~$c_{e}$ but~$e$ belong to~$M$. By Lemma~\ref{lem:flat}, flatness implies that~$k''$ and~$k'$ coincide on this last corner~$e$ as well. Therefore, the family of integers~$k''$ and~$k'$ are equal, so~$k$ is equivalent to~$k'=k''$. $\hfill\square$

\subsection{Additional features}
\label{sub:rem}

This section contains additional features of isoradial immersions. First, restricting to the case where the graph $\Gs$ is finite, we translate results of Theorem~\ref{thm:main} into information on the linear system~\eqref{equ:flat_2} defining flatness; we also give an alternative explicit geometric construction of the solution. Then, we provide a geometric interpretation of folded rhombi of isoradial immersions using the infinite dihedral group. Last, we prove a discrete version of the Gauss--Bonnet formula. 

\bigskip

\textbf{Flatness condition revisited.} In this section and the next, we suppose that the graph $\Gs$ is finite and that $\Gs^\T$ does not contain closed loops. Recall the linear system~\eqref{equ:flat_2} defining flatness :
\begin{align*}
\sum_{e\sim v} k_e=\textstyle{1-\frac{1}{2\pi}s_v}\quad\text{and}\quad -\displaystyle{\sum_{e^* \sim f} k_e} =\textstyle{1-\frac{1}{2\pi}s_f},
\end{align*}
where $s_v=\sum_{e\sim v}\theta_e$ and~$s_f=\sum_{e^*\sim f}(\pi-\theta_e)$. Then, the matrix $\Ms$ corresponding to this linear system has rows indexed by faces of $\Gs^\T$, columns indexed by vertices of $\Gs^\T$, or equivalently by edges of $\Gs$; the coefficient $\Ms_{\fs,e}$ is non zero iff $e$ is on the boundary of the face $\fs$; when this is the case, it is equal to $+1$, resp. $-1$, if $\fs$ corresponds to an inner vertex, resp. an inner face, of $\Gs$.

In order to state the translation of Theorem~\ref{thm:main} to the linear system, we need one more notation. 
Given two integer families~$k,k'\in\ZZ^{\Vs^\T}$ related by a shift along a train-track~$t$ (recall Definition~\ref{def:equiv}), we let~$k^t\in\ZZ^{\Es}$ denote their difference (with an arbitrary sign). 
Then, rephrasing Theorem~\ref{thm:main} in this context yields:
\begin{enumerate}
 \item The rank of~$\Ms$ is equal to~$|\Fs^\T|$ (so the associated map is onto).
 \item The kernel of~$\Ms$ has dimension~$|\T|-1$, and is generated by~$\{ k^t:\, t\in\T\}$.
\end{enumerate}

The penultimate paragraph of the proof of Theorem~\ref{thm:main}, which uses
Lemma~\ref{lem:corner} and Lemma~\ref{lem:flat}, can be seen in this context as a
way to explicitly trigonalize the linear system defined by the matrix $\Ms$.

\bigskip

\textbf{Geometric construction of solutions}. We now propose an alternative geometric construction of a solution to the linear system~\eqref{equ:flat_2}. For every inner vertex $v$ of $\Gs$, let $\delta_v$ be the function defined on faces of $\Gs^\T$, taking value 1 at the vertex $v$, and 0 elsewhere; the function $\delta_f$ for an inner face $f$ of $\Gs$ is defined similarly. 
We explicitly construct functions $k^v$ and $k^f$ such that $\Ms k^v=\delta_v$ and $\Ms k^f=\delta_f$, thus proving surjectivity of the map associated to $\Ms$.
Without loss of generality, let us restrict to an inner vertex $v$ of $\Gs$ since it will be clear from the proof how to proceed for a face $f$. We use three building blocks.

For every vertex $e$ of $\Gs^\T$ consider the function
$\delta_e\in\CC^{\Vs^\T}$, taking value 1 at $e$ and 0 elsewhere. Then, if $e$
is not the end point of a loop edge of $\Gs^\T$, it is on the boundary of two
faces corresponding to primal vertices and two faces corresponding to dual ones.
In this case, $\Ms \delta_e$ is equal to 1 on the two vertices, $-1$ on the two
faces and 0 elsewhere, see Figure~\ref{fig:proof_alt_1} (first figure). If the
vertex $e$ is the end point of a loop edge, it is a self-intersection point of a
single train-track, and the vertex $e$ is on the boundary of two faces
corresponding to primal vertices and one corresponding to a dual one (or the
reverse). Then $\Ms \delta_e$ takes value 1 at the two primal vertices and $-2$ at the dual one (or the reverse), see Figure~\ref{fig:proof_alt_1} (second figure). 
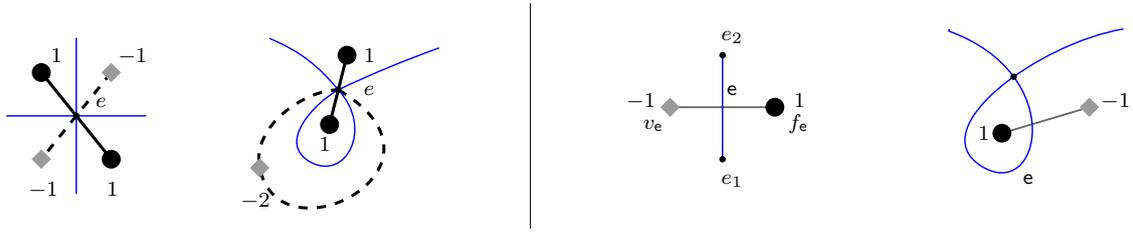
\begin{figure}[ht]
\centering
\begin{overpic}[width=\linewidth]{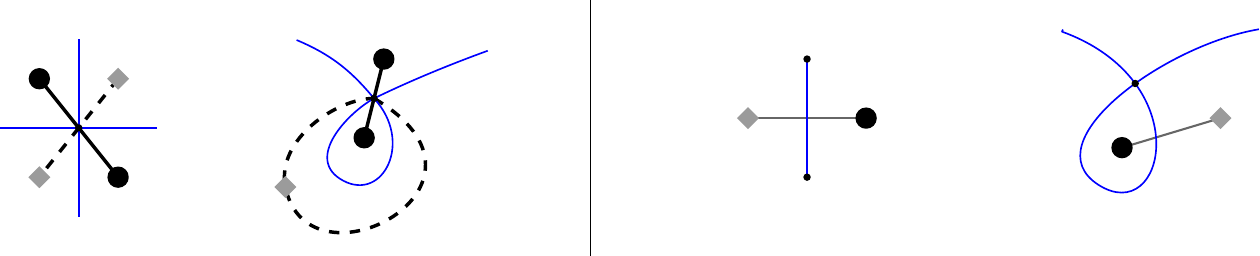}
    \put(4,15){\scriptsize $1$}
    \put(9,3){\scriptsize $1$}
    \put(10,15){\scriptsize $-1$}
    \put(2,3){\scriptsize $-1$}
    \put(8,11){\scriptsize $e$}
    \put(32,15){\scriptsize $1$}
    \put(28,7){\scriptsize $1$}
    \put(21,2){\scriptsize $-2$}
    \put(32,12){\scriptsize $e$}
    \put(64.5,12){\scriptsize $\es$}
    \put(70.5,11){\scriptsize $1$}
    \put(55.5,11){\scriptsize $-1$}
    \put(64,4){\scriptsize $e_1$}
    \put(64,17){\scriptsize $e_2$}
    \put(57,9){\scriptsize $v_\es$}
    \put(70,9){\scriptsize $f_\es$}
    
    \put(91,4){\scriptsize $\es$}
    \put(87,8){\scriptsize $1$}
    \put(98,11){\scriptsize $-1$}
\end{overpic}
\caption{Left: the image $\Ms\delta_e$ around the vertex $e$ of $\Gs^\T$ when
$e$ is not the end-point of a loop edge (first picture) and when it is (second
picture). Right: the image $\Ms k^\es$ when $\es$ is a simple edge of $\Gs^\T$ (third picture) or a loop edge (fourth picture).}
\label{fig:proof_alt_1}
\end{figure}

Consider an edge $\es$ of $\Gs^\T$
(where the notation $\es$ for an edge of $\Gs^\T$ should not be confused with
the notation $e$ for a vertex of $\Gs^\T$),
then $\es$ belongs to a unique train-track $t$ and bounds two faces of $\Gs^\T$ corresponding to a vertex $v_\es$ and a face $f_\es$. From the functions $(\delta_e)$, we construct a function $k^\es$ such that $\Ms k^\es$
takes value 1 at $v_\es$ and $-1$ at $f_\es$. Let $(e_1,e_2)$ denote the edge $\es$ oriented so that $v_\es$ is on the left. Consider the path in the train-track $t$ running from $e_1$ to the boundary without crossing $e_2$, and define $k^\es$ to be the alternate sum of the functions $(\delta_e)$ running over the encountered vertices along the path, starting with a $+1$; note that one could equivalently take a path from $e_2$ to the boundary of the graph not crossing $e_1$. Then, since two consecutive vertices share one face corresponding to a vertex and one to a face of $\Gs$ (this is because the graph $\Gs^\T$ has degree 4 vertices), we have that $\Ms k^{\es}$ is equal to 1 at $v_\es$, $-1$ at $f_\es$ and 0 at all other inner faces of $\Gs^\T$. Note that if $\es$ is a loop, all of the above makes sense with minor modifications.

Consider an inner vertex $v$ of $\Gs$. From the functions $(\delta_\es)$ we now construct a function $k^v$ such that $\Ms k^v=\delta_v$.
Consider a simple path in the quad-graph $\GR$ from $v$ to a boundary vertex, and the dual edges (which belong to $\Gs^{\T}$). Define $k^v$ to be the alternate sum of the functions $k^\es$ crossing this path, starting from a +1. Then, since two consecutive edges bound a common face of $\Gs^\T$ corresponding to a primal or a dual vertex of $\Gs$, we indeed have that $\Ms k^v=\delta_v$ except on the boundary vertex or face, which does not enter the linear system. A graphical representation of $k^v$ is given in Figure~\ref{fig:proof_alt_2} below.

As an immediate consequence, we recover that the rank of $\Ms$ is equal to $|\Fs^\T|$. 

\begin{figure}[ht]
\centering
\begin{overpic}[width=7cm]{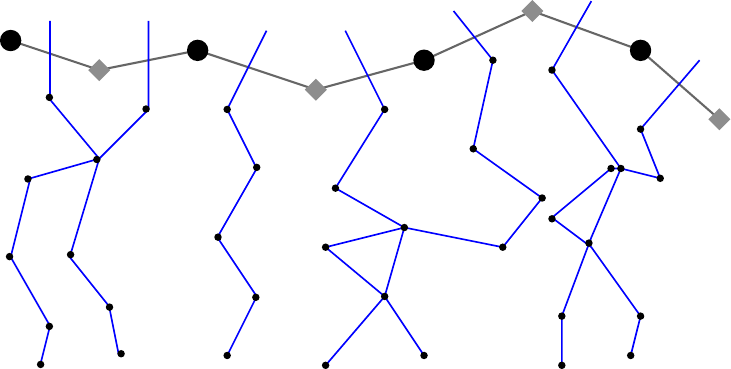}
    \put(-3,43){\scriptsize $v$}
    \put(100,32){\scriptsize boundary vertex of $\GR$}
\end{overpic}
\caption{Graphical representation of the solution $k^v$: the marked vertices of $\Gs^\T$ are those where the function $k^v$ is possibly non-zero.}
\label{fig:proof_alt_2}
\end{figure}

\begin{rem}
This concrete way of solving the linear system has a nice geometric interpretation, best explained when the graph $\Gs$ is isoradial, so let us assume that this is the case. Consider angles $\alpha\in Z_\Gs$, then a natural question is to understand how the solution $k$ varies as one of the angles $\alpha(\tr_j)$ moves on the other side of a neighboring angle $\alpha(\tr_i)$. This local transformation is generic in the sense that any pair $\alpha,\alpha'$ of angles in $Z_\Gs$ can be obtained from each other by a sequence of such local moves. This move amounts to exchanging the angles of the directed train-tracks $\tr_i,\tr_j$ crossing at a vertex $e$ in the graph $\Gs^\T$, and this vertex is incident to two faces corresponding to vertices of $\Gs$ and two faces corresponding to faces of $\Gs^*$. Since the angles $\alpha(\tr_i),\alpha(\tr_j)$ are adjacent, this amounts to increasing the value $\frac{1}{2\pi}s_v$ by $1$ at the two primal vertices and decreasing the value $\frac{1}{2\pi}s_f$ by $-1$ at the two dual faces, or the opposite. 
From the first part of the proof involving the function $\delta_e$, we deduce that this local move on the angles has the effect of increasing or decreasing $k$ by 1 at the vertex $e$.

The next relevant geometric question is: given that the quantity $\frac{1}{2\pi}s_v$ increases or decreases by 1 at a given inner vertex $v$ or face $f$, how does the solution $k$ vary? What we prove is that the solution gets modified by adding or subtracting 1 along a family of paths of $\Gs^\T$ to the boundary starting from vertices of a dual path essentially joining the vertex in question to the boundary.
\end{rem}

\bigskip

\textbf{Folded rhombi and the infinite dihedral group.} In our definition of isoradial immersions, we have adopted a combinatorial viewpoint. Recall, see Remark~\ref{rem:def_iso_imm}, that an isoradial immersion is equivalent to the data of an angle map~$\alpha\in Z_\Gs$ and of a set of integers~$k=(k_e)_e\in\ZZ^\Es$.
There is a more geometric viewpoint on these choices of lifts, that we now present. As already mentioned, a rhombus with angle~$\tilde\theta_e=\theta_e\in(0,\pi)$ should be thought of as embedded in the plane: let us denote this \emph{rhombus state} by~$1$. When the angle~$\tilde\theta_e$ increases and crosses the value~$\pi$, the rhombus opens more and more until it folds along the (primal) edge~$e$: let us denote this state by~$p$ (for primal). If the angle~$\tilde\theta_e$ continues to grow and crosses the value~$2\pi$, the rhombus shrinks until it folds again, but this time, along the dual edge~$e^*$: let us denote this new state by~$pd$ (for primal-dual). Continuing in this way, we see that each angle~$\tilde\theta_e>0$ determines a state of the form~$pdp\cdots$. In the same way, when the angle~$\tilde\theta_e\in(0,\pi)$ of an embedded rhombus decreases and becomes negative, the rhombus shrinks until it folds along the dual edge~$e^*$, a state denoted by~$d$. Decreasing~$\tilde\theta_e$ further leads to the rhombus folding again, leading to states of the form~$dpd\cdots$. Note that folding a rhombus twice along the same (primal or dual) edge does not change its state, a fact denoted by~$p^2=d^2=1$.

In a nutshell, the set of states for a given rhombus form the \emph{infinite dihedral group}~$D_\infty$, best understood and presented in the current situation as the free product of two copies of~$\ZZ/2\ZZ$:
\[
D_\infty=\ZZ/2\ZZ*\ZZ/2\ZZ=\left<p,d\,|\,p^2=d^2=1\right>\,.
\]  
The possible deformations of a given rhombus fit nicely in the Cayley graph of
this group, with respect to the generators~$\{p,d\}$, as illustrated in
Figure~\ref{fig:rhombus_states}.

\begin{figure}
  \centering
  \def\svgwidth{\textwidth}
  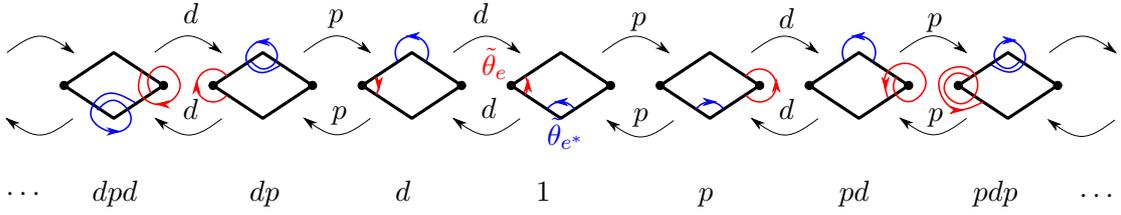
  \caption{The possible states of a fixed rhombus, as the vertices of the Cayley graph of $D_\infty=\langle p,d|p^2=d^2=1\rangle$, with the corresponding rhombus angles $\tilde{\theta}_e$ and $\tilde{\theta}_{e^*}=\pi-\tilde{\theta}_e$.}
  \label{fig:rhombus_states}
\end{figure}

There is more structure on the group~$D_\infty$, that sheds an interesting light on isoradial immersions, and on the proof of Theorem~\ref{thm:main}.

First, there is a natural \emph{orientation homomorphism}~$\epsilon\colon D_\infty\to\ZZ/2\ZZ$ mapping both generators~$p$ and~$d$ to the non-trivial element of~$\ZZ/2\ZZ$. A rhombus is \emph{positively oriented} if its state lies in the kernel of~$\epsilon$, and \emph{negatively oriented} otherwise. The group~$D_\infty^+=\operatorname{ker}(\epsilon)$ of positively oriented states is isomorphic to~$\ZZ$ (generated by~$pd$ or~$dp=(pd)^{-1}$), while~$D_\infty\setminus D_\infty^+\simeq pD_\infty^+$ is also in canonical bijection with~$\ZZ$. This leads to the integer~$k_e$ of the combinatorial viewpoint of Section~\ref{sub:def}.

Also, the group~$D_\infty$ is endowed with a natural involution determined by~$p\mapsto d$ and~$d\mapsto p$. It commutes with the orientation homomorphism, and therefore induces an involution on~$D_\infty^+$, which is nothing but~$k\mapsto -k$ via the isomorphism~$D_\infty^+\simeq\ZZ$. This involution corresponds geometrically to exchanging the primal and dual edges of the rhombi. In other words, it is the algebraic counterpart of the geometric duality on the graph~$\Gs$, exhibiting the fact that isoradial immersions are self-dual, \emph{i.e.}, coherent with duality, see Point 4. of Remark~\ref{rem:def_iso_imm}.

Therefore, instead of defining an isoradial immersion as coming from an angle map~$\alpha\in Z_\Gs$ and integers~$k\in\ZZ^\Es$, we could just as well have defined this object as coming from an~$\alpha\in Z_\Gs$ together with some states in~$(D_\infty)^\Es$ with orientations determined by~$\alpha$. As the reader will easily check, the whole of Section~\ref{sec:flat} can be rewritten in this way, using the duality involution on~$D_\infty$ instead of~$k\mapsto -k$ (see in particular the third part of the proof of Proposition~\ref{prop:Reidemeister}).

\bigskip

\textbf{Discrete Gauss--Bonnet formula.} 
Consider a planar graph $\Gs$ such that $\Gs^\T$ has no closed loops, and  
a simple, connected cycle $c$ of $\Gs^\T$. By adding the
left-hand side of the flatness equations~\eqref{equ:flat} for all faces inside~$c$, similarly to what we have done in the beginning of the proof of Lemma~\ref{lem:flat} for oriented closed curves, we obtain a discrete version of the Gauss--Bonnet formula. 

Recall the definition of corners given in Section~\ref{sub:cycles}. In order to state our result, we need the following notation.
Denote by $\Fs^\T_c$ the set of faces of $\Gs^\T$
bounded by~$c$. Recall that faces of $\Gs^\T$ correspond to inner vertices and faces of $\Gs$; denote by $\Vs_c$, resp. $\Vs_c^*$, the set of vertices, resp. faces of $\Gs$, whose corresponding faces are in $\Fs^\T_c$. Note that each corner vertex $e$ of $c$ is on the boundary of exactly one or three faces of $\Fs^\T_c$;
in the first case, we shall write~$\sgn(e)=+1$ and in the second case,~$\sgn(e)=-1$.
Finally, we shall write~$e\sim v$ (resp.~$e\sim f$) if among the (one or three) faces in $\Fs^\T_c$ bounded by the corner vertex~$e$, exactly one corresponds to a vertex (resp. a face) of~$\Gs$.

\begin{prop}[Discrete Gauss--Bonnet formula]\label{lem:lem_sum}
Fix angles $\alpha\in Z_\Gs$,
integers~$k=(k_e)_{e\in\Es}$,
and consider the corresponding isoradial immersion. Let~$c$ be a simple, connected cycle in~$\Gs^\T$.
Then, we have the equality 
\begin{equation}\label{equ:lem_sum}
\sum_{v\in \Vs_c}(2\pi-\tilde{s}_v)+\sum_{f\in \Vs_c^*}(2\pi-\tilde{s}_f)+
\sum_{e \text{ corner},\, e\sim v}\!\!\!\sgn(e)\,\tilde{\theta}_e+\sum_{e \text{ corner},\, e\sim f}\!\!\!\sgn(e)\,\tilde{\theta}_{e^*}=2\pi\,,
\end{equation}
where $\tilde{s}_v=\sum_{e\sim v}\tilde{\theta}_e$ and~$\tilde{s}_f=\sum_{e\sim f}\tilde{\theta}_{e^*}$.
\end{prop}

\begin{proof}
Throughout this proof, we work in the subgraph~$\Gs_c^\T$
of~$\Gs^\T$ bounded by the simple cycle~$c$.
There are three types of vertices in this subgraph:
vertices of degree 2 in~$\Gs_c^\T$ (which are corners bounding exactly one face), vertices of degree 3 (which are non-corner boundary vertices), and vertices of degree 4 (which can be either inner vertices, or corners bounding three faces).
We denote these set of vertices of~$\Gs_c^\T$ by $\Vs_2,\Vs_3,\Vs_4$ respectively.

When summing the left-hand side of equations~\eqref{equ:flat} defining flatness on all faces of~$\Gs_c^\T$, we have the following contributions:
\begin{itemize}
 \item every (corner) vertex $e$ of degree~$2$ is incident to a single
face of $\Gs^\T$; it contributes~$\tilde{\theta}_e$, resp.\ $\tilde{\theta}_{e^*}$
to the sum if this face corresponds to a vertex, resp.\ a face of $\Gs$.
 \item every vertex $e$ of degree~$3$ contributes~$\tilde{\theta}_e+\tilde{\theta}_{e^*}=\pi$.
 \item every inner vertex $e$ (of degree 4) contributes $\tilde{\theta}_e+\tilde{\theta}_{e^*}+\tilde{\theta}_e+\tilde{\theta}_{e^*}=2\pi$.
 \item every corner vertex $e$ of degree~$4$ is incident to three
faces of $\Gs^\T$; it contributes~$\tilde{\theta}_{e^*}+\tilde{\theta}_e+\tilde{\theta}_{e^*}=2\pi-\tilde{\theta}_e$ if only one of these faces corresponds to a vertex of $\Gs$, and~$2\pi-\tilde{\theta}_{e^*}$ in the other case.
\end{itemize}
As a consequence, we have the equation:
\[
\sum_{v\in \Vs_c}\tilde{s}_v +\sum_{f\in \Vs_c^*}\tilde{s}_f = \pi |\Vs_3|+2\pi |\Vs_4|
+\sum_{e \text{ corner},\, e\sim v}\!\!\!\sgn(e)\,\tilde{\theta}_e +\sum_{e \text{ corner},\, e\sim f}\!\!\!\sgn(e)\,\tilde{\theta}_{e^*}\,.
\]
Since $|\Vs_c|+|\Vs_c^*|=|\Fs^\T_c|$, we are left with showing that~$|\Vs_3|+2|\Vs_4|=2|\Fs^\T_c|-2$.
This follows from the three equalities~$|\Vs_2|+|\Vs_3|+|\Vs_4|=|\Vs^\T_c|$,
~$2|\Vs_2|+3|\Vs_3|+4|\Vs_4|=2|\Es^\T_c|$ and~$2|\Vs_2|+|\Vs_3|=2|\T_c|$
(which already appeared in the proof of Lemma~\ref{lem:corner}), together with the Euler
characteristic computation~$|\Vs^\T_c|-|\Es^\T_c|+|\Fs^\T_c|=1$. This concludes the proof.
\end{proof}

\begin{rem}\label{rem:lem_sum}$\,$
\begin{itemize}
\item If the isoradial immersion is assumed to be flat, then by definition, we have~$\tilde{s}_v=\tilde{s}_f=2\pi$ for every vertex $v$ and every face $f$ of $\Gs$. Therefore, Equation~\eqref{equ:lem_sum} simply becomes~$
\sum_{e\sim v}\sgn(e)\,\tilde{\theta}_e +\sum_{e\sim f}\sgn(e)\,\tilde{\theta}_{e^*}=2\pi$, the sum being over corner vertices.
\item If $k\equiv 0$, or equivalently if we consider the rhombic immersion of $\GR$ corresponding to $\alpha\in Z_\Gs$, then the lifted rhombus angles are in $[0,2\pi)$, and the same equation as~\eqref{equ:lem_sum} holds without the ``tilde''. In particular, in the flat case, we get
\begin{equation}\label{equ:lem_sum_1}
\sum_{e \text{ corner},\, e\sim v}\!\!\!\sgn(e)\,\theta_e +\sum_{e \text{ vertex},\, e\sim f}\!\!\!\sgn(e)\,\theta_{e^*}=2\pi\,.
\end{equation}
It is this special case of our Gauss--Bonnet formula that will prove useful.
\end{itemize}
\end{rem}

\section{Minimal immersions of bipartite graphs}
\label{sec:minimal}

As already mentioned, Theorem~\ref{thm:main} is in the same spirit as the main result of~\cite{KeSchlenk} but concerns more general objects, namely flat isoradial immersions of graphs without train-track loops. We now obtain an analogous result for an intermediate class, namely minimal immersions of minimal bipartite graphs, and give an application to the study of the dimer model on such graphs. The main definitions and statements are given in Section~\ref{sub:min}, and the proofs in Section~\ref{sub:proofmin}. Section~\ref{sub:condition} contains our application to dimers.

In all of this section, we suppose that the planar graph~$\Gs$ is such that~$\Gs^\T$ does not contain closed loops. For simplicity, we also assume that~$\Gs$ does not contain any unbounded faces.

\subsection{Minimal immersions}
\label{sub:min}

Recall from Section~\ref{sub:def} that given a planar embedded graph~$\Gs$, an
isoradial immersion is built from angles~$\alpha\in Z_\Gs$ and
integers~$k\in\ZZ^\Es$; an isoradial immersion is flat if the corresponding
total angle at each
face and vertex is equal to~$2\pi$.

We now assume that~$\Gs$ is bipartite (recall Section~\ref{sub:ttmin}). Using the consistent orientation of train-tracks, the space~$Z_\Gs$ can be identified with the space of angle maps
\[
\alpha\colon\Tbip\to\RR/2\pi\ZZ.
\]
We focus on \new{\emph{flat}} isoradial immersions given by such an~$\alpha\in Z_\Gs$, but with~$k$ identically zero
and non-degenerate rhombi. In other words:

\begin{defi}\label{def:min}
A flat isoradial immersion of a planar, embedded bipartite graph~$\Gs$ is called
a \emph{minimal immersion} of~$\Gs$ if all rhombus angles at vertices of~$\Gs$ belong to~$(0,2\pi)$.
\end{defi}

We give a concrete geometric characterization of minimal immersion in Proposition~\ref{prop:min} below, but let us first progress towards stating the main result of this section.

Note that an isoradial immersion with~$k=0$ is fully determined by
the angle map~$\alpha\in Z_\Gs=\{\Tbip\to\RR/2\pi\ZZ\}$. However, because of the flatness
and non-degeneracy conditions, not every~$\alpha$ gives a 
minimal immersion, and not every graph actually admits such an immersion.
Our main result characterizes graphs admitting a minimal immersion, and for such a graph~$\Gs$, the subspace of~$Z_\Gs$ given by the angles which define such an immersion.

To state this theorem, we need two additional notation. Recall the global cyclic order on~$\Tbip$ given in Definition~\ref{def:order}, and for each vertex~$v$ and face~$f$ of~$\Gs$, the local cyclic orders on the sets~$\Tbip(v)$ and~$\Tbip^\bullet(f),\Tbip^\circ(f)$ defined in Section~\ref{sub:ttmin}.

Recall also that a map~$\alpha\colon T\to\RR/2\pi\ZZ$ on a cyclically ordered set~$T$ is called \emph{monotone} if it respects the cyclic order on~$T$ and the natural cyclic order on~$\RR/2\pi\ZZ$, \emph{i.e.}, whenever we have~$[t_1,t_2,t_3]$ in~$T$, then we have~$[\alpha(t_1),\alpha(t_2),\alpha(t_3)]$ in~$\RR/2\pi\ZZ$. Note that in general we do not ask for strict monotonicity.

\begin{defi}
\label{def:X-Y}
Let~$X_\Gs\subset Z_\Gs$ denote the set of monotone maps~$\alpha\colon\Tbip\to\RR/2\pi\ZZ$ mapping
pairs of intersecting or anti-parallel train-tracks to distinct angles.
Similarly, we denote by~$Y_\Gs$ the subset of~$Z_\Gs$ given by the angle maps whose restriction to~$\Tbip(v)$ is monotone, injective and non-constant for all
vertices~$v$, and whose restrictions to~$\Tbip^\bullet(f)$ and~$\Tbip^\circ(f)$ are monotone and non-constant for all faces~$f$ of~$\Gs$.
\end{defi}

A few remarks are in order. First note that for any graph~$\Gs$, the set~$X_\Gs$ is clearly non-empty. Next, as we will prove below, Lemmas~\ref{lem:vertex_1} and~\ref{lem:face_1} imply that if~$\Gs$ is minimal, then we have the inclusion~$X_\Gs\subset Y_\Gs$. Since $X_\Gs$ is non-empty, this implies that $Y_\Gs$ is also non-empty for~$\Gs$ minimal.

We will actually show that~$Y_\Gs$ is non-empty if and only if~$\Gs$ is minimal.
This is a direct consequence of the main result of this section, that we are now ready to state.

\begin{thm}
\label{thm:min}
An embedded, planar, bipartite graph~$\Gs$ admits a minimal immersion if and only if~$\Gs$ is minimal. In any case, the space of immersions is given by~$Y_\Gs$.
\end{thm}

The proof of Theorem~\ref{thm:min} is postponed to the next section.

Since the space of minimal immersions of a graph is given by~$Y_\Gs$, which contains~$X_\Gs$
if~$\Gs$ is minimal, it is natural to wonder whether these two spaces coincide for minimal graphs.
This turns out to be the case in the~$\ZZ^2$-periodic case, see Corollary~\ref{cor:X=Y} below.

\medskip

The following proposition gives the promised geometric characterization of minimal immersions. Prior to stating it, let us make several remarks.

\begin{rem}\label{rem:def_min}$\,$
\begin{enumerate}
\item By Point 2. of Remark~\ref{rem:def_iso_imm}, boundary vertices of
  primal/dual faces corresponding to dual/primal vertices of $\Gs$ are mapped by
  the rhombic immersion to points on a circle of unit radius.
These points on the unit circle are endowed with a cyclic order, and it makes sense to compare the cyclic order of these vertices
 in the planar embeddings of $\Gs,\Gs^*$
 and the cyclic order of their images in the rhombic immersion of $\GR$. We use these two orders in Proposition~\ref{prop:min} below.

\item If a graph has a vertex of degree~$1$,
 then one easily checks that it does not admit any minimal immersion.
 Therefore, in stating the following geometric description of minimal immersions, we can assume that the graph does not have such a vertex.
\end{enumerate}
\end{rem}

\begin{prop}\label{prop:min}
Let~$\Gs$ be a planar, embedded, bipartite graph without degree~$1$ vertices.
Then, a minimal immersion of~$\Gs$ is equivalent to a rhombic immersion of~$\GR$ satisfying the following conditions:
\begin{enumerate}
 \item For every vertex of~$\Gs$, the boundary vertices of the corresponding face of~$\Gs^*$ are
 all mapped to distinct points,
 and the cyclic orders on these points coincide in the embedding of~$\Gs^*$ and in the rhombic immersion; in particular there is at most one folded rhombus around the vertex.
\item For every face of~$\Gs$, the white (resp. black) boundary vertices are not all mapped to the same point,
and the cyclic orders on the boundary vertices in the embedding of~$\Gs$ and in the rhombic immersion  differ exactly by elementary transpositions given by the boundary edges corresponding to the folded rhombi (in the case of a degree 4 face, there is at most one such edge).
\end{enumerate}
\end{prop}

The second point is illustrated in the left side of Figure~\ref{fig:typical_face_with_folded}.
The proof of this proposition is given in Section~\ref{sub:proofmin}, but let us already
state several consequences.

First, folded rhombi in minimal immersions are isolated, in the sense
that they meet at most along dual vertices.
The right side of Figure~\ref{fig:typical_face_with_folded} shows a typical such folded rhombus,
together with adjacent
rhombi.

Also, note that the first condition implies that a minimal immersion of~$\Gs$ yields a
local embedding of its dual graph~$\Gs^*$
(with the slight abuse that degree 2 faces collapse to a segment) so that each face of $\Gs^*$ is inscribed in a circle of radius~$1$.
However, this is not an isoradial embedding of~$\Gs^*$ in general, as the circumcenter of a face might not belong to the interior of that face.

\begin{figure}[ht]
  \begin{subfigure}[b]{.58\linewidth}
    \centering
    \def\svgwidth{8cm}
\begingroup%
  \makeatletter%
  \providecommand\color[2][]{%
    \errmessage{(Inkscape) Color is used for the text in Inkscape, but the package 'color.sty' is not loaded}%
    \renewcommand\color[2][]{}%
  }%
  \providecommand\transparent[1]{%
    \errmessage{(Inkscape) Transparency is used (non-zero) for the text in Inkscape, but the package 'transparent.sty' is not loaded}%
    \renewcommand\transparent[1]{}%
  }%
  \providecommand\rotatebox[2]{#2}%
  \newcommand*\fsize{\dimexpr\f@size pt\relax}%
  \newcommand*\lineheight[1]{\fontsize{\fsize}{#1\fsize}\selectfont}%
  \ifx\svgwidth\undefined%
    \setlength{\unitlength}{668.42794056bp}%
    \ifx\svgscale\undefined%
      \relax%
    \else%
      \setlength{\unitlength}{\unitlength * \real{\svgscale}}%
    \fi%
  \else%
    \setlength{\unitlength}{\svgwidth}%
  \fi%
  \global\let\svgwidth\undefined%
  \global\let\svgscale\undefined%
  \makeatother%
  \begin{picture}(1,0.3713086)%
    \lineheight{1}%
    \setlength\tabcolsep{0pt}%
    \put(0,0){\includegraphics[width=\unitlength,page=1]{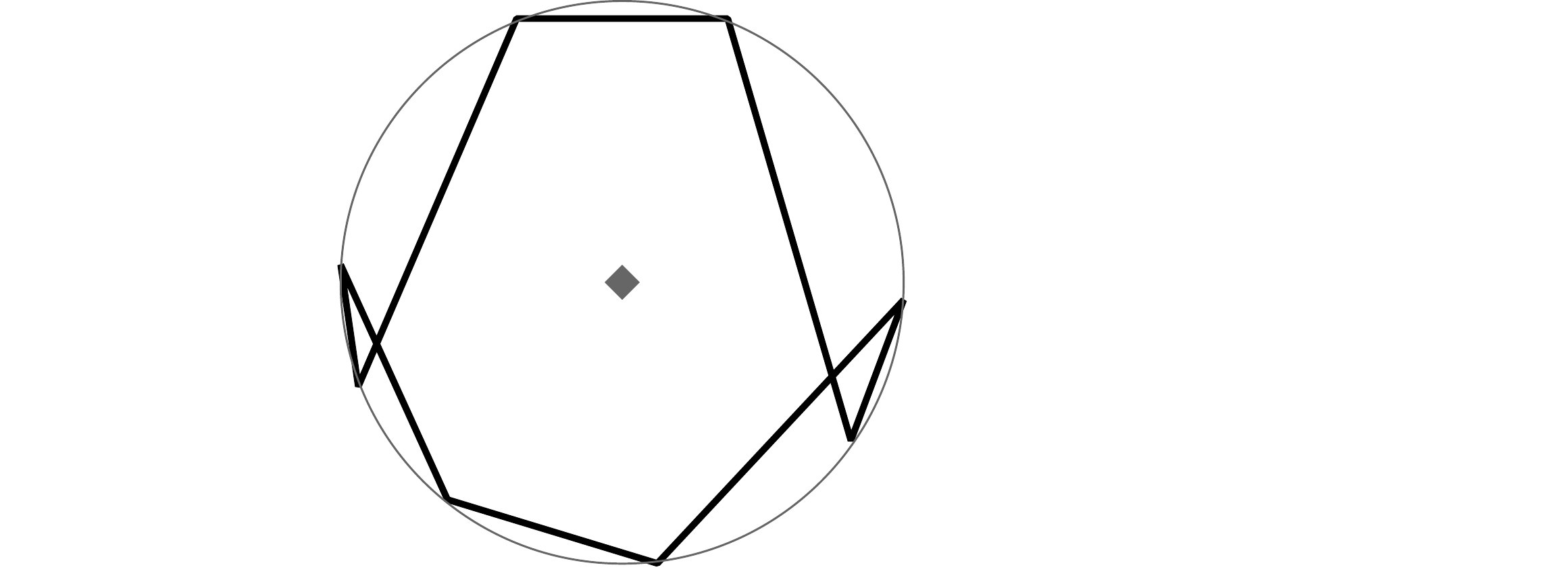}}%
    \put(0.41925257,0.19114678){\color[rgb]{0,0,0}\makebox(0,0)[lt]{\lineheight{1.25}\smash{\begin{tabular}[t]{l}$f$\end{tabular}}}}%
    \put(0,0){\includegraphics[width=\unitlength,page=2]{typical_face_iso_immersion.pdf}}%
    \put(-0.00102999,0.17800278){\color[rgb]{0,0,0}\makebox(0,0)[lt]{\lineheight{1.25}\smash{\begin{tabular}[t]{l}folded$\searrow$\end{tabular}}}}%
    \put(0.6044386,0.15896178){\color[rgb]{0,0,0}\makebox(0,0)[lt]{\lineheight{1.25}\smash{\begin{tabular}[t]{l}$\swarrow$folded\end{tabular}}}}%
  \end{picture}%
\endgroup%

  \end{subfigure}
  \begin{subfigure}[b]{.38\linewidth}
    \centering
    \def\svgwidth{3cm}
    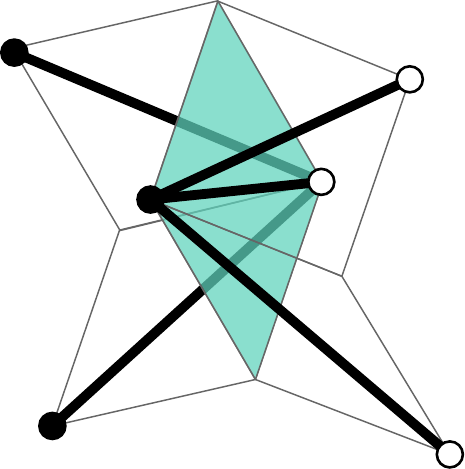
  \end{subfigure}
  \caption{Left: a typical minimal immersion around a face. Right: a folded rhombus in a minimal immersion is isolated.} 
  \label{fig:typical_face_with_folded}
\end{figure}

More technically, a minimal immersion of~$\Gs$ defines a flat metric with cone singularities
on the surface~$S$ given by the union of all (bounded) faces of~$\Gs$, together with a continuous map~$S\to\RR^2$ which is a local isometry away from the ``fold lines'' given by the boundary of the folded rhombi.
The angle at each cone singularity in~$S$ can be different from~$2\pi$ (if it lies on such a fold line),
but the average of these angle around any folded rhombus is equal to~$2\pi$.

\subsection{Proof of Theorem~\ref{thm:min}, of Proposition~\ref{prop:min}, and consequences}
\label{sub:proofmin}

We begin by identifying the set of maps~$\alpha\in Z_\Gs$ defining a flat isoradial immersion of~$\Gs$.

\begin{lem}
\label{lem:minflat}
If~$\Gs$ is a bipartite, planar graph, then the isoradial immersion of~$\Gs$ defined by~$\alpha\in Z_\Gs$
and~$k\equiv 0$ is a minimal immersion if and only if~$\alpha$ belongs to~$Y_\Gs$.
\end{lem}
\begin{proof}
Recall that the isoradial immersion of $\Gs$ with $k\equiv 0$ is minimal if and only if
all the rhombus angles belong to~$(0,2\pi)$, and the sum of the rhombus angles around every vertex and face of $\Gs$ is equal to $2\pi$.
Let us first consider a vertex~$v$ of~$\Gs$. If this vertex is of degree~$1$, then the isoradial immersion of~$\Gs$ defined by any~$\alpha\in Z_\Gs$ has angle~$0$ at~$v$, so no angle map~$\alpha$ defines a flat immersion. On the other hand, the set~$\Tbip(v)$ has a single element so the restriction of any~$\alpha\in Z_\Gs$ to~$\Tbip(v)$ is constant; it follows that~$Y_\Gs$ is empty. Hence, the lemma holds in this degenerate case, and it can be assumed that~$v$ has degree~$n\geq 2$.
Let us number the adjacent train-tracks
strands~$\Tbip(v)$ and adjacent edges as
described in Figure~\ref{fig:t_bip_v} (left). Then, the total angle around~$v$ is equal
to~$\sum_{j=1}^n\theta_{e_j}$ where by definition, the rhombus
angle~$\theta_{e_j}$ is the unique lift in~$[0,2\pi)$ of the
angle~$[\alpha(\tr_{j+1})-\alpha(\tr_j)]\in\RR/2\pi\ZZ$. Therefore, the total
angle around~$v$ is equal to~$2\pi$ with all rhombus angles in~$(0,2\pi)$ 
if and only if the restriction of~$\alpha$
to~$\Tbip(v)$ is monotone and injective (hence non-constant since~$n>1$).
This is the first condition for~$\alpha$ belonging to the set~$Y_\Gs$. 

Let us now assume that this condition is satisfied around each vertex~$v$ of~$\Gs$
(which implies that all the rhombus angles lie in~$(0,2\pi)$),
and consider a face~$f$.
If this face is of degree~$2$, then any isoradial immersion of~$\Gs$ has a cone angle~$0$ at~$f$
(so the flatness condition is not satisfied)
unless the two corresponding train-tracks are assigned identical angles,
in which case we have two degenerate rhombi; in any case, this is not a minimal immersion.
Also, both sets~$\Tbip^\bullet(f)$ and~$\Tbip^\circ(f)$ have a single element, so~$Y_\Gs$ is empty. Hence, the statement holds in this case, and it can be assumed that~$f$ has degree~$2m$ with~$m\ge 2$. Numbering the adjacent train-track strands~$\Tbip(f)$ as in Figure~\ref{fig:tt_around_face}, the total angle around the face~$f$ is equal to
\[
\sum_{j=1}^{m}(\theta_{e^*_j}+\theta_{e'^*_j})=\sum_{j=1}^{m}(2\pi-(\theta_{e_j}+\theta_{e_j'}))\,,
\]
where by definition, the angle~$\theta_{e_j}\in(0,2\pi)$  is the lift of~$[\alpha(\tr'_j)-\alpha(\tr_j)]$ while~$\theta_{e'_j}\in(0,2\pi)$ is the lift of~$[\alpha(\tr_j)-\alpha(\tr'_{j+1})]$. Observe that, for every~$j$, the train-tracks~$\tr'_j,\tr_j,\tr'_{j+1}$ are three consecutive train-tracks around a (white) vertex~$v$ in the boundary of~$f$. By assumption, the map~$\alpha$ restricted to~$\Tbip(v)$ is monotone, implying that~$2\pi-(\theta_{e_j}+\theta_{e_j'})$ is the unique lift in~$[0,2\pi)$ of~$[\alpha(\tr'_{j+1})-\alpha(\tr'_j)]$.
As a consequence, the total angle around~$f$ is equal to~$2\pi$ if and only if the restriction of~$\alpha$ to~$\Tbip^\bullet(f)$ is monotone and non-constant. This is the second condition for~$\alpha$ belonging to the set~$Y_\Gs$.

Pairing up the angles around~$f$ in the other natural way, and using the first condition around black vertices, we obtain that flatness at~$f$ is equivalent to the restriction of~$\alpha$ to~$\Tbip^\circ(f)$ being monotone. This is the last condition for~$\alpha$ belonging to the set~$Y_\Gs$, and the proof is completed.\qedhere
\end{proof}

\begin{rem}
\label{rem:Y}
It follows from the proof above that if the restriction of~$\alpha\in Z_\Gs$
to~$\Tbip(v)$ is monotone, injective and non-constant for all~$v\in\Vs$, then its restriction
to~$\Tbip^\bullet(f)$ is monotone and non-constant for all~$f\in\Fs$ if and only if its
restriction to~$\Tbip^\circ(f)$ is monotone and non-constant for all~$f\in\Fs$. Therefore, one
could drop one of these two conditions in the definition of~$Y_\Gs$.
\end{rem}

We are now ready to prove our main result.

\begin{proof}[Proof of Theorem~\ref{thm:min}]
We first prove the ``only if'' part using the discrete Gauss--Bonnet formula in the flat case,
\emph{i.e.}, Equation~\eqref{equ:lem_sum_1}.
Let us assume by means of contradiction that a bipartite planar graph $\Gs$ is not minimal but admits a minimal immersion.
Since the graph is not minimal, it either contains a self-intersection or a parallel bigon. 
In the first case, any self-intersection defines a degenerate rhombus (with angle~$\theta=0$),
 contradicting the definition of minimal immersion.
In the second case, consider the simple closed curve~$c$ on $\Gs^\T$ defined by an innermost parallel bigon. There are two corners of $c$, denoted by~$e_1$ and~$e_2$, each of which is incident to a face of~$\Gs$. By definition of the rhombus angles, we have~$\theta_{e_1}=\theta$ and~$\theta_{e_2}=2\pi-\theta$ for some~$\theta>0$. The equality~$\theta_{e_1^*}+\theta_{e_2^*}=0$ follows, contradicting Equation~\eqref{equ:lem_sum_1}.
This shows that a non-minimal graph does not admit a minimal immersion. 
By Lemma~\ref{lem:minflat}, this implies that the set~$Y_\Gs$ is empty for~$\Gs$ non-minimal.

Let us now assume that~$\Gs$ is minimal. By Lemma~\ref{lem:minflat}, an angle map~$\alpha\in Z_\Gs$ defines a minimal immersion of~$\Gs$ if and only if~$\alpha$ belongs to~$Y_\Gs$, so we are left with proving that~$Y_\Gs$ is non-empty. Since~$X_\Gs$ is clearly non-empty, it is enough to check the inclusion~$X_\Gs\subset Y_\Gs$.
Hence, let us fix an element~$\alpha\in X_\Gs$, a vertex~$v$ of~$\Gs$ of degree~$n\ge 1$, and
study the restriction of~$\alpha$ to~$\Tbip(v)$.
Since~$\Gs$ is minimal, we cannot have~$n=1$. If~$n=2$, then the monotonicity condition is trivially satisfied while the injectivity follows from the fact that the two elements of~$\Tbip(v)$ intersect.
For a vertex of degree~$n\ge 3$, Lemma~\ref{lem:vertex_1} together with the definition of~$X_\Gs$
readily imply that the restriction of~$\alpha$ to~$\Tbip(v)$ is monotone and injective.
Let us now turn to a face~$f$ of~$\Gs$ of degree~$2m$, with~$m\ge 2$ since~$\Gs$ is minimal.
Lemma~\ref{lem:face_1} and the definition of~$X_\Gs$ immediately imply that the restrictions
of~$\alpha$ to~$\Tbip^\bullet(f)$ and~$\Tbip^\circ(f)$ are monotone, and non-constant
since non-parallel train-tracks either intersect or are anti-parallel.
This completes the proof.
\end{proof}

We now give the proof of the geometric characterization of minimal immersions.

 \begin{proof}[Proof of Proposition~\ref{prop:min}]
 Let us fix an embedded bipartite planar graph~$\Gs$ without degree~$1$ vertices.
 As first requirement, a minimal immersion of~$\Gs$ is an isoradial immersion with~$k=0$. By Point 1. of Remark~\ref{rem:def_min}, this is equivalent to a rhombic immersion of~$\GR$ or to being given an angle map~$\alpha\in Z_\Gs$. The second requirement is that it needs to be flat with non-degenerate rhombi,
 which by Lemma~\ref{lem:minflat} is equivalent to asking that~$\alpha$ belongs to~$Y_\Gs$. 
We are therefore left with checking that Points 1. and 2. are equivalent to~$\alpha\in Y_\Gs$. 

 Using the notation of the proof of Lemma~\ref{lem:minflat}, we have that in the rhombic immersion
 of $\GR$, boundary vertices of a face of $\Gs^*$ corresponding to a black vertex $b$ of $\Gs$ are immersed as the endpoints of the vectors $e^{i\alpha(\tr_1)},\dots,e^{i\alpha(\tr_n)}$ drawn in a circle of center~$b$. An analogous statement holds for white vertices with minus those vectors.
 Translating the first condition of~$\alpha\in Y_\Gs$ in the language of the immersed boundary vertices, we get Point 1: the immersion is injective and preserves the cyclic order around this face.
 As a consequence of flatness, since a folded rhombus has angle in $(\pi,2\pi)$, there can be at most one
 adjacent to any given vertex.
 
 In a similar way, white, resp. black, boundary vertices of a face $f$ of $\Gs$ are immersed 
 as the endpoints of the vectors $e^{i\alpha(\tr_1)},\dots,e^{i\alpha(\tr_m)}$, resp. $-e^{i\alpha(\tr_1')},\dots,-e^{i\alpha(\tr_m')}$ drawn in a circle of center $f$. Translating the second condition of $\alpha\in Y_\Gs$ in the language of the immersed boundary vertices exactly says that the immersion preserves the cyclic order of the black, resp. white vertices around this face, which we refer to as Point~$2'$.
 We are left with proving that Points 1. and $2$. are equivalent to Points~1. and $2'$.  
Around any face~$f$, several folded rhombi can appear, but the corresponding folded edges are disjoint by Point 1. Let us move around the boundary of the immersed face~$f$ via the (say, counterclockwise) orientation given by the embedding of~$\Gs$. By construction, an immersed edge in~$\partial f$ oriented in this way has the image of the face to its right if and only if it is folded. By flatness, the immersed oriented boundary of~$f$ makes exactly one positive turn around the image of the face. One now easily checks that if the cyclic orders on the boundary vertices given by the embedding of~$\Gs$ and by the rhombic immersion differ by more than the elementary transpositions given by these folded edges, then Point~1. or Point~$2'.$ is contradicted. (In the case of a face of degree~4, flatness readily implies that there is at most one folded edge on its boundary). Conversely, Point~2. readily implies that the cyclic order of the black, resp. the white vertices is preserved in the immersion, \emph{i.e.} Point~$2'.$, thus ending the proof.
\end{proof}

As a corollary of Theorem~\ref{thm:min}, we obtain the following converse to
Lemmas~\ref{lem:vertex_1} and~\ref{lem:face_1}.

\begin{cor}\label{cor:min}
Let~$\Gs$ be a bipartite, planar graph without train-track loops and without degree~$1$ vertices.
Let us assume that the conclusions of Lemmas~\ref{lem:vertex_1} and~\ref{lem:face_1} hold, i.e.
\begin{enumerate}
\item{For any vertex~$v$ of~$\Gs$, all the train-track strands in~$\Tbip(v)$ belong to distinct train-tracks, and the global cyclic order on~$\Tbip$ restricts to the local cyclic order on~$\Tbip(v)$.}
\item{For any face~$f$ of~$\Gs$, there is a pair of train-track strands in~$\Tbip^\bullet(f)$ (resp.~$\Tbip^\circ(f)$) belonging to distinct non-parallel train-tracks, and the global cyclic order on~$\Tbip$ restricts to the local cyclic orders on~$\Tbip^\bullet(f)$ and on~$\Tbip^\circ(f)$.}
\end{enumerate}
Then, the graph~$\Gs$ is minimal.
\end{cor}
\begin{proof}
Let~$\Gs$ be a bipartite planar graph satisfying the assumptions stated above. By the end of the proof of Theorem~\ref{thm:min}, this implies that the non-empty set~$X_\Gs$ is contained in the set~$Y_\Gs$.
(Note that~$\Gs$ has neither degree~$1$ vertices, by hypothesis, nor degree~$2$ faces, by Point 2.) Thus, the set~$Y_\Gs$ is non-empty,
which by Theorem~\ref{thm:min} implies that~$\Gs$ is minimal.
\end{proof}

We conclude this section with the aforementioned result for~$\ZZ^2$-periodic minimal graphs. Its proof provides
a positive answer to the question at the end of Section~\ref{sub:ttmin} for this specific class of graphs.

\begin{cor}
\label{cor:X=Y}
If~$\Gs$ is minimal and~$\ZZ^2$-periodic, then both spaces~$X_\Gs$ and~$Y_\Gs$ coincide with the space of
monotone maps~$\alpha\colon\Tbip\to\RR/2\pi\ZZ$ mapping intersecting train-tracks to distinct angles.
\end{cor}

\begin{proof}
Let us denote by~$X'_\Gs$ the space of monotone maps~$\alpha\colon\Tbip\to\RR/2\pi\ZZ$ mapping
intersecting train-tracks to distinct angles. By definition, we have the inclusion~$X_\Gs\subset X'_\Gs$.
Furthermore, by the end of the proof
of Theorem~\ref{thm:min}, the inclusion~$X_\Gs\subset Y_\Gs$ holds for any minimal
graph. Hence, we are left with the proof of the inclusions~$Y_\Gs\subset X_\Gs'$ and~$X'_\Gs\subset X_\Gs$ for~$\Gs$
a~$\ZZ^2$-periodic minimal graph. This is an immediate consequence of the following three claims:
\begin{enumerate}
\item{Any~$\alpha\in Y_\Gs$ maps intersecting train-tracks to distinct angles.}
\item{If~$\alpha$ belongs to~$Y_\Gs$, then~$\alpha\colon\Tbip\to\RR/2\pi\ZZ$ is monotone.}
\item{Any~$\alpha\in X'_\Gs$ maps (disjoint) anti-parallel train-tracks to distinct angles.}
\end{enumerate}
The proof of the first point is straightforward. Indeed, if~$\alpha\in Y_\Gs$ were to map intersecting
train-tracks to the same angle, then its restriction to~$\Tbip(v)$ would fail to be injective for~$v$ being any
of the two vertices bounding the edge corresponding to this intersection.

We now turn to the proof of the second point. To show that~$\alpha\colon\Tbip\to\RR/2\pi\ZZ$ is monotone,
the idea is to apply the discrete Gauss-Bonnet
formula to the minimal immersion of~$\Gs$ given by~$\alpha\in Y_\Gs$. Let us fix three pairwise non-parallel
train-tracks~$\tr,\tr',\tr''\in\Tbip$, and let us first assume that they intersect pairwise.
In such a case, there exists an innermost triangle,
\emph{i.e.}, a triangle~$c\subset\Gs^\T$ whose three sides are segments of the three train-tracks~$\tr,\tr'$ and~$\tr''$,
and which does not contain any other such triangle. Since~$\Gs$ is minimal, train-tracks do not self-intersect, so
this ``innermost'' condition ensures that~$c$ is a simple cycle with three
corners of positive signs.
Hence,
Proposition~\ref{lem:lem_sum}
can be applied in the form of Equation~\eqref{equ:lem_sum_1}: the
three rhombus angles corresponding to these three corners add up to~$2\pi$. The precise form
of these angles depends on the orientation of~$\tr,\tr',\tr''$ along~$c$, but the argument is always the same;
we shall therefore assume that the three train-tracks~$\tr,\tr',\tr''$ are consistently oriented along the curve~$c$
(say, counterclockwise), and leave the other cases to the reader. We shall also assume without loss of generality
that~$c$ runs along~$\tr,\tr',\tr''$ in this (cyclic) order; by minimality of~$\Gs$ this implies that these
train-tracks are cyclically ordered as~$[\tr,\tr',\tr'']$ in~$\Tbip$. In such a case, Equation~\eqref{equ:lem_sum_1}
reads~$2\pi=\theta_{e}+\theta_{e'}+\theta_{e''}$, where~$\theta_{e}$ (resp.~$\theta_{e'}$,~$\theta_{e''}$) is the
unique lift in~$(0,2\pi)$ of the angle~$[\alpha(\tr')-\alpha(\tr)]$ (resp.~$[\alpha(\tr'')-\alpha(\tr')],[\alpha(\tr)-\alpha(\tr'')]$). This equality holds if and only if we have the cyclic order~$[\alpha(\tr),\alpha(\tr'),\alpha(\tr')]$
in~$\RR/2\pi\ZZ$. Therefore,~$\alpha\colon\Tbip\to\RR/2\pi\ZZ$
does preserve the cyclic order in this case.

Let us now assume that~$\tr,\tr',\tr''$ do not intersect pairwise, and finally use the~$\ZZ^2$-periodicity
of~$\Gs$.
In absence of unbounded faces, a train-track is infinite and periodic, so
that its quotient by $\ZZ^2$ is an oriented closed curve on the
torus $\RR^2/\ZZ^2=\TT^2$.
We shall write~$\Pi(\tr)\in H_1(\TT^2;\ZZ)\simeq\ZZ^2$ for the
homology class of the oriented closed curve~$\tr/\ZZ^2\subset\TT^2$.
Since $\Gs$ is bipartite and $\ZZ^2$-periodic,
the union of all toric, consistently oriented, closed curves
coming from train-tracks is the boundary of
the faces of~$\Gs^\T/\ZZ^2\subset\TT^2$ corresponding to vertices of~$\Gs$;
therefore, the sum of the corresponding
homology classes is equal to zero.

Recall that the only way for two closed curves in~$\TT^2$ not to meet is if their homology classes coincide
(which corresponds to the case of parallel planar curves), or
if they are opposite (this gives anti-parallel planar curves). Since~$\tr,\tr',\tr''$ are supposed to be pairwise
non-parallel and not all to intersect, it can be assumed without loss of generality that~$\tr$ and~$\tr''$ are
disjoint (with opposite homology classes~$\Pi(\tr)+\Pi(\tr'')=0$), while~$\tr'$ meets both~$\tr$ and~$\tr''$.
It can also be assumed that these train-tracks are cyclically ordered as~$[\tr,\tr',\tr'']$ in~$\Tbip$;
equivalently, the corresponding classes are cyclically ordered as~$[\Pi(\tr),\Pi(\tr'),\Pi(\tr'')]$ in~$\ZZ^2$.
Since the homology classes of closed curves of $\TT^2$
coming from train-tracks sum to zero,
there must exist a train-track homology class~$e_1\in\ZZ^2$ so that we have the
order~$[\Pi(\tr),\Pi(\tr'),\Pi(\tr''),e_1]$ in~$\ZZ^2$. Let~$\tr_1\in\Tbip$ denote a train-track realizing the
class~$e_1$. Translating it far enough in the~$\Pi(\tr)$-direction, we can assume that if it
intersects~$\tr'$, it does so far away from the (finite
number of) intersections of~$\tr'$ with~$\tr\cup\tr''$. Similarly to the previous case, we can now find a simple closed
curve~$c\subset\Gs^\T$ which, when travelled along in the counterclockwise direction, consists of oriented segments
of~$\tr,\tl_1,\tr'',\tl'$, and all of whose (four) corners have positive sign. As in the triangular case above,
Equation~\eqref{equ:lem_sum_1} applied to the curve~$c$ readily implies that the cyclic order
relation~$[\alpha(\tr),\alpha(\tr'),\alpha(\tr''),\alpha(\tr_1)]$ holds in~$\RR/2\pi\ZZ$. In particular,
we have the order~$[\alpha(\tr),\alpha(\tr'),\alpha(\tr'')]$, and the second point is proved.

We finally turn to the third point, fixing~$\alpha\in X'_\Gs$ and two disjoint anti-parallel
train-tracks~$\tr,\tr'\in\Tbip$.
Since~$\Gs$ is bipartite, $\ZZ^2$-periodic and does not have unbounded faces, its train-track homology classes add up to zero and
span a~$2$-dimensional lattice. Hence, there exist train-track homology classes~$e_1,e_2\in\ZZ^2$, each of which is
linearly independent from~$\Pi(\tr)=-\Pi(\tr')$, with cyclic order~$[\Pi(\tr),e_1,\Pi(\tr'),e_2]$.
This implies that any train-tracks~$\tr_1,\tr_2\in\Tbip$ realizing these classes will intersect~$\tr$ and~$\tr'$
and satisfy the cyclic order~$[\tr,\tr_1,\tr',\tr_2]$ in~$\Tbip$. Since~$\alpha$ is monotone and maps intersecting train-tracks to distinct angles, we cannot have the equality~$\alpha(\tr)=\alpha(\tr')$. This concludes the proof.
\end{proof}

\subsection{The Kasteleyn condition for the dimer model}
\label{sub:condition}

The goal of this section is to see the implications of the results obtained in Section~\ref{sub:min} and~\ref{sub:proofmin} for the dimer model defined on a planar, bipartite graph $\Gs=(\Bs\sqcup \Ws,\Es)$, see~\cite{KenyonIntro} for an overview of the dimer model and references therein. More specifically, we study the implications on the Kasteleyn condition.

Suppose that edges are assigned positive weights $\nu=(\nu_{e})_{e\in\Es}$. A \emph{dimer configuration} of $\Gs$ is a \emph{perfect matching}, that is, a subset of edges such that every vertex is incident to exactly one edge of this subset. When the graph $\Gs$ is finite, the weight of a dimer configuration is the product of the weights of its edges, and the \emph{partition function} is the weighted sum of all the dimer configurations of $\Gs$. 

In the following, we first give some background on the Kasteleyn condition, then describe and characterize the more general question stemming from our work, and finally give the
answer using the aforementioned results of Section~\ref{sub:min}
together with work of Thurston~\cite{Thurston}.

\bigskip
\textbf{The Kasteleyn condition.}
In the study of the dimer model a fundamental tool is the 
\emph{Kasteleyn matrix}, which is the
weighted, oriented, adjacency matrix~$\Ks$ associated to~$(\Gs,\nu)$, whose rows are indexed by white vertices, columns by black ones, and where the orientation of the edges satisfies \emph{Kasteleyn's condition}, \new{see below}.
Kasteleyn and Temperley--Fisher's celebrated theorem~\cite{Kast61,Kast63,Kast67,TF} (see also~\cite{Percus} for the bipartite setting),
states that the absolute value of the determinant of $\Ks$ is equal to the 
partition function of the dimer model on the weighted graph~$(\Gs,\nu)$. 

An orientation of the edges of~$\Gs$ is equivalent to assigning an element~$\omega(e)\in\{-1,+1\}$ to each edge~$e=(w,b)$ with the convention that~$\omega(e)=+1$ if and only if the orientation of~$e$ is, say, from~$w$ to~$b$.
\new{Then, the coefficient of~$\Ks$ corresponding to a white vertex~$w$ and a black vertex~$b$ is defined by
\[
\Ks_{w,b}=\sum_{e=(w,b)}\omega(e)\,\nu_e\,,
\]
the sum being over all the edges between~$w$ and~$b$; in particular, the coefficient is equal to~$0$ if there is no such edge.}
Also, the Kasteleyn condition of~\cite{Kast67} can be written as follows. For every bounded face~$f$ of~$\Gs$,
\begin{equation}
\label{equ:Kast}
\omega(\partial f):=\prod_{e\in\partial f}\omega(e)=-(-1)^{|f|/2}\,,
\end{equation}
where~$|f|$ denotes the degree of~$f$ (which is even since~$\Gs$ is bipartite). 

An extension of this theory 
is due to Kuperberg~\cite{Kuperberg}.
He proves that the group $\{-1,+1\}$ 
can be replaced by any subgroup~$G$ of~$\CC^*$ 
containing~$\{-1,+1\}$. The map~$\omega$ now assigns to each oriented edge~$e=(w,b)$ an element~$\omega(e)\in G$ so that the same edge with opposite orientation~$\overline{e}=(b,w)$ is mapped to~$\omega(\overline{e})=\omega(e)^{-1}$. 
Kasteleyn's condition on such an~$\omega$ remains as displayed above, the product being on the set of edges in~$\partial f$ oriented consistently (say, counterclockwise around~$f$).
The partition function theorem
then extends as follows: for any~$\omega$ satisfying Kasteleyn's condition, the determinant of the associated adjacency matrix~$\Ks$ is equal to the dimer partition function, up to multiplication by an element of~$G$. A natural choice of group is~$G=S^1$: in this case, the modulus of the determinant of~$\Ks$ is equal to the dimer partition function.

In his study of the dimer model on isoradial graphs~\cite{Kenyon:crit}, Kenyon 
introduces 
the
\emph{discrete~$\overline{\partial}$ operator}. The associated matrix is 
the adjacency matrix~$\Ks$ as above, with~$\nu_e\in(0,2)$ given by the length of the 
dual edge~$e^*$ of~$e$, 
and~$\omega(e)\in S^1$ given by the direction of the edge~$e=(w,b)$. 
Using the notation of Section~\ref{sub:gen_def}, the phase~$\omega=\omega_\alpha$  
can be defined in terms of the train-track angles~$\alpha\in Z_\Gs$ as
\begin{equation}
\label{equ:omega}
\arg(\omega_\alpha(e)):=\new{C}+\arg(e^{i\alpha(\tr_2)}-e^{i\alpha(\tr_1)})
\end{equation}
\new{for some constant~$C$.}
(Note that if~$\alpha(\tr_1)$ and~$\alpha(\tr_2)$ coincide, then the edge-weight vanishes, which amounts to removing the edge from~$\Gs$; alternatively, one can set~$\omega_\alpha(e)=-e^{i\alpha(\tr_1)}=-e^{i\alpha(\tr_2)}$).
Kenyon shows 
that the phase~$\omega_\alpha$ satisfies Kasteleyn's condition~\eqref{equ:Kast} for any~$\alpha\in Z_\Gs$ defining an isoradial embedding of~$\Gs$, thus proving that the modulus of the determinant of the 
discrete~$\overline{\partial}$ operator gives the dimer partition function on~$(\Gs,\nu)$.

More general weights have recently been introduced by Fock~\cite{Fock}, which exhibit the phase~$\omega_\alpha$ defined in Equation~\eqref{equ:omega} \new{up to a global sign}.
The corresponding dimer model in the elliptic case
is studied in~\cite{nous}, where instead of considering isoradial embeddings as in~\cite{Kenyon:crit}, we deal with the more general
setting of a bipartite planar graph~$\Gs$ together with an arbitrary angle map~$\alpha\in Z_\Gs$.
Again we need to know whether the associated
phase~$\omega_\alpha$ satisfies the Kasteleyn condition~\eqref{equ:Kast}. This, together with the more general framework of this paper, naturally leads to a more general question which is the subject of the next paragraph. 

\bigskip

\textbf{The Kasteleyn condition on rhombic immersions.}
Let~$\Gs$ be \emph{any} bipartite planar graph and let~$\alpha\in Z_\Gs$ be \emph{any} set of angles.
Recall that, as far as the present study of the dimer model is concerned, assigning the same angle
to train-tracks intersecting at an edge amounts to erasing this edge from~$\Gs$.
For this reason, we can assume without loss of generality that~$\alpha$ belongs to the
subset~$Z'_\Gs$ of~$Z_\Gs$ given by maps assigning distinct angles to intersecting train-tracks.
Note that, in particular, $Z'_\Gs$ is trivially empty when $\Gs$ contains a
self-intersecting train-track.

A natural question is to describe the subset $K_\Gs$ of $Z'_\Gs$ given by
\[
K_\Gs=\{\alpha\in Z'_\Gs\,|\,\omega_\alpha \text{ satisfies Kasteleyn's condition}\}\,.
\] 
In particular, is it possible to characterize the class of graphs $\Gs$ for which $K_\Gs$ is non-empty?

The following lemma gives a characterization of $K_\Gs$ using the rhombic immersion of $\GR$. Note that we will be using the notation of Section~\ref{sub:def} for the rhombus angles.

\begin{lem}
\label{lem:phase}
Let $\Gs$ be a bipartite, planar graph. Then, 
for any~$\alpha\in Z'_\Gs$, the phase~$\omega_\alpha$ satisfies the Kasteleyn condition around a face $f$ of $\Gs$ if and only if, in the rhombic immersion of $\GR$ defined from $\alpha$, the angle of the conical singularity at the face $f$ is an odd multiple of~$2\pi$, that is, if and only if 
\[
\sum_{e^*\sim f}\theta_{e^*}=2\pi\, [4\pi]\,.
\]
\end{lem}
\begin{proof}
Fix a planar, bipartite graph~$\Gs$ and~$\alpha\in Z'_\Gs$, and consider the corresponding 
rhombic immersion of $\GR$
near a given arbitrary bounded face~$f$ of~$\Gs$. Recall that any edge~$e\subset\partial f$ defines a rhombus (which can be folded or not), and that all these rhombi are pasted cyclically following the cyclic order of the vertices in~$\partial f$. 
Recall also that these vertices are on a circle, whose center we denote by~$\iota(f)$. Fix an arbitrary black vertex~$b\in\partial f$. Let us write~$e'=(w',b)$ and~$e''=(b,w'')$ for the adjacent edges in~$\partial f$ oriented counterclockwise around~$f$, write
$\theta',\theta''$ for the corresponding rhombus angles, and
$f'$, $f''$ for the corresponding faces adjacent to~$f$.
By definition, the phase~$\omega_\alpha$ satisfies
\[
\omega_\alpha(e')\omega_\alpha(e'')=\frac{\omega_\alpha(w',b)}{\omega_\alpha(w'',b)}=\frac{e^{i\frac{\pi}{2}}\,e^{i\arg(\iota(f')-\iota(f))}}{e^{i\frac{\pi}{2}}\,e^{i\arg(\iota(f)-\iota(f''))}}=-\frac{e^{i\arg(\iota(f')-\iota(f))}}{e^{i\arg(\iota(f'')-\iota(f))}}=-e^{\frac{i}{2}({\theta'}^*+{\theta''}^*)}\,,
\]
where~${\theta'}^*\in(-\pi,\pi]$ denotes the dual rhombus angle of the edge $e'$ at  
~$\iota(f)$, and similarly for~${\theta''}^*$ and~$e''$.
Note that this equality holds whatever the status of the edges~$e'$ and~$e''$ (folded or unfolded).
Multiplying the equation displayed above for all the~$|f|/2$ black vertices of~$\partial f$, we get
\[
\omega_\alpha(\partial f)=\prod_{e\in\partial f}\omega_\alpha(e)=(-1)^{|f|/2}\,e^{\frac{i}{2}\sum\limits_{e^*\sim f}\theta_{e^*}},
\]
and the lemma follows.
\end{proof}

\bigskip

\textbf{The Kasteleyn condition and minimal graphs.}
Using Theorem~\ref{thm:min} and~\cite{Thurston},
we now show that~$K_\Gs$ is non-empty
if and only if~$\Gs$ is a minimal graph. The ``if'' part of this statement
extends the aforementioned result of Kenyon~\cite{Kenyon:crit} which deals with the smaller class of bipartite graphs which can be isoradially embedded in the plane, while the ``only if'' part
provides an alternative proof of the fact that non-minimal graphs do not admit minimal immersions.

Recall that the weights that are of interest to us vanish on an edge if both train-tracks crossing this edge have the same angle. Therefore, we can assume without loss of generality that the train-tracks do not self-intersect.

\begin{thm}
\label{thm:phase}
Given a bipartite, planar graph~$\Gs$ whose train-tracks do not self-intersect,
the set~$K_\Gs$ is non-empty if and only if~$\Gs$ is minimal.
In such a case we have the inclusions~$X_\Gs\subset Y_\Gs\subset K_\Gs$.
\end{thm}

While the ``if'' part of this statement follows readily from Theorem~\ref{thm:min},
our proof of the other direction relies on two lemmas. The first one deals with
two local transformations for bipartite graphs introduced by Kuperberg and called
\emph{shrinking/expanding of a~$2$-valent vertex} and \emph{spider move} in~\cite{GK}.
They are illustrated in Figure~\ref{fig:elementary}.

\begin{figure}[htb]
    \centering
    \includegraphics[width=\linewidth]{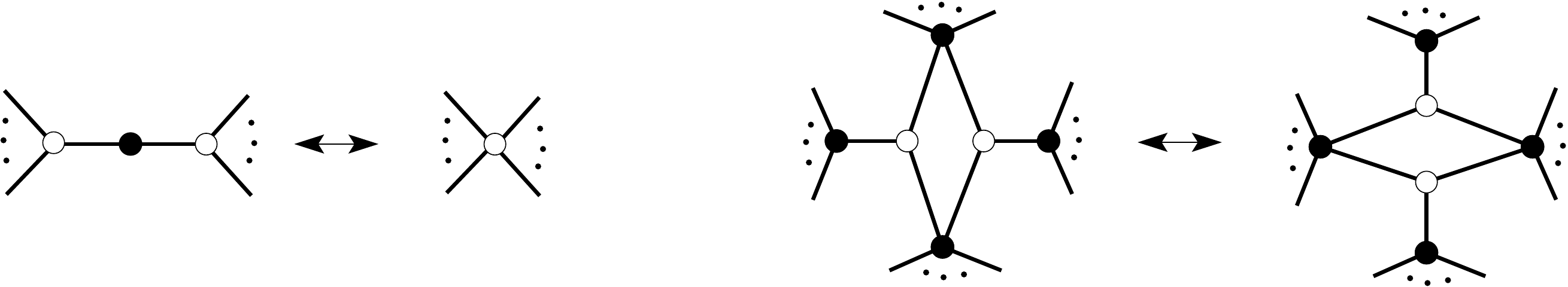}
    \caption{Shrinking/expanding of a~$2$-valent (black) vertex, and spider move (with black boundary vertices).}
    \label{fig:elementary}
  \end{figure}

\begin{lem}
\label{lemma:inv}
The space~$K_\Gs$ is invariant under shrinking/expanding of~$2$-valent vertices and spider moves.
\end{lem}
\begin{proof}
To prove this result, we will make use of Lemma~\ref{lem:phase} in the following form:
the phase~$\omega_\alpha$ satisfies the Kasteleyn condition around a face~$f$ of~$\Gs$ as in
Figure~\ref{fig:tt_around_face} if and only if
\[
\delta_\alpha(f):=\sum_{j=1}^m(2\pi-(\theta_j+\theta'_j))
\]
is an odd multiple of~$2\pi$, where~$\theta_j\in (0,2\pi)$ is the
lift of~$[\alpha(\tr_j')-\alpha(\tr_j)]$
and~$\theta'_j\in (0,2\pi)$ the lift of~$[\alpha(\tr_j)-\alpha(\tr'_{j+1})]$
(with~$\tr'_{m+1}=\tr'_1$).

\begin{figure}[htb]
    \centering
    \begin{overpic}[width=0.6\linewidth]{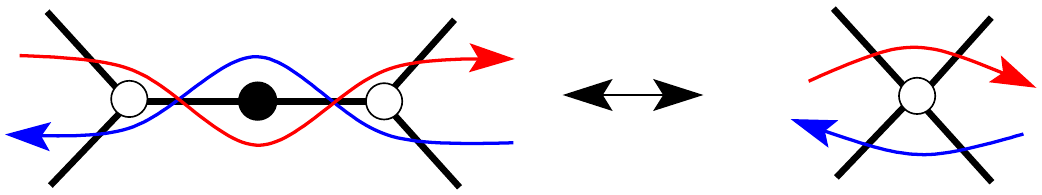}
    \put(23,17){\scriptsize $f_1$}
        \put(23,0){\scriptsize $f_2$}
        \put(-4,5){\scriptsize $\tr_2$}
        \put(50.5,12.5){\scriptsize $\tr_1$}
        \put(87,17){\scriptsize $f'_1$}
        \put(87,0){\scriptsize $f'_2$}
        \put(72,7){\scriptsize $\tr_2$}
        \put(101,9){\scriptsize $\tr_1$}
\end{overpic}
   \caption{Invariance of~$K_\Gs$ under shrinking/expanding of a~$2$-valent black vertex.}
    \label{fig:inv1}
  \end{figure}
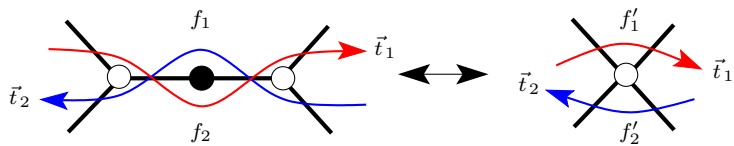
  
Let us start with the invariance of~$K_\Gs$ under shrinking/expanding of a~$2$-valent vertex 
and, without loss of generality, assume that it is black.
Using the notation of Figure~\ref{fig:inv1},
we get~$\delta_\alpha(f_1)=\delta_\alpha(f'_1)+2\pi-(\theta_{12}+\theta_{21})$,
where~$\theta_{12}\in(0,2\pi)$ is the lift of~$[\alpha(\tr_1)-\alpha(\tr_2)]$
and~$\theta_{21}\in(0,2\pi)$ the lift of~$[\alpha(\tr_2)-\alpha(\tr_1)]$.
This leads to~$\theta_{12}+\theta_{21}=2\pi$, so~$\delta_\alpha(f_1)$ and~$\delta_\alpha(f'_1)$ coincide.
The same argument shows the equality~$\delta_\alpha(f_2)=\delta_\alpha(f'_2)$,
proving the invariance of~$K_\Gs$ under this move. 

\begin{figure}[htb]
    \centering
    \begin{overpic}[width=0.8\linewidth]{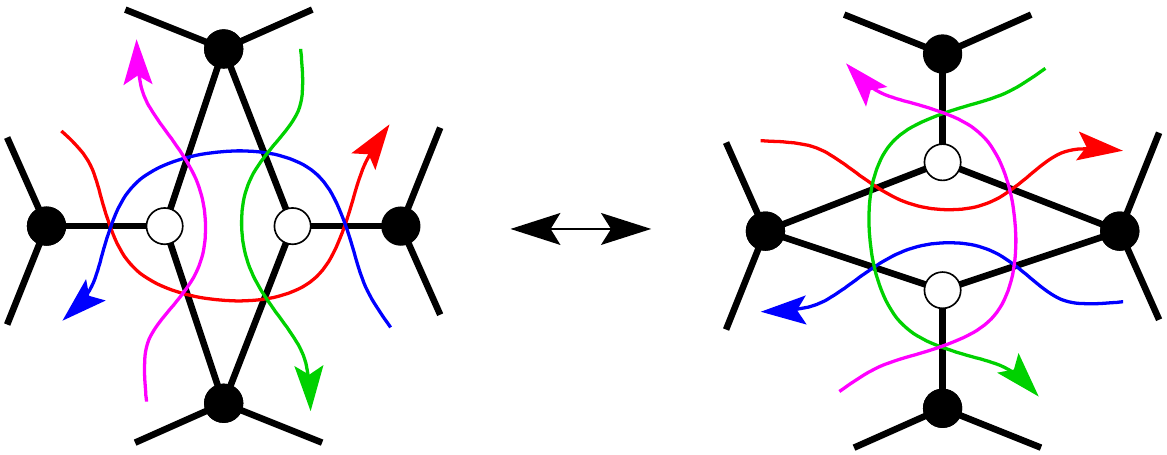}
    \put(18.5,19.5){\scriptsize $f$}
    \put(76,19){\scriptsize $f'$}
    \put(2,35){\scriptsize $f_1$}
    \put(65,35){\scriptsize $f'_1$}
    \put(5,5){\scriptsize $f_2$}
    \put(65,5){\scriptsize $f'_2$}
    \put(32,5){\scriptsize $f_3$}
    \put(95,5){\scriptsize $f'_3$}
    \put(32,35){\scriptsize $f_4$}
    \put(95,35){\scriptsize $f'_4$}
    \put(8,33){\scriptsize $\tr_1$}
    \put(3,10){\scriptsize $\tr_2$}
    \put(28,4){\scriptsize $\tr_3$}
    \put(33,29){\scriptsize $\tr_4$}
    \put(70,32){\scriptsize $\tr_1$}
    \put(64,9){\scriptsize $\tr_2$}
    \put(89.5,3){\scriptsize $\tr_3$}
    \put(95,28){\scriptsize $\tr_4$}

\end{overpic}
    \caption{Invariance of~$K_\Gs$ under spider move with black boundary vertices.}
    \label{fig:inv2}
  \end{figure}
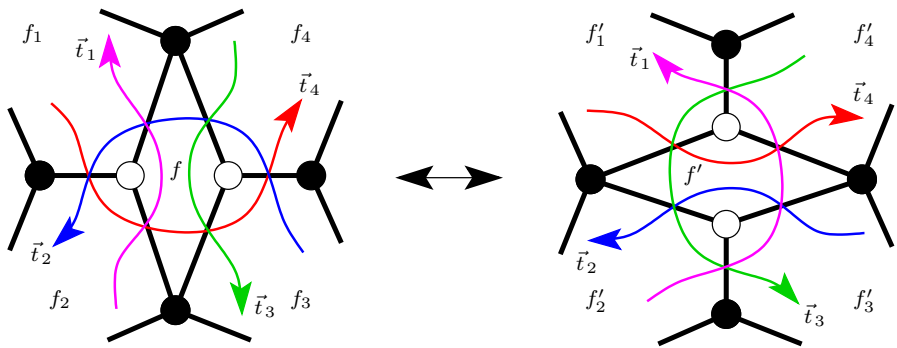
 
We now turn to the invariance of~$K_\Gs$ under the spider move,
assuming the notation of Figure~\ref{fig:inv2}.
We handle the case of black boundary vertices, the case of white ones being similar.
First note the equality
\[
\delta_\alpha(f)=4\pi-(\theta_{21}+\theta_{14}+\theta_{43}+\theta_{32})=\delta_\alpha(f')\,,
\]
with~$\theta_{ij}\in(0,2\pi)$ the lift of~$[\alpha(\tr_i)-\alpha(\tr_j)]$ for~$1\le i,j\le 4$.
Therefore, the phase~$\omega_\alpha$ satisfies the Kasteleyn condition 
around~$f$ if and only if it does around~$f'$. Moreover,
this holds for exactly two of the six
cyclic orders of these four train-track angles,
namely~$[\alpha(\tr_1),\alpha(\tr_2),\alpha(\tr_3),\alpha(\tr_4)]$
and~$[\alpha(\tr_1),\alpha(\tr_4),\alpha(\tr_3),\alpha(\tr_2)]$.
Therefore, it can be assumed that these four evaluations of~$\alpha$
are cyclically ordered in one of these two ways.
Let us now study the Kastelyn condition around the faces~$f_1$ and~$f'_1$.
Using the same notation as above, we have
\[
\delta_\alpha(f_1)-\delta_\alpha(f'_1)=\theta_{43}+\theta_{31}-(\theta_{21}+\theta_{32})\,.
\]
As one easily checks, the cyclic order constraint imposed on~$\alpha$
by the Kasteleyn condition around~$f$ implies the
equality~$\theta_{43}+\theta_{31}=\theta_{21}+\theta_{32}$. 
Therefore,~$\delta_\alpha(f_1)$ and~$\delta_\alpha(f'_1)$ coincide, so
the phase~$\omega_\alpha$ satisfies the Kasteleyn condition 
around~$f_1$ if and only if it does around~$f_1'$.
The equalities~$\delta_\alpha(f_j)=\delta_\alpha(f'_j)$ for~$j=2,3,4$
can be checked in the same way, proving the invariance of~$K_\Gs$ under this move.
\end{proof}

We now turn to the second and last lemma needed fo the proof of Theorem~\ref{thm:phase}.

\begin{lem}
\label{lemma:bigon}
Let~$\Gs$ be a bipartite, planar, non-minimal graph whose train-tracks do not self-intersect.
Then, it is related via a finite number of shrinking/expanding of~$2$-valent vertices and spider moves
to a bipartite, planar graph~$\Gs'$ which admits a face of degree~$2$.
\end{lem}
\begin{proof}
Since~$\Gs$ is non-minimal, connected, and its train-tracks do not self-intersect, it contains a finite
subgraph~$\Gs_1$ whose train-tracks form at least one parallel bigon.
Since~$\Gs_1$ is finite and has no~$1$-valent vertex (as its train-tracks do not self-intersect), it
can be transformed into a graph~$\Gs_2$ whose white vertices are~$3$-valent, via a finite number of
shrinking of (white)~$2$-valent vertices and expanding of (black)~$2$-valent vertices.
Following~\cite{GK}, consider the \emph{triple-crossing diagram}~$D_2$ obtained
from the train-tracks~$\Tbip$ of~$\Gs_2$
as follows: for each white vertex~$w$ of~$\Gs$, move the three oriented train-tracks
of~$\Tbip(w)$ so that they intersect into a single~$6$-valent vertex.
Fix an innermost parallel bigon~$B$ in~$D_2$, and consider the triple-crossing
diagram~$D\subset D_2$ given by all the train-tracks strands contained in some
small neighborhood of~$B$.
The proof of~\cite[Lemma~12]{Thurston} now implies the following:
via a finite number of so-called~\emph{$2-2$ moves}, illustrated in Figure~\ref{fig:2-2} (left),
one can transform~$D$ so that the parallel bigon~$B$ no longer contains crossings in its interior,
and so that one of the strands bounding~$B$ no longer forms anti-parallel bigons in~$B$.
This is illustrated in Figure~\ref{fig:2-2} (right).
\begin{figure}[htb]
    \centering
    \includegraphics[height=2.5cm]{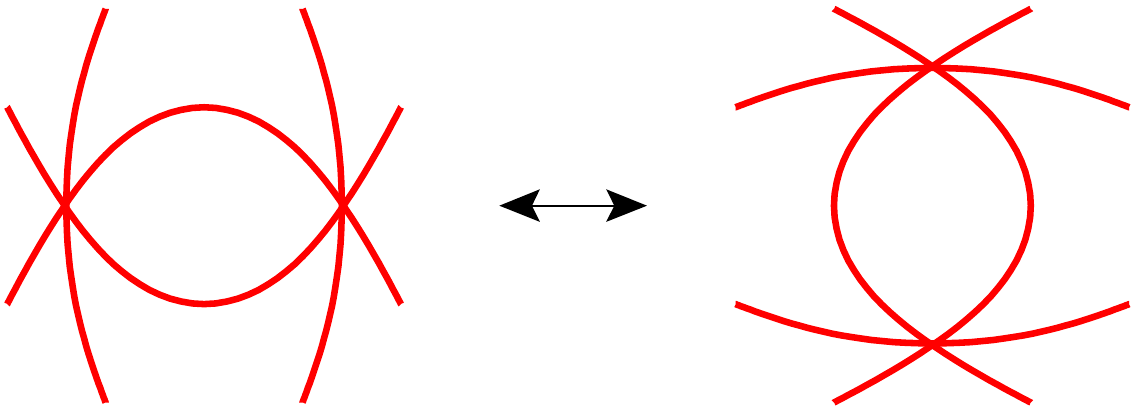}\hspace{2cm}\includegraphics[height=2.5cm]{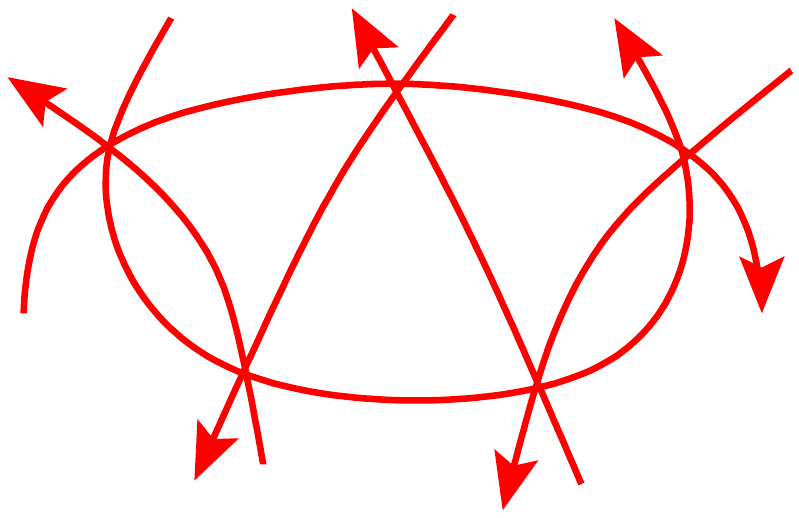}
    \caption{Thurston's~$2-2$ move (left), and a parallel bigon with neither interior crossings
    nor anti-parallel bigons bounded by the upper strand (right).}
    \label{fig:2-2}
  \end{figure}
Note that two other moves appear in~\cite[Lemma~12]{Thurston}, but they are only required in order
to get rid of self-intersections, parallel bigons and disconnected components.
Since~$D$ admits neither self-intersections nor loops, and since~$B$ is innermost among parallel bigons,
these other moves are not required in our situation.
By~\cite[Lemma~11(c)]{Thurston} (see also Figure~3(c) of this same article), a finite number of
additional~$2-2$ moves transforms~$D$ into a triple-crossing diagram~$D'$ containing a diagram
of the form illustrated in Figure~\ref{fig:D'} (left).
Now observe that the two types of~$2-2$ moves given by the two possible orientations
of the train-tracks correspond to shrinking/expanding of (black)~$2$-valent vertices
and spider moves (with black boundary vertices), see Figures~\ref{fig:inv1} and~\ref{fig:inv2}. 
Finally, the graph~$\Gs'$ corresponding to~$D'$ contains the graph illustrated in
Figure~\ref{fig:D'} (right), which admits a face of degree~$2$. This concludes the proof.
\end{proof}

\begin{figure}[htb]
    \centering
    \includegraphics[height=4cm]{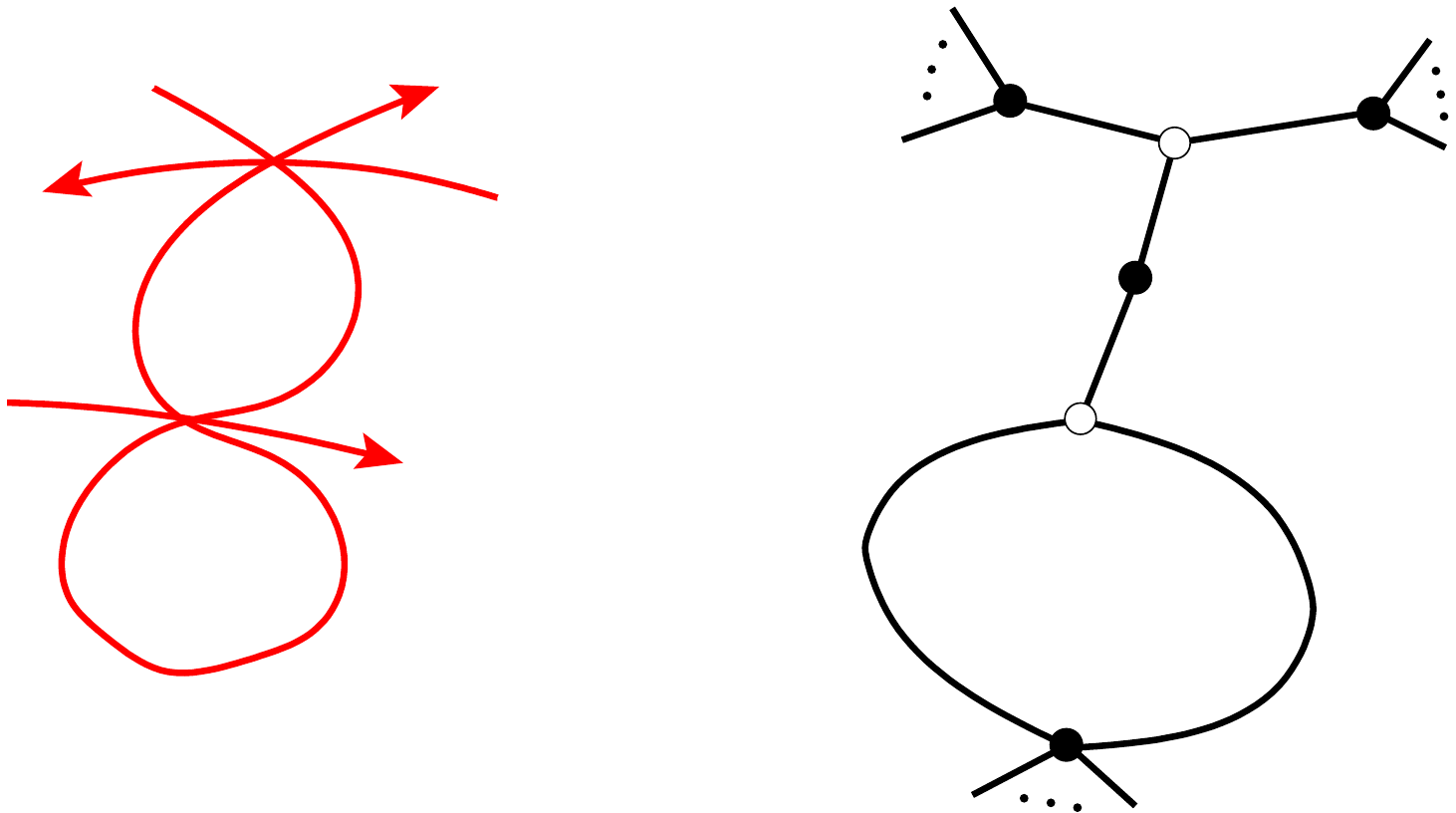}
    \caption{A portion of the triple-crossing diagram~$D'$ (left), and the corresponding subgraph of~$\Gs'$ (right).}
    \label{fig:D'}
  \end{figure}
 
As noticed by the reader, our proof of Lemma~\ref{lemma:bigon}
relies on Thurston's theory of triple-crossing diagrams~\cite{Thurston}.
It should be noted that this result also follows from the work of Postnikov,
more precisely ~\cite[Lemma~13.6]{Postnikov} applied to connected, bipartite planar
graphs whose train-tracks do not self-intersect.
We thank Marianna Russkikh for pointing out this additional reference.

We are finally ready to prove the main result of this section.

\begin{proof}[Proof of Theorem~\ref{thm:phase}]
Let us first assume that~$\Gs$ is minimal.
By Theorem~\ref{thm:min}, any~$\alpha\in Y_\Gs$ defines a flat isoradial immersion of~$\Gs$. In particular, the angle of the conical singularity at any face is equal to~$2\pi$. By Lemma~\ref{lem:phase}, this shows that~$\omega_\alpha$ satisfies the Kasteleyn condition.
Since the inclusion~$Y_\Gs\subset Z'_\Gs$ is immediate,
we have the inclusion~$Y_\Gs\subset K_\Gs$. By the proof of Theorem~\ref{thm:min}, we also have the inclusion~$X_\Gs\subset Y_\Gs$, with~$X_\Gs$ non-empty, so~$K_\Gs$ is non-empty as claimed.
Let us now assume that~$\Gs$ is a non-minimal graph whose train-tracks do not self-intersect.
By Lemmas~\ref{lemma:inv} and~\ref{lemma:bigon}, the space~$K_\Gs$ coincides with~$K_{\Gs'}$ for some planar
bipartite graph~$\Gs'$ admitting a face of degree~$2$.
Let us denote by~$\tr_1,\tr_2$ the two corresponding train-tracks.
For any~$\alpha\in Z'_{\Gs}=Z'_{\Gs'}$, the cone angle inside this face
is equal to~$2\pi-(\theta+\theta')$, where~$\theta\in(0,2\pi)$ is the lift of~$[\alpha(\tr_1)-\alpha(\tr_2)]$
and~$\theta'\in(0,2\pi)$ the lift of~$[\alpha(\tr_2)-\alpha(\tr_1)]$. This leads to~$\theta+\theta'=2\pi$,
so this cone angle vanishes.
By Lemma~\ref{lem:phase}, this arbitrary~$\alpha$ does not belong to~$K_{\Gs'}$, so~$K_{\Gs'}=K_\Gs$ is empty.
This concludes the proof.
\end{proof}

A natural question is whether the inclusions~$X_\Gs\subset Y_\Gs\subset K_\Gs$
for~$\Gs$ minimal
are actually equalities.
The equality between~$X_\Gs$ and~$Y_\Gs$ holds in the~$\ZZ^2$-periodic case (recall Corollary~\ref{cor:X=Y}), but whether it remains true in the general case is still open. We
conclude this article with an example showing
that the second inclusion
is not an equality in general.

\begin{exm}
Consider the bipartite minimal graph~$\Gs$ given by the square lattice,
and denote the four train-tracks around an arbitrary face by~$t_1,t_1',t_2,t_2'$,
as illustrated in Figure~\ref{fig:tt_around_face}.
We focus on~$\ZZ^2$-periodic elements of~$Z_\Gs$, \emph{i.e.},
maps assigning the same value to parallel train-tracks.
Since each train-track of~$\Gs$ is parallel to exactly one of the four train-tracks listed above, we are dealing with maps~$\alpha\colon\{t_1,t_1',t_2,t_2'\}\to\RR/2\pi\ZZ$.
It is not difficult to check that, on this easy example, the set
of~$\ZZ^2$-periodic elements of~$Y_\Gs$ consists of all
injective maps~$\alpha$ such that the angles are cyclically ordered as~$[\alpha(\tr_1),\alpha(\tr'_1),\alpha(\tr_2),\alpha(\tr'_2)]$. 
However, any injective~$\beta\colon\{t_1,t_1',t_2,t_2'\}\to\RR/2\pi\ZZ$ such that the corresponding angles are cyclically ordered as~$[\beta(\tr'_1),\beta(\tr_1),\beta(\tr'_2),\beta(\tr_2)]$ will define an isoradial immersion with cone angle equal to~$-2\pi$ inside each face of~$\Gs$. Therefore, by Lemma~\ref{lem:phase}, such a~$\beta$ belongs to~$K_\Gs\setminus Y_\Gs$.
\end{exm}

\bibliographystyle{alpha}
\bibliography{ellipt_iso}

\end{document}